\documentclass[letterpaper,11pt]{article}
\usepackage[margin=1in]{geometry}
\usepackage{graphicx}
\usepackage{amsmath}
\usepackage{amssymb}
\usepackage{amsthm}
\usepackage{mathrsfs}
\usepackage{bm}
\usepackage{enumerate}
\usepackage{color}
\usepackage{booktabs}

\usepackage{algorithm}
\usepackage{algpseudocode}
\usepackage{xspace}

\usepackage{subcaption}

\usepackage[%dvips,
            CJKbookmarks=true,
            bookmarksnumbered=true,
            bookmarksopen=true,
            colorlinks=true,
            citecolor=red,
            linkcolor=blue,
            anchorcolor=red,
            urlcolor=blue,
            ]{hyperref}

\allowdisplaybreaks

\newcommand{\norm}[1]{{\left\Vert #1 \right\Vert}}
\newcommand{\Norm}[1]{{\Vert #1 \Vert}}
\newcommand{\Fnorm}[1]{\Norm{#1}_{\rm F}}

\newcommand{\ones}{\bm{1}}
\def\R{\mathbb{R}}

\def\P{\mathbb{P}}
\def\E{\mathbb{E}}

\def\naturals{\mathbb{N}}

\def\cE{\mathcal{E}}
\def\cI{\mathcal{I}}

\def\N{\mathcal{N}}
\def\cC{\mathcal{C}}
\def\cF{\mathcal{F}}
\def\cM{\mathcal{M}}

\def\tT{\tilde{T}}

\def\a{\mathbf{a}}
\def\b{\mathbf{b}}

\def\d{\mathbf{d}}
\def\e{\mathbf{e}}

\def\g{\mathbf{g}}
\def\m{\mathbf{m}}
\def\r{\mathbf{r}}
\def\u{\mathbf{u}}
\def\v{\mathbf{v}}
\def\w{\mathbf{w}}
\def\x{\mathbf{x}}
\def\y{\mathbf{y}}
\def\z{\mathbf{z}}

\def\B{\mathbf{B}}
\def\D{\mathbf{D}}
\def\H{\mathbf{H}}
\def\S{\mathbf{S}}
\def\I{\mathbf{I}}
\def\M{\mathbf{M}}
\def\U{\mathbf{U}}

\def\Q{\mathbf{Q}}
\def\T{\mathbf{T}}

\def\X{\mathbf{X}}

\def\bN{\mathbb{N}}
\def\bS{\mathbb{S}}

\def\bmu{\boldsymbol{\mu}}
\def\btheta{\boldsymbol{\theta}}
\def\bbeta{\boldsymbol{\beta}}
\def\bgamma{\boldsymbol{\gamma}}

\def\bvarphi{\boldsymbol{\varphi}}
\def\bSigma{\boldsymbol{\Sigma}}
\def\bLambda{\boldsymbol{\Lambda}}
\def\eps{\varepsilon}
\def\beps{\boldsymbol{\varepsilon}}
\def\d{\mathrm{d}}

\newcommand{\calG}{\mathcal{G}}

\newcommand{\iprod}[2]{\left\langle#1,#2\right\rangle}

\newcommand{\Indc}{\mathbf{1}}

\newcommand{\floor}[1]{\left\lfloor #1 \right\rfloor}

\newcommand{\pnorm}[2]{\lVert #1\rVert_{#2}}

\newcommand{\Bigpnorm}[2]{\Big\lVert#1\Big\rVert_{#2}}
\newcommand{\biggpnorm}[2]{\bigg\lVert#1\bigg\rVert_{#2}}
\newcommand{\abs}[1]{\left| {#1} \right|}
\newcommand{\opnorm}[1]{\| #1 \|_{\rm op}}
\newcommand{\pth}[1]{\left( #1 \right)}
\newcommand{\qth}[1]{\left[ #1 \right]}

\newcommand{\iid}{iid}

\newcommand{\iiddistr}{\stackrel{\mathrm{\iid}}{\sim}}

\DeclareMathOperator{\Var}{Var}
\DeclareMathOperator{\var}{Var}
\DeclareMathOperator{\Cov}{Cov}

\DeclareMathOperator{\diag}{Diag}
\DeclareMathOperator{\Tr}{Tr}
\DeclareMathOperator{\tr}{Tr}

\DeclareMathOperator{\dTV}{\mathrm{TV}}
\DeclareMathOperator{\TV}{\mathrm{TV}}

\DeclareMathOperator{\op}{op}
\DeclareMathOperator{\tiid}{iid}
\DeclareMathOperator{\block}{block}

\DeclareMathOperator{\PI}{PI}

\def\op{\mathrm{op}}

\def\Prob{\mathbb{P}}

\newcommand{\argmin}{\mathrm{argmin}}

\newtheorem{theorem}{Theorem}
\newtheorem{lemma}[theorem]{Lemma}
\newtheorem{corollary}[theorem]{Corollary}

\newtheorem{proposition}[theorem]{Proposition}

\theoremstyle{definition}

\newtheorem{remark}{Remark}
\newtheorem{assumption}{Assumption}

\newcommand{\reals}{\mathbb{R}}
\newcommand{\Expect}{\mathbb{E}}
\usepackage{prettyref}
\newrefformat{eq}{(\ref{#1})}
\newrefformat{chap}{Chapter~\ref{#1}}
\newrefformat{sec}{Section~\ref{#1}}
\newrefformat{alg}{Algorithm~\ref{#1}}
\newrefformat{fig}{Figure~\ref{#1}}
\newrefformat{tab}{Table~\ref{#1}}
\newrefformat{rmk}{Remark~\ref{#1}}
\newrefformat{clm}{Claim~\ref{#1}}
\newrefformat{def}{Definition~\ref{#1}}
\newrefformat{cor}{Corollary~\ref{#1}}
\newrefformat{lmm}{Lemma~\ref{#1}}
\newrefformat{prop}{Proposition~\ref{#1}}
\newrefformat{app}{Appendix~\ref{#1}}
\newrefformat{hyp}{Hypothesis~\ref{#1}}
\newrefformat{thm}{Theorem~\ref{#1}}
\newrefformat{assump}{Assumption~\ref{#1}}
\newrefformat{conj}{Conjecture~\ref{#1}}
\newrefformat{ex}{Example~\ref{#1}}

\renewcommand{\hat}{\widehat}
\renewcommand{\tilde}{\widetilde}

\newcommand{\stepa}[1]{\overset{\rm (a)}{#1}}
\newcommand{\stepb}[1]{\overset{\rm (b)}{#1}}

\begin{document}

% paper title
% \title{Method of moments for 
% empirical Bayes in high-dimensional linear model}
\title{Empirical Bayes linear regression in high dimensions: 
Method of moments and sub-linear sample complexity
}

\author{Zhou Fan, Yandi Shen, Haoyu Wang, Yihong Wu\thanks{
	Z.~Fan and Y.~Wu are with the Department of Statistics and Data Science, Yale University, New Haven, Connecticut, USA, \texttt{\{zhou.fan,yihong.wu\}@yale.edu}.
	Y.~Shen is with the Department of Statistics and Data Science, Carnegie Mellon University, Pittsburgh, PA, USA, \texttt{yandis@andrew.cmu.edu}.
	H.~Wang is with the Department of Mathematics, Yale University, New Haven, Connecticut, USA, \texttt{wanghaoyu417@gmail.com}.
	}}

\date{\today}

\maketitle

\begin{abstract}

We study empirical Bayes estimation of the prior in high-dimensional linear regression
$\y=\X\bbeta+\beps$, where the regression coefficients are drawn independently from an unknown sub-Gaussian prior. In contrast to the sequence model, the design matrix couples the latent coefficients, so that recovering the prior requires deconvolving it from both the noise and copies of itself. We introduce the \emph{Empirical Bayes Method of Moments} (EBMoM), a computationally efficient procedure for general designs that recursively estimates the prior moments through a lower-triangular system of estimating equations and runs in time $O(np^2)$.
% given the Gram matrix $\X^\top\X$, runs in linear time.

Under mild design conditions, satisfied in particular by a broad class of correlated random designs, we show that EBMoM consistently estimates a growing number of moments and hence the prior itself, 
provided that $n\geq p^{1-o(1)}$. A matching information-theoretic lower bound, valid for a broad class of designs, shows that this sublinear sample complexity is optimal for nonparametric prior estimation.
% Thus, consistent prior estimation is possible well below the linear regime $n=\Omega(p)$ in which consistency has so far been established for likelihood-based empirical Bayes methods.
This improves on existing results for likelihood-based methods whose consistency requires a linear sample size $n=\Omega(p)$.

% We study empirical Bayes estimation of the prior in high-dimensional Bayesian linear regression $\y = \X\bbeta + \beps$, where the coordinates of $\bbeta$ are drawn i.i.d.~from an unknown prior $g$. Unlike the sequence model, where this is a classical deconvolution problem, observations from the linear model are comprised of convolutions of $g$ both with itself and with the noise. We propose the Empirical Bayes Method of Moments (EBMoM), which resolves this coupling by recursively estimating the moments of $g$  through a triangular system of estimating equations. The procedure runs in time $O_k((n\vee p)^2)$ once the Gram matrix $\X^\top \X$ is computed, where $k$ is the number of moments to be estimated. We establish non-asymptotic error bounds for the moment estimates under mild design conditions, satisfied by a broad class of correlated random designs, and deduce that a sample complexity $n \geq p^{1-c(k)}$ with some $c(k)>0$ is sufficient for estimating each $k^\text{th}$ order moment. By taking $k\rightarrow\infty$, this further implies that a \emph{sub-linear} sample size $n = p^{1-o(1)}$ suffices for consistent estimation of $g$ in Wasserstein-$1$ distance, where $o(1)$ can be arbitrarily slow as $\min\{n,p\} \rightarrow\infty$. This threshold is better than the sample size condition $n\gtrsim p$ required by existing likelihood-based methods, and is shown to be information-theoretically optimal for a broad class of designs by a matching lower bound.
\end{abstract}

\setcounter{tocdepth}{3}
\tableofcontents

\section{Introduction}

Empirical Bayes, introduced by Robbins \cite{robbins1951asymptotically, robbins1956empirical}, is a powerful paradigm for large-scale inference. It was originally developed for the \emph{sequence model}, in which one observes 
\begin{align*}
y_i = \theta_i + \eps_i, \quad i = 1,\ldots,n,
\end{align*}
where the latent means $\theta_i$ are drawn i.i.d.\ from an unknown prior $g$, and $\eps_i \sim \N(0,\sigma^2)$ is independent Gaussian noise; the goal is to learn the prior $g$ and denoise the observations $y_1,\dots,y_n$ to perform inference on $\theta_1,\ldots,\theta_n$. The core idea of empirical Bayes is to estimate this prior from the data, rather than positing it a priori as in classical Bayesian inference, and then to carry out the downstream Bayesian inference using the estimated prior. Since the $y_i$ are i.i.d.\ samples from the convolution $g * \N(0,\sigma^2)$, learning $g$ in the sequence model is a classical deconvolution problem, and empirical Bayes has been extensively studied in this setting. Such methods have been widely explored in statistics \cite{casella1985introduction, zhang2003compound, efron2012large} and have also received attention in many applied areas \cite{ver1996parametric, efron2001empirical, brown2008season}.

While the theoretical study of empirical Bayes is well developed for the sequence model, many modern applications involve more intricate dependence structures between the observations and the latent variables. A prominent instance of such a model is high-dimensional linear regression
\begin{align}\label{eq:model}
\y = \X\bbeta + \beps, 
\end{align}
where $\X = [\x_1,\ldots,\x_p] \in \R^{n\times p}$ is a fixed or random design matrix, $\bbeta  = (\beta_1,\ldots, \beta_p) \in \R^p$ has coordinates $\beta_1,\ldots, \beta_p$ drawn i.i.d.\ from an unknown prior $g$, and $\beps \sim \N(0,\sigma^2 \I_n)$ is independent of $\bbeta$ and $\X$. As in the sequence model, the goal is to estimate the prior $g$ from the data $(\X, \y)$, now in the high-dimensional regime where the sample size $n$ and the dimension $p$ are both large.

When $\X$ has full column rank, the linear model can be viewed as a sequence model with correlated errors, and in fact contains the sequence model as the special case $\X = \I_n$. Extending empirical Bayes theory from sequence models to linear models with general designs $\X$, however, is far from straightforward. In the sequence model the observations are i.i.d.\ from $g * \N(0,\sigma^2)$, so the marginal likelihood factorizes across coordinates, and estimating $g$ reduces to a well-understood deconvolution problem for which both maximum likelihood and the method of moments are statistically optimal and computationally efficient. In the linear model, the design $\X$ couples the coordinates of $\bbeta$: the marginal law of each observation is no longer a simple convolution of $g$ with the noise, but also involves $g$ convolved with copies of itself. Recovering $g$ therefore requires deconvolving it \emph{simultaneously} from the noise and from itself---a \emph{self-deconvolution} phenomenon with no analogue in the sequence model (see Section \ref{sec:lb} for a detailed discussion). As a consequence, the marginal likelihood neither factorizes across coordinates nor is convex in $g$ in general, and the moments of the data depend nonlinearly on those of $g$. These features make both the statistical analysis and the computation substantially more demanding, and rigorous high-dimensional empirical Bayes theory for the linear model has remained limited; earlier attempts either restrict the prior to an inflexible parametric class or incur considerable computational cost \cite{nebebe1986bayes,george2000calibration,yuan2005efficient}.

To date, most empirical Bayes methods for the linear model have been likelihood-based. A common starting point is the Gibbs variational representation of the marginal data likelihood, which reformulates maximum likelihood as a bivariate optimization over the prior and the posterior \cite{mukherjee2023mean,akyildiz2023interacting, kuntz2023particle,fan2023gradient,fan2025dynamicalI,fan2025dynamicalII,lee2026parametric}. The main challenge is then how to deal with the computational intractability of the posterior, and existing methods take a variety of approaches, including Monte Carlo sampling, coupled gradient flows, and tractable variational approximations; we review these methods in detail in Section \ref{sec:related_work}. Despite their different computational strategies, the available theory has so far remained limited in two important respects. Consistency has only been established for linear sample complexity $n=\Omega(p)$, and the guarantees require specific  assumptions on the design, such as an i.i.d.\ random design. This leaves open the following question:
% These methods differ mainly in how they handle the computational hardness of the posterior---for instance via Monte Carlo sampling, coupled gradient flows, or tractable variational approximations---and we review them in detail in Section \ref{sec:related_work}. Nevertheless, they share two recurring limitations: current consistency results for these methods have only been established when $n \geq cp$ for a constant $c>0$, and their guarantees often rely on specific design assumptions, such as i.i.d.\ random design. This leaves open the following question:
\begin{quotation}
\em Under a general design, what is the optimal sample complexity for consistently estimating the prior, and can it be achieved by a computationally efficient procedure?    
\end{quotation}
The goal of this paper is to answer both questions. We summarize our main contributions as follows.
\begin{enumerate}
\item We propose the \emph{Empirical Bayes Method of Moments} (EBMoM), a computationally efficient estimator of the prior $g$ that applies to general designs, including a broad class of correlated random designs. To estimate $k$ moments, the procedure runs in time $O_k(p(n+ p))$ in addition to computing the Gram matrix $\H = \X^\top \X$ which takes at most $O(np^2)$ time.

\item We establish non-asymptotic error bounds for the moment estimates under mild design conditions, and deduce that a sample complexity $n \geq p^{1-c(k)}$ with some $c(k)>0$ is sufficient for estimating each $k^\text{th}$ order moment. Consequently, by taking $k\rightarrow\infty$, a \emph{sub-linear} sample size $n \geq p^{1-o(1)}$ suffices for consistent estimation of the full prior $g$ in the Wasserstein-$1$ distance, where the $o(1)$ term can vanish arbitrarily slowly as $\min\{n,p\} \rightarrow \infty$.
\item We prove a matching information-theoretic lower bound, showing that this sub-linear $n \geq p^{1-o(1)}$ threshold is optimal. In particular, it improves upon the linear requirement $n \geq cp$ that has been shown for existing likelihood-based methods.
\end{enumerate}

\subsection{Related work}\label{sec:related_work}

\paragraph{MoM for sequence models}
The method of moments (MoM) is one of the classical approaches to parameter estimation and dates back to Pearson \cite{pearson1894contributions}. Despite its conceptual simplicity, classical MoM can face several practical difficulties, including the solvability of the polynomial systems induced by the moment equations, computational inefficiency, and the strong assumptions sometimes needed for consistency. A broad extension is Hansen's Generalized Method of Moments (GMM) \cite{hansen1982large}, which replaces exact moment matching by optimization over discrepancies between empirical and model moments and has been widely used in economics and finance \cite{hall2005generalized}.

We highlight two MoM estimators that are particularly relevant to our experiments. First, the \emph{denoised method of moments} (DMM) \cite{wu2020optimal,WY-fnt} gives an efficient implementation of GMM for location mixture models by optimizing directly over the truncated moment space: the estimated moments are projected onto this space by a semidefinite programming, after which a discrete distribution with matching moments is found using Gauss quadrature. Second, Lindsay's algorithm \cite{lindsay1989moment} is a moment-based estimator for homoscedastic Gaussian mixtures that also estimates the unknown common variance.
% ; see also \cite[Section 4.2]{wu2020optimal} for an analysis of this estimator. 
In Section \ref{sec:exp}, we apply both procedures as subroutines to convert moment estimates into estimates for  the prior.

\paragraph{MoM for Bayesian linear models}
MoM has been widely used in parametric Bayesian linear models where $g$ corresponds to a centered Gaussian prior with unknown variance. MoM is a classical approach to this variance-estimation task, first developed in \cite{henderson1953estimation}. A more comprehensive study of MoM for linear models was pioneered by Rao in his work on minimum-variance quadratic unbiased estimation (MIVQUE) \cite{rao1971estimation,rao1972estimation}. Specifically, for the linear model $ \y = \X \bbeta + \beps $ with Gaussian prior $ \beta_1,\dots,\beta_p \iiddistr \N(0, s^2) $ and noise $ \beps \sim \N(0,\sigma^2 \I) $, the MoM estimator takes the form
\begin{equation}\label{eq:general_mom_bayesian}
\hat{s}^2 = \y^\top \B \y,\ \ \mbox{where } \B \mbox{ satisfies } \E[\y^\top \B \y] = s^2.
\end{equation}
A popular modern procedure for MoM estimation is LD-score regression \cite{bulik2015ld}. Assuming the variants are standardized so that $ \|\x_j\| = 1 $, LD-score regression is based on the identity
$$ \E[(\x_j^\top \y)^2] = s^2 \sum_{k=1}^p \Big[ (\x_j^\top \x_k)^2 - \frac{1}{n} \Big] + \frac{p}{n} s^2  + \sigma^2. $$
This method was shown to fall within the general MoM framework \eqref{eq:general_mom_bayesian} in \cite{zhou2017unified}.

Motivated by applications, prior models beyond the Gaussian distribution have been proposed, along with more complex estimation procedures. For example, \cite{o2019extreme} introduced a fourth-moment/excess-kurtosis-type statistic to measure quantitative polygenicity in genetics,
$$ \kappa_e = \frac{3\E[\alpha^2 \beta^2] - 2\E[\beta^4]}{(\E[\alpha \beta])^2}, $$
where $ \alpha_j = \x_j^\top \X \bbeta $ and $ (\alpha, \beta) $ denotes a uniformly random pair $ (\alpha_j, \beta_j) $, $ j = 1,\dots,p $, with the expectation taken over both this uniform choice and the randomness of $ \bbeta $. In the Bayesian setting $ \beta_j \sim \N(0,s_j^2) $, an MoM estimator for $ \kappa_e $ was proposed in \cite{o2019extreme,o2021distribution}, based on the identity
$$ \kappa_e = \frac{3p}{\sum s_k^2} \sum_{j=1}^p \frac{s_j^2}{\sum s_k^2} \E[\alpha_j^2]. $$
A further extension is the linear mixed model $ \y = \X_1 \bbeta_1 + \cdots + \X_K \bbeta_K + \beps $, now one of the predominant models in applied fields such as genetics \cite{yang2011gcta}, where MoM has been used for a range of inference problems \cite{loh2015contrasting}.

\paragraph{Other EB methods for Bayesian linear models}
Beyond Gaussian priors, in contrast to the moment-based approach taken in this work, most EB methods in the literature are likelihood-based, with the common starting point being the Gibbs variational representation of the marginal data likelihood as a more tractable bivariate optimization problem. From here, existing methods can be roughly categorized as follows.
\begin{itemize}
\item \emph{(EM and its variants.)} This line of work, originating from \cite{neal1998view}, alternates between updating the prior and computing the corresponding posterior; because the posterior step is computationally difficult, it is often replaced by an approximation or a sampling-based analogue such as Monte Carlo EM (MCEM) or hierarchical Bayes sampling. These methods are widely adopted in applied contexts such as statistical genetics \cite{zhou2013polygenic,lloyd2019improved,ge2019polygenic,zhou2021fast}. Classical theoretical treatments of MCEM include \cite{chan1995monte, fort2003convergence, neath2013convergence}, and related methods such as stochastic EM (SEM) are studied in \cite{diebolt1993asymptotic,delyon1999convergence}.
\item \emph{(Gradient flow.)} 
%\nbwu{Is ``Bivariate optimization'' the right headline? Is it better to say ``gradient flow'' then in the paragraph be more specific and say gradient flow implemented by diffusion? }
In contrast to MCEM and related methods, which require sampling from the exact posterior given the current prior, recent work \cite{fan2023gradient,fan2025dynamicalI,fan2025dynamicalII} derives sufficient conditions for the idealized nonparametric maximum likelihood estimator (NPMLE) and analyzes a bivariate gradient-flow algorithm that simultaneously updates the prior and samples from the posterior via Langevin diffusion. See also \cite{akyildiz2023interacting, kuntz2023particle,lim2024momentum,caprio2025error} for related methods in more general settings beyond the linear model.
\item \emph{(Variational inference.)} To circumvent the intractability of posterior computation, variational inference instead approximates the posterior by a more tractable family of distributions. These methods are widely applied in statistical genetics \cite{spence2022flexible,morgante2023flexible}, though rigorous theory is scarcer. The recent work \cite{mukherjee2023mean} establishes conditions under which consistency continues to hold for the optimizer of a naive mean-field variational approximation. This is further sharpened in \cite{lee2026parametric}, which provides necessary and sufficient conditions for an $O(1/\sqrt{p})$ convergence rate and a central limit theorem for the variational EB estimator.
\end{itemize}

\subsection{Notations}
We use the standard big-O notations: for two positive sequences $ \{a_n\} $ and $ \{b_n\} $, we denote $ a_n = O(b_n) $ if $ a_n \leq C b_n $ for some constant $ C > 0 $ and we write $ a_n = O_k(b_n) $ if $ C $ depends on some other parameter $ k $.

For a tensor $\T \in (\reals^p)^{\otimes k}$, its (diagonal) trace is defined as
$\Tr(\T) := \sum_{a=1}^p T_{a,\ldots,a}$. The inner product between tensors $\S $ and $\T$ is defined as $ \iprod{\S}{\T} = \sum_{a_1,\dots,a_k = 1}^p S_{a_1,\dots,a_k} T_{a_1,\dots,a_k} $. For $\x \in\reals^p$, $\x^{\otimes k}$ is a rank-one tensor given by 
$(\x^{\otimes k})_{a_1,\ldots,a_k} = x_{a_1} \cdots  x_{a_k}$. $\bm{1}_p$ is the all-1's vector in $\R^p$, and we omit the subscript when the dimension is clear.

For a probability distribution $ g $, we use $ m_k(g) := \E[\beta ^k]
$ to denote its $ k $th moment and $ \mu_k(g) := \E[(\beta - \E[\beta])^k]$ for its $ k $th central moment, where $ \beta \sim g $. For $ \ell \geq 1 $, the moment vector up to order $ \ell $ is defined as $ \bm{m}_\ell(g) = (m_1(g),\dots,m_\ell(g)) $. The moment tensor of a random vector $ \bbeta $ is defined as the symmetric order-$\ell$ tensor
$M^{(\ell)}(\bbeta) = \Expect[\bbeta^{\otimes \ell}]$, with 
$M^{(\ell)}_{s_1,\ldots,s_\ell} = 
\Expect[\beta_{s_1} \ldots \beta_{s_\ell}]$ for $s_1,\ldots,s_\ell \in[p]$.

\subsection{Organization}
The paper is organized as follows. In Section \ref{sec:method}, we describe the proposed \emph{Empirical Bayes Method of Moments} (EBMoM) along with its statistical guarantees, as well as an efficient linear-time algorithm to compute the moment estimates. In Section \ref{sec:lb}, we present a matching sample complexity lower bound for nonparametric estimation of the prior. Numerical experiments are reported in Section \ref{sec:exp}. Section \ref{sec:discussion} concludes with discussions and open questions. Proofs of risk bounds for EBMoM are given in Section \ref{sec:proof_ub}, and the  minimiax lower bound is proved in Section \ref{sec:lb-pf}.

\section{Method of moments for EB linear regression}
\label{sec:method}

\subsection{Construction of moment estimators}
\label{sec:construction}
To estimate the first $L$ moments of the prior, we aim to set up $L$ estimating equations, which, upon taking expectations, result in a system of polynomial equations that are \textit{lower triangular}. That is, for each $i$,
the expectation of the $i$th summary statistic is a polynomial of $m_1,\ldots,m_{i-1}$ plus a \textit{linear} term of $m_i$. This  
suggests we may use previous estimates of lower-order moments to successively peel off the nonlinear terms and obtain an estimate of the next moment.

For a fixed degree $L \in \bN$, we construct estimators for the mean  $m_1(g)$ and central moments $\mu_2(g),\ldots,\mu_L(g)$. 
These central moments are estimated recursively, with variance as the base case.
We refer to this estimator as the \textit{Empirical Bayes Method of Moments} (EBMoM).

% After introducing estimators for the mean and the variance, the higher-order moment estimates are defined in a recursive manner.

\textbf{Step 1: Mean estimator.}
We estimate the first moment $m_1$ by 
\begin{align}\label{eq:m1hat}
\hat{m}_1 = \frac{\bm{1}^\top \X^\top \y}{\pnorm{\X \bm{1}}{}^2},
\end{align}
where $\bm{1}$ denotes the all one vector in $\R^p$. 
This estimator, which is linear in the response $\y$, is an unbiased estimator of $m_1$. 
% There is in fact a continuum of unbiased linear estimators. See \prettyref{app:bettermom}.

\textbf{Step 2: Variance estimator.}
Denote the central moments of the prior by 
\begin{align}\label{eq:central_moment}
\mu_\ell := \E_{\beta\sim g}[(\beta - m_1)^\ell].
\end{align}
The estimator for $\mu_2$, the variance of $g$, is defined as
\begin{equation}\label{eq:mu2hat}
\hat{\mu}_2 := \frac{\|\X^\top \hat{\y}\|^2-\sigma^2 \Fnorm{\X}^2}{\Fnorm{\X^\top \X}^2}, \qquad \hat{\y} := \y - \hat{m}_1 \X\bm{1}. 
\end{equation}
If $\hat m_1$ were to be replaced by the true value $m_1$, then $\hat\mu_2$ would be unbiased for $\mu_2$.

\textbf{Step 3: Estimating higher moments.}
% For $\ell\geq 2$, the estimator for $m_\ell$ is recursively defined as follows.
The estimator for central moments $\{\mu_k\}$ of order $k\geq 3$ are recursively defined as follows, with the above $\hat{\mu}_2$ as the base case.

First, we introduce the needed tensor notations. Denote the Gram matrix by 
\begin{align*}
\H:= \X^\top \X \in \R^{p\times p}.
\end{align*}
% For any $\ell\in\bN$, define a symmetric order-$\ell$ tensor $T^{(\ell)} \in (\R^p)^{\otimes \ell}$ as
% \begin{align}\label{eq:T_tensor}
% (T^{(\ell)})_{s_1,\ldots,s_\ell} := \sum_{j=1}^p \H_{j,s}\ldots \H_{j,s_\ell}, \qquad s_1,\ldots, s_\ell \in [p]. 
% \end{align}
For each $k$, define a symmetric order-$k$ tensor $\T^{(k)} \in (\R^p)^{\otimes k}$ as
\begin{align}\label{eq:T_tensor}
T^{(k)}_{s_1,\ldots,s_k} := \sum_{j=1}^p H_{j,s_1}\ldots H_{j,s_k}, \qquad s_1,\ldots, s_k \in [p]. 
\end{align}
Given any partition of $k$ into $t$ ordered 
positive integers $d_1,\ldots, d_t \geq 1$ such that $d_1 + \ldots + d_t = k$, define an upper-triangular order-$t$ tensor $\tilde{\T}^{(d_1,\ldots,d_t)} \in (\R^p)^{\otimes t}$ by
\begin{align}
\notag \tilde{T}^{(d_1,\ldots,d_t)}_{s_1,\ldots,s_t} &:= {k  \choose d_1,\ldots,d_t} T^{(k)}_{\underbrace{s_1,\ldots, s_1}_{d_1\text{ times}},\ldots, \underbrace{s_t,\ldots, s_t}_{d_t\text{ times}}}\\
&= {k  \choose d_1,\ldots,d_t} \sum_{j=1}^p H_{j,s_1}^{d_1}\ldots H_{j,s_t}^{d_t}, & \text{ if } 1 \leq s_1 < \ldots < s_t \leq p,\label{eq:T_diag_free}\\
\notag \tilde{T}^{(d_1,\ldots,d_t)}_{s_1,\ldots,s_t} &:=0 & \text{ otherwise.}
\end{align}
Note that for the finest partition with $t = k$ and $d_1=\ldots=d_k=1$, 
$\tilde{\T}^{(1,\ldots,1)}$ coincides with the upper-triangular entries of $\T^{(k)}$ up to the multiplicative constant $k!$.
Let
\begin{align}\label{eq:Ak}
A_k := \Tr \T^{(k)} = \sum_{i,j=1}^p H_{i,j}^k.
\end{align}
Note that we always have $A_k \geq 0$ since $A_k = \bm{1}^\top \H^{(k)} \bm{1}$, where $\H^{(k)}=\H\circ \ldots \circ \H$ denotes the $k$th Hadamard (entrywise) product of $\H$ and is positive semidefinite \cite[Theorem 7.5.3]{horn2012matrix}.

Next, we estimate the central moments. Given $\hat{\mu}_2,\ldots,\hat{\mu}_{k-1}$, we
construct an estimator 
$\hat{\mu}_k$ for $\mu_k$ by:
% Let $K_\ell = M^\ell \ell^{\ell/2}$ where $M$ is the subgaussian constant of the prior $g$ (see Assumption \ref{assump:prior} below for precise definitions), and $\hat{\mu}_\ell^\circ := (\hat{\mu}_\ell \wedge 2K_\ell) \vee (-2K_\ell)$ be the truncated version of $\hat{\mu}_\ell$.
\begin{align}\label{eq:moment_def}
\hat{\mu}_k = \frac{1}{A_k}\Big[F_k(\hat{\y}) - \sum_{t=2}^k \sum_{d_1 + \ldots + d_t = k: d_j \geq 2} \hat{\mu}_{d_1} \ldots \hat{\mu}_{d_t} \sum_{s_1,\ldots,s_t = 1}^p \tilde{T}^{(d_1,\ldots,d_t)}_{s_1,\ldots,s_t}\Big],
\end{align} 
where the summation is over ordered $t$-tuples $(d_1,\ldots,d_t)$ as in (\ref{eq:T_diag_free}), and 
\begin{align}\label{eq:F_k}
F_k(\hat{\y}) := \sum_{j=1}^p (\sigma\pnorm{\x_j}{})^k H_k\Big(\frac{\x_j^\top \hat{\y}}{\sigma\pnorm{\x_j}{}}\Big).
% \qquad \check{\y} = \y - \hat{m}_1 \bm{1}.
% \check{\y} = \X (\bbeta - \hat{m}_1 \bm{1}) + \beps.
\end{align}
Here $H_k(x) \equiv (-1)^k e^{x^2/2}\frac{\d^k}{\d x^k} e^{-x^2/2}$ denotes the degree-$k$ Hermite polynomial. 
Note that \prettyref{eq:moment_def}
recovers the variance estimator 
$\hat{\mu}_2$ in \prettyref{eq:mu2hat}, 
since $\hat\mu_1\equiv 0$ and $H_2(x)=x^2-1$.
Furthermore, in the special case of identity design ($\X=\I$ and $n=p$),
all tensors $\tilde{\T}^{(d_1,\ldots,d_t)}$ are zero and so is the correction term in \prettyref{eq:moment_def}, which reduces to the usual method of moments in the Gaussian sequence model.

% Finally, we define the estimator for the original (uncentered) moment $ m_\ell $ 
% \begin{equation}\label{eq:moment_def}
% \hat{m}_\ell = \sum_{k=0}^\ell \binom{\ell}{k} \hat{\mu}_k \hat{m}_1^{\ell-k}.
% \end{equation}

%for any $\x\in\R^d$ and $q\in \bN$, we write $\pnorm{\x}{\ell_q}^q = \sum_{i=1}^d x_i^q$ (without absolute value). 

Next, we discuss several salient aspects of the proposed estimator.

\paragraph{From moments to distribution}
After obtaining the moment estimates, we can convert it to an estimate of the prior. There are multiple ways to do so. For example, 
we may apply the \textit{denoised method of moments} \cite{wu2020optimal}, a special case of the generalized method of moments \cite{hansen1982large}, to produce an estimator that is a discrete distribution. Specifically, 
\begin{enumerate}
    \item For some odd $L=2k-1$, we first project the noisy moment estimates $\hat{\bmu}_L=(\hat{\mu}_1,\ldots, \hat{\mu}_L)$ with $\hat{\mu}_1 \equiv 0$ onto the space of moments by
solving the following semidefinite program \cite{wu2020optimal}:
\begin{equation}\label{eq:DMM}
\tilde{\bmu}_L := \argmin \{\pnorm{\m - \hat{\bmu}_L}{} : \m \in \cM_L([-A,A])\},
\end{equation}
where
\begin{equation}
\cM_{L}([-A,A]) := \{\bm{m}_L(\pi):\pi \text{ supported on } [-A,A]\},
\end{equation}
is the space of moment vectors $\bm{m}_L(\pi) :=(m_1(\pi),\ldots,m_L(\pi))$ over all distributions $\pi$ supported on $[-A,A]$. Here the hyperparameters $L\in\naturals$ and $A>0$ may be chosen depending on $n,p$ and the tail assumption of the true prior $g$.
\item Next, from the solution $\tilde \bmu_L = (\tilde\mu_1,\ldots,\tilde\mu_L)$ of (\ref{eq:DMM}), we construct a discrete probability distribution $\tilde{g}$ with matching moments, i.e., 
\begin{align*}
m_\ell(\tilde{g}) = \tilde{\mu}_\ell, \quad \ell=1,\ldots,L.
\end{align*}
This $\tilde g$ can be constructed to be $k$-atomic using, e.g., Gauss quadrature \cite{wu2020optimal}; we refer to \cite{golub1969calculation,gautschi2004orthogonal} for implementation details of Gauss quadrature.

\item Finally, we shift the atoms of $\tilde{g}$ by the estimated mean $\hat{m}_1$ to produce the final estimator $\hat g$.
\end{enumerate}

Alternatively, if  continuous estimates are desired, one may fit a continuous parametric family of densities (e.g.~mixture of Gaussians) by moment matching, or estimate a nonparametric density by moment matching and spline smoothing. We explore these possibilities in numerical  experiments in \prettyref{sec:exp}.

% \paragraph{Rationale of the construction}
\paragraph{Motivation of the EBMoM estimator}

The intuition behind the EBMoM estimator is the following.
Unlike the sequence model, where each moment can be estimated separately, in the linear model, the independent entries of $\bbeta$ interact through the design matrix, so that estimating a given moment requires estimates of lower order moments.

Recall that the Hermite polynomial provides an unbiased estimator of monomials in the normal mean model:
\begin{equation}
\E_{Z\sim \N(\mu,1)}[H_k(Z)] = \mu^k, \quad \mu\in\R.     \label{eq:hermite-unbiased}
\end{equation}
For simplicity, let us assume the true prior has zero mean and consider
a simplified version of the summary statistic \prettyref{eq:F_k} without centering $\y$:
\begin{align}\label{eq:F_k1}
F_k({\y}) = \sum_{j=1}^p (\sigma\pnorm{\x_j}{})^k H_k\Big(\frac{\x_j^\top \y}{\sigma\pnorm{\x_j}{}}\Big).
% \qquad \check{\y} = \y - \hat{m}_1 \bm{1}.
% \check{\y} = \X (\bbeta - \hat{m}_1 \bm{1}) + \beps.
\end{align}
Note that conditioned on $\bbeta$, we have
\begin{equation}\label{eq:y_proj}
    \frac{\x_j^\top {\y}}{\sigma\pnorm{\x_j}{}} \;\bigg|\; \bbeta  \sim \N\pth{\frac{\x_j^\top \X\bbeta}{\sigma\pnorm{\x_j}{}},1}, \qquad j = 1,\ldots,p.  
\end{equation}
Averaging over the noise $\beps$ and applying \prettyref{eq:hermite-unbiased}, we obtain
\[\Expect[F_k({\y}) \mid \bbeta] = \sum_{j=1}^p (\x_j^\top \X\bbeta)^k
= \iprod{\bbeta^{\otimes k}}{
\sum_{j=1}^p (\X^\top \X \e_j)^{\otimes k}
}
= \iprod{\bbeta^{\otimes k}}{\T^{(k)}},\]
where $\T^{(k)} \in (\R^p)^{\otimes k}$ is the tensor defined in \prettyref{eq:T_tensor}, and $\e_j$ is the $j$th standard basis of $\R^p$.
Further averaging over $\bbeta$ yields
\begin{align}\label{eq:F_k2}
\Expect[F_k({\y})] 
= \iprod{\M^{(k)}}{\T^{(k)}},
\end{align}
where 
\[
\M^{(k)} := \Expect[\bbeta^{\otimes k}]
\]
is the degree-$k$ moment tensor of $\bbeta\sim g^{\otimes p}$; equivalently, \prettyref{eq:F_k2} also equals the trace of the moment tensor of the random vector $\X^\top \X\bbeta$, whose coordinates are no longer i.i.d.
Since $M^{(k)}_{s_1,\ldots,s_k} = 
\Expect[\beta_{s_1} \ldots \beta_{s_k}]$, depending on the multiplicity of the indices, it can be expressed as monomials of the moments of $g$ thanks to $\beta_i$'s being i.i.d. For example,
$M^{(3)}_{aaa} = m_3,M^{(3)}_{aab} = m_1m_2$ and $M^{(3)}_{abc} = m_1^3$ for distinct indices $a,b,c$.
% $M^{(2)}_{i,j} = m_2$ if $i=j$ and $m_1^2$ otherwise.
Crucially, the $k$th moment of $g$ only appears on the diagonal of $\M^{(k)}$: $M^{(k)}_{a,\ldots,a} = m_k$ for $a=1,\ldots,p$. Furthermore, using the centering of $g$ (i.e., $m_1 = 0$), the inner product \prettyref{eq:F_k2} can be expanded as
\begin{equation}
\Expect[F_k({\y})]  = A_k m_k
+ \sum_{t=2}^k \sum_{d_1 + \ldots + d_t = k: d_j \geq 2} m_{d_1} \ldots m_{d_t} \sum_{s_1,\ldots,s_t = 1}^p \tilde{T}^{(d_1,\ldots,d_t)}_{s_1,\ldots,s_t},
    \label{eq:trace-expansion}
\end{equation}
where 
$A_k = 
\sum_{i,j=1}^p (\X^\top \X)_{ij}^k
$ is the trace of $\T^{(k)}$ defined in \prettyref{eq:Ak}, the middle summation is over ordered $t$-tuples $(d_1,\ldots,d_t)$, and 
$\tilde{\T}^{(d_1,\ldots,d_t)}$ is the tensor defined in 
\prettyref{eq:T_diag_free}.
In this decomposition of 
$\Expect[F_k({\y})]$, the leading term  $m_k A_k$ is desired and the remaining terms only depend on the lower-order moments. This motivates the iterative construction of the estimator \prettyref{eq:moment_def}, wherein we substitute the lower moments by their estimates. In other words, if the estimated lower moments in \prettyref{eq:moment_def} were the true values, the estimator would be unbiased.

For analysis, the major difficulty is that, even if conditioned on $\bbeta$, the $p$ columnwise projections in \prettyref{eq:y_proj} are still dependent because the columns of $\X$ are not orthogonal. This correlation is particularly pronounced in the sublinear regime of $n \ll p$.
We present the statistical guarantees in \prettyref{sec:ub}.

% \nbwu{I think here it would be good to add an example by spelling out what the resulting estimator for the second moment (and variance) and describe what are the required sample complexity for their consistency. (I remember it is much milder than .) 
% This has dual purposes.
% (a) make the exposition more concrete.
% (b) Zhou mentioned the other day that in the previous literature estimating the first two moments are well-studied. So it would be good to compare with what is known before and established methods.}
% \nbwu{OK I added this myself. Need to integrate into narrative. Details can be in app.}

% \paragraph{Connection to Rao's MIVQUE}
% \paragraph{Minimum variance unbiased estimators}
\paragraph{Unbiased polynomial estimators}

As explained in \prettyref{eq:trace-expansion}, the recursive definition of the EBMoM estimators is based on the construction of an unbiased estimator of a given moment assuming that the lower moments are known. There are in fact a continuum of such unbiased estimators, so it is natural to consider the one that minimizes the variance.
% ; what is presented in \prettyref{sec:method} amounts to a special (and natural) choice.

As an example, to estimate the mean $m_1$, for a given \textit{test vector} $\b \in \reals^n$, the estimator
\begin{align*}
\tilde m_1 = \frac{\iprod{\b}{\y}}{\iprod{\b}{\X \ones}}
\end{align*}
is unbiased for $m_1$, with variance
\begin{align*}
\Var(\tilde m_1)
= \frac{\sigma^2 \|\b\|^2 + \mu_2 \|\X^\top \b\|^2}{\iprod{\b}{\X \ones}^2}.
\end{align*}
Minimizing this variance among all test vectors leads to the optimal choice
\[\b_* = (\mu_2 \X\X^\top + \sigma^2 \I_n)^{-1} \X \ones.\]
% leading to the minimal variance $\frac{1}{(\X\ones)^\top Q^{-1} \X\ones }$
This test vector is an oracle choice because it relies on knowledge of the true variance $\mu_2$ and requires an estimate $\hat \mu_2$ in practice. (See \prettyref{app:bettermom} for details and further experiments.)
In comparison, our estimator \prettyref{eq:m1hat} corresponds to choosing $\b = \X\ones$.
Note that not all test vectors work equally well. As a bad example, choosing $\b=\ones_n$ results in an estimator that is inconsistent under the simple isotropic Gaussian design $X_{ij} \iiddistr \N(0,1/p)$ regardless of $n$ and $p$.

Moving to the variance estimator, a general strategy is to consider a quadratic form $\y^\top \Q \y$ for some $\Q$ depending on the design $\X$. Assuming the prior is centered, 
 every $\Q$ gives rise to an unbiased estimator of $\mu_2$ given by 
\begin{align*}
\tilde \mu_2 = \frac{\y^\top \Q \y - \sigma^2\tr(\Q)}{\tr(\X^\top \Q \X)}.
\end{align*}

Finding the one with minimum variance has close connections to Rao's theory of optimal quadratic estimation \cite{rao1971estimation,rao1972estimation}, which studied the problem of jointly estimating the unknown $\mu_2$ and $\sigma^2$ (with extension to multiple variance components) and found the minimum-variance quadratic unbiased estimator (MIVQUE). 
In both Rao's model and the current model (i.e., assuming $\sigma^2$ is known), the oracle MIVQUE turns out to depend on both the unknown variance $\mu_2$ itself and the kurtosis.
(See \prettyref{app:bettermom}.)

For higher-order moments, the problem for finding the optimal tensor that minimizes the variance is much more involved, and the solution, similar to the mean and variance, will ultimately depend on even higher moments. For this reason, we opt to make an explicit (and natural) choice in the EBMoM estimators. 

Finally, we note that the current EBMoM estimator (\ref{eq:F_k}) assumes that the noise variance $\sigma^2$ is given. This is a mild assumption, in the sense that when $\sigma^2$ is identifiable, it may be estimated consistently at a much lower sample complexity ($n\gg \sqrt{p}$) than that for estimating higher moments or the prior itself. See \prettyref{sec:discussion} for details.

\subsection{Statistical guarantees}\label{sec:ub}

Throughout the theoretical analysis, we make the following assumption on the prior.
% Recall that $A_1 = \pnorm{\X\bm{1}}{}^2$ and $A_2 = \pnorm{\X^\top \X}{F}^2$ as defined by (\ref{eq:Ak}).

\begin{assumption}[Prior distribution]\label{assump:prior}
The regression coefficients $\bbeta = (\beta_1,\ldots,\beta_p)$ are drawn i.i.d. from a subgaussian distribution $g$, i.e.\ for some constant $M > 0$, $(\E_{\beta \sim g} |\beta-\E_{\beta \sim g} [\beta]|^p)^{1/p} \leq M\sqrt{p}$.
%\footnote{Recall that a random variable $X$ is called subgaussian with constant $M$ if $(\E|X-\E[X]|^p)^{1/p} \leq M\sqrt{p}$ for all positive integers $p$. A random vector $\X$ is called subgaussian with constant $M$ if $\v^\top \X$ has this property for any unit-norm $\v$.} 
\end{assumption}

We start with bounding the estimation error of the first two moments, namely $\hat m_1$ in (\ref{eq:m1hat}) and $\hat\mu_2$ in (\ref{eq:mu2hat}), based on a simple bias and variance calculation. (See Section \ref{sec:proof_first_two} for a proof.)

\begin{proposition}\label{prop:MSE_first_two}
Suppose that for some fixed $c_0 \in (0,1)$ it holds that
\begin{align}\label{eq:first_two_assumption}
\frac{\pnorm{\X}{\op}}{\sigma} \geq c_0, \quad A_1=\pnorm{\X\bm{1}}{}^2 \geq c_0(n\wedge p)\pnorm{\X}{\op}^2, \quad A_2=\Fnorm{\X^\top \X}^2\geq c_0(n \wedge p)\pnorm{\X}{\op}^4.
\end{align}
Then for some $C > 0$ that depends on $(c_0,M)$, it holds that
\begin{align*}
\E(\hat m_1 - m_1)^2 + \E(\hat \mu_2 - \mu_2)^2 \leq \frac{C}{n\wedge p}. 
\end{align*}
\end{proposition}

Consequently, under the condition (\ref{eq:first_two_assumption}), which is satisfied by typical random designs (more on this in \prettyref{lem:random_lsi} next),
% $(\hat m_1,\hat \mu_2)$ is consistent for $(m_1,\mu_2)$ 
both the mean and the variance can be estimated consistently, as long as $n$ and $p$ both tend to infinity individually, without imposing conditions on their relative growth. 

%We note that in (\ref{eq:first_two_assumption}), the condition on $A_1$ is identical to the general Assumption \ref{assump:design} below, but the one for $A_2$ is slightly stronger ($n\wedge p$ versus $n^2/p \wedge p$). It can be readily checked that this stronger assumption is still satisfied with high probability by the random assemble of Lemma \ref{lem:random_lsi} below, and is essential to achieve the minimal consistency condition $n\wedge p\rightarrow\infty$.
% \nbwu{This paragraph breaks the flow and should be relocated.. Maybe somewhere after Assumption 1 or even as a footnote there.

% Also, note that Assumption 1 only includes $k\geq 3$, so we need to say in Lemma 2 that \prettyref{eq:first_two_assumption} is also satisfied. This is awkward...
% I think there is a case to be made to absorb \prettyref{eq:first_two_assumption} into Assumption 1. I am not sure yet, but we need to think this through.
% }

% \ys{I added the $A_1,A_2$ condition in Assumption 1 for now and deleted the "minimal assumption of $n\wedge p \rightarrow \infty$" - this is not true, for block design, consistent estimation of $m_1$ is possible for $n = 1$ and $p \rightarrow \infty$.}

Next we state the error bound for estimating higher central moments $\{\mu_k\}_{k\geq 3}$ in (\ref{eq:central_moment}). We make the following assumptions on the design matrix $\X$. (Note that these assumptions are only needed for theoretical analysis, not by the actual algorithm.)
\begin{assumption}[Design matrix]\label{assump:design}
There exists some $c_0 > 0$ and positive integer $k_0$ such that $ \opnorm{\X}/\sigma \geq c_0$. Moreover, $A_1$ and $A_2$ satisfy the lower bounds in (\ref{eq:first_two_assumption}), and for all $3\leq k\leq k_0$,
\begin{align}\label{eq:design_lower}
A_k = \sum_{i,j=1}^p \H_{ij}^k \geq c_0^k \pnorm{\X}{\op}^{2k} \cdot p\Big(\frac{n}{p} \wedge 1\Big)^k.
\end{align}
%\ys{the rightside is the same as $\pnorm{\X}{\op}^{2k} p [(\frac{n}{p})^k\wedge 1]$. This amounts to $p (n/p)^k$ for iid Gaussian, which is the diagonal contribution as calculated below. For large $p$, this lower bound is not tight -- for $k=2$, the bound is $n^2/p$ but the true order of $A_2$ is $n^2/p + n$.}
%\ys{For general $k$, we expect $A_k \asymp n^k/p^{k-1} + n^{k/2}/p^{k-2}$ for iid Gaussian design, and $A_k \asymp n^{k-1}/p^{k-2}$ for block design. So for $k\geq 3$, the heuristic $A_k \rightarrow\infty$ suggests the sample-complexity condition $n>> p^{1-1/k}$ for Gaussian design and $n>> p^{1-1/(k-1)}$ for block design. To summarize, for $k\geq 3$ it's ok to just use $n^k/p^{k-1}$ as the lower bound under the condition $n >> p^{1-1/k}$, but for $k=2$, we need to use the lower bound $n^2/p + n$ to obtain the optimal sample complexity.}
\end{assumption}

Let us pause to interpret the  statistical meaning of Assumption \ref{assump:design}. 
\begin{itemize}
\item First of all, the effective signal-to-noise ratio (SNR) in the linear model $\y=\X\bbeta+\beps$
may be defined as $\frac{\E_{\bbeta}[\pnorm{\X\bbeta}{}^2]}{\E[\pnorm{\beps}{}^2]} = m_2\cdot\frac{\Fnorm{\X}^2}{n\sigma^2}$ (assuming the prior is centered) \cite{Verzelen2018Adaptive}.
Since we focus on the regime where $n$ is at most proportional to and may be far smaller than $p$, the lower bound on $\opnorm{\X}/\sigma$ is necessary for the model SNR to be non-vanishing.

\item The key assumption is \prettyref{eq:design_lower} that lower bounds $A_k$. Recall from \prettyref{eq:trace-expansion} that
$A_k$ is the leading coefficient in expanding the tensor product $\iprod{\M^{(k)}}{\T^{(k)}}$  and hence may be interpreted as the signal strength for estimating the $k$th moment.
 The required lower bound in \prettyref{eq:design_lower} is in fact very mild and allows for \textit{highly correlated designs}. 
To see this, 
let us build a ``bad'' design matrix by replicating a ``good'' design matrix multiple times so that the columns are perfectly correlated: 
\begin{equation}
\X = \frac{1}{\sqrt{m}} \underbrace{[\mathbf{O},\ldots,\mathbf{O}]}_{\text{$m=p/n$ times}}
    \label{eq:baddesign}
\end{equation}
where $\mathbf{O}$ is an $n\times n$ orthogonal design matrix, and the normalizing factor $\frac{1}{\sqrt{m}}$ ensures that $\opnorm{\X}=\opnorm{\mathbf{O}} = 1$.
Even for such an extreme case of correlated design, we have $A_k = p (n/p)^{k-1}$ for all $k\geq 1$, which satisfies \prettyref{assump:design}.
\item The factor $(\frac{n}{p}\wedge 1)^k$ in (\ref{eq:design_lower}) can in fact be relaxed to $(\frac{n}{p}\wedge 1)^{\lambda_k}$ for any non-negative sequence $\{\lambda_k\}$, under which the MSE bound in (\ref{def:Delta_k}) below will be replaced by $(Ck)^{5k^2}\frac{1}{n\wedge p}(\frac{p}{n}\vee 1)^{C\Lambda_k}$ with $\Lambda_k = \sum_{\ell=3}^k \lambda_\ell$ and some universal $C > 0$. We choose $\lambda_k = k$ in (\ref{eq:design_lower}) to be the smallest exponent that can accommodate random designs in Lemma \ref{lem:random_lsi} below. 
\end{itemize}

The next lemma shows that \prettyref{assump:design} is satisfied with high probability by a wide class of random design matrices, including Gaussian designs whose rows are sampled independently from $\N(0,\frac{\sigma^2}{n+p} \bSigma)$ for some 
$\bSigma$ with bounded spectrum; see Section \ref{subsec:proof_assumption} for a proof.

% \nbwu{Need to explain along the lines of either ``why the RHS of \prettyref{eq:design_lower} is natural'', or something like (Is this actual the case?) ``in fact, we may assume any lower bound
% $\Big(\frac{n}{p} \wedge 1\Big)^{m_k}$
% for any exponent $m_k$ growing with $k$, which will then lead to 
% Thm 3 below with $(\frac{p}{n}\vee 1)^{C M_k}$ with $M_k=m_1+\ldots+m_k$ and suffices for consistency under the sample complexity $n=p^{1-o(1)}$.''
% }

% As shown in \prettyref{eq:trace-expansion}, $A_k$ is the leading coefficient in expanding the tensor product $\iprod{\M^{(k)}}{\T^{(k)}}$ and hence may be interpreted as the signal strength of the problem. 
% The following lemma shows that Assumption \ref{assump:design} is satisfied with high probability by a wide class of random design matrices, which includes isotropic Gaussian design ($X_{ij} \iiddistr \N(0,\frac{1}{p})$) as a special case. Its proof is given in Section \ref{subsec:proof_assumption}.

\begin{lemma}\label{lem:random_lsi}
Suppose that 
\begin{align*}
\X = \lambda(n,p)\bar\X
\end{align*}
for some $\lambda(n,p) \geq \sigma/\sqrt{n+p}$, where $\bar \X$ satisfies the following assumptions.
\begin{enumerate}
\item The rows $(\bar\x_1,\ldots,\bar\x_n)$ of $\bar\X \in \R^{n\times p}$ are independent and subgaussian\footnote{A random vector $\x$ is subgaussian with constant $C$ if $\v^\top \x$ is subgaussian with constant $C$ for any unit-norm vector $\v$.} with constant $C_0$, and have mean zero and covariance $\bSigma$ satisfying $ \kappa \I_p \prec \bSigma \prec \kappa^{-1}\I_p $ for some $\kappa > 1$.
\item Each $\bar\x_i \in \R^p$ satisfies a Poincar\'e inequality with constant $C_{\PI} > 0$:
\begin{align*}
\Var(f(\bar\x_i)) \leq C_{\PI}\E\pnorm{\nabla f(\bar\x_i)}{}^2, \quad i\in [n]
\end{align*}
for every continuously differentiable $f:\R^p\rightarrow \R$.
\end{enumerate}
Further suppose that $n \geq p^{1-r}$ for some $r \geq 0$. Then with probability converging to $1$ as $n,p\rightarrow\infty$, Assumption \ref{assump:design} is satisfied with some $c_0 = c_0(\kappa, C_{\PI}, C_0) > 0$ and all $k_0 \leq c(r^{-1}\wedge n \wedge \frac{\log p}{\log\log p})$ for some small universal constant $c > 0$.
\end{lemma}

Equipped with Assumptions \ref{assump:prior} and \ref{assump:design}, we are ready to state the main result of the paper, which bounds the MSE of the EBMoM estimator truncated on a high-probability event. See Section \ref{subsec:proof_upper_main} for a proof.
\begin{theorem}\label{thm:main_moment}
Suppose that Assumption \ref{assump:design} holds for some $c_0, k_0$, and Assumption \ref{assump:prior} holds with some $M$. Let
\begin{align}\label{def:Delta_k}
\Delta_k = (Ck)^{5k^2}  \frac{1}{n \wedge p}\Big( \frac{p}{n} \vee 1 \Big)^{2k^2},
\end{align}
where $C>0$ is a large enough constant depending only on $c_0,M$.
Then for any $2\leq k\leq k_0$, there exists a $(\X,\y)$-measurable event $\cC_k$ with probability at least $1-\Delta_k$, such that
\begin{align}\label{eq:MSE}
\E[(\hat\mu_k - \mu_k)^2\bm{1}_{\cC_k}] \leq \Delta_k.
\end{align}
\end{theorem}

Using Markov's inequality, 
it is straightforward to convert the above truncated MSE into a high-probability error bound, which states that the typical order of the error $|\hat\mu_k - \mu_k|$ is $\sqrt{\Delta_k}$.
As such, if $n\geq p^{1-\eps}$ and $p\to\infty$, the EBMoM method consistently and simultaneously estimates the first $\Theta(\frac{1}{\sqrt{\eps}})$ moments. We note that this $\eps$-dependence, or equivalently, the $k$-dependence in the exponent of (\ref{def:Delta_k}), is not tight; we refer to Appendix \ref{app:m3} for a refined analysis of $\hat\mu_3$ in the special case of i.i.d.\ Gaussian design, where we show that $n \gg p^{2/3}$ suffices for consistency.

The above reasoning suggests that if $n\wedge p\rightarrow\infty$ and $n\geq p^{1-\eps}$ for any $\eps = o(1)$, the prior can be estimated consistently. This is formalized by the next corollary (proved in \prettyref{sec:pf-cor_prior}), which converts the statistical guarantee on estimating the moments to that of estimating the prior under the 1-Wasserstein loss.\footnote{The $W_1$ distance between one-dimensional distributions is the $L_1$-distance between their cumulative distribution functions.}

% The next following corollary which we convert the statistical guarantee on estimating the moments to that of estimating the prior.
% If $n=p^{1-o(1)}$, the moment-based estimator is consistent under the 1-Wasserstein loss.\footnote{The $W_1$ distance between one-dimensional distributions is the $L_1$-distance between their cumulative distribution functions.}

% Finally, we convert the statistical guarantee on estimating the moments to that of estimating the prior. If $n=p^{1-o(1)}$, the moment-based estimator is consistent under the 1-Wasserstein loss.\footnote{The $W_1$ distance between one-dimensional distributions is the $L_1$-distance between their cumulative distribution functions.}

\begin{corollary}\label{cor:prior}
Suppose that Assumptions~\ref{assump:prior}--\ref{assump:design} hold, and let
\[
N:=n\wedge p.
\]
Suppose that $0<\varepsilon\le 1/2$ and $n\ge p^{1-\varepsilon}$.
Let
\begin{align}\label{eq:k_def}
k
=
\left\lfloor
c\min\left\{
\frac{1}{\sqrt{\varepsilon}},
\sqrt{\frac{\log N}{\log\log N}}
\right\}
\right\rfloor,
\qquad
L=2k-1,
\end{align}
where $c>0$ is any sufficiently small constant, and suppose that
$L\le k_0$. Let
\[
A=\sqrt{C_A k\log k},
\]
where $C_A>0$ is sufficiently large. Let $\widetilde g$ be the DMM estimator obtained by applying
\eqref{eq:DMM} to the centered moment vector
\[
(0,\widehat\mu_2,\ldots,\widehat\mu_L)
\]
on the interval $[-A,A]$, and define the final estimator
\[
\widehat g
=
\tau_{\widehat m_1}\#\widetilde g,
\qquad
\tau_a(x):=x+a.
\]
Then there exist constants $C_1,C_2,C>0$, depending only on
$c_0$ and $M$, such that
\[
\mathbb P\left(
W_1(g,\widehat g)
>
C\sqrt{\frac{\log k}{k}}
\right)
\le
\frac{k^{C_1k^2}}{N}
+
k^{-C_2}.
\]
\end{corollary}
As long as $\eps \rightarrow 0$ in the sample complexity condition $n\geq p^{1-\eps}$ as $n\wedge p \rightarrow\infty$, by taking the constant $c$ in (\ref{eq:k_def}) to be sufficiently small, $k$ diverges sufficiently slow such that the right side of the above display vanishes as $N\rightarrow\infty$, implying that $W_1(g,\hat g) \rightarrow 0$ with probability converging to $1$. In contrast, the closet sample complexity in the literature is \cite{fan2023gradient}, where it was shown that for a class of random sub-Gaussian designs, the idealized NPMLE is consistent for the true (non-parametric) prior when $n\geq c p$ for any fixed $c > 0$. While this already suggests that consistency may be possible at a sub-linear sample complexity, there is no known algorithms that provably compute the NPMLE for the linear model in general, and existing analysis of its implementable algorithms requires various extra assumptions \cite{fan2023gradient,fan2025dynamicalII}. To our best knowledge, Corollary \ref{cor:prior} gives the first theoretical guarantee for a computationally feasible and consistent prior estimator with the optimal sub-linear sample complexity.

% \begin{corollary}\label{cor:prior}
% Given Assumptions \ref{assump:design}--\ref{assump:prior} with parameters $c_0,\sigma,M$, there exist constants $c,C$ such that the following holds.
% Let $n\geq p^{1-\eps}$. 
% Let $\hat g$ be the DMM estimator \prettyref{eq:DMM} applied to the   $L=2k-1$ moments $\hat m_1,\hat\mu_2\ldots,\hat \mu_{L}$ estimated by the EBMoM, where
% \begin{align*}
% k = c \min\sth{\frac{1}{\sqrt{\eps}}, 
% \sqrt{\frac{\log n}{\log\log n}}},
% \quad
% A = \sqrt{C k \log k}.
% \end{align*}
% % Under the assumption of \prettyref{thm:main_moment}, 
% Then with probability $1-k^{Ck^2} n^{-\frac{1}{Ck}}$,
% \begin{align*}
% W_1(\hat g,g) \leq C \sqrt{ \frac{\log k}{k}}.
% \end{align*}
% \end{corollary}

\subsection{Computation of the EBMoM estimator}

We discuss the computational details of the moment estimates $\widehat\mu_k$ in (\ref{eq:moment_def}), 
where the correction term  involves a summation of 
\[\gamma^{(d_1,\ldots,d_t)} := \sum_{s_1,\ldots,s_t=1}^p \tilde T_{s_1,\ldots,s_t}^{(d_1,\ldots,d_t)},\]
over partitions of 
$k$ as $d_1 + \ldots + d_t$
with $d_i \geq 2$ and $t \leq k/2$. 
A na\"ive term-by-term calculation results in a time complexity of $O(p^{k/2})$. Nevertheless, we show below that 
it can be computed in linear time $O_k(p)$.
Taking into account the computation of the main term $F_k$'s, the overall time complexity for (\ref{eq:moment_def}) is $O_k(p(n+ p))$ once $\H=\X^\top\X$ is available.

To this end,  we group the ordered
partitions $(d_1,\ldots,d_t)$ of $k$ according to their underlying multiset.
For $2\le t\le \lfloor k/2\rfloor$, let
\[
\mathcal P_t(k)
:=
\left\{
(f_1,\ldots,f_t)\in\mathbb N^t:
2\le f_1\le\cdots\le f_t,\quad
f_1+\cdots+f_t=k
\right\}.
\]
For $\boldsymbol f=(f_1,\ldots,f_t)\in\mathcal P_t(k)$, let
\[
\operatorname{Orb}(\boldsymbol f)
:=
\left\{
(f_{\pi(1)},\ldots,f_{\pi(t)}):\pi\in\mathfrak S_t
\right\}
\]
denote the set of distinct permutations of $\boldsymbol f$, and define
\[
m_r(\boldsymbol f)
:=
\#\{a\in[t]:f_a=r\},
\qquad
n(\boldsymbol f)
:=
\prod_{r=2}^k m_r(\boldsymbol f)!.
\]
Although $\gamma^{(d_1,\ldots,d_t)}$ generally depends on the order of
$(d_1,\ldots,d_t)$, the moment product
$\widehat\mu_{d_1}\cdots\widehat\mu_{d_t}$ is invariant under permutations.
Consequently,
\begin{align}
\sum_{\substack{d_1+\cdots+d_t=k\\d_a\ge2}}
\widehat\mu_{d_1}\cdots\widehat\mu_{d_t}
\gamma^{(d_1,\ldots,d_t)} =
\sum_{\boldsymbol f\in\mathcal P_t(k)}
\widehat\mu_{f_1}\cdots\widehat\mu_{f_t}
\Gamma_{\boldsymbol f},
\label{eq:correction1}
\end{align}
where
\[
\Gamma_{\boldsymbol f}
:=
\sum_{\boldsymbol d\in\operatorname{Orb}(\boldsymbol f)}
\gamma^{\boldsymbol d}.
\]
Using the definition of $\widetilde T$ in~\eqref{eq:T_diag_free}, we have
\begin{align}
\Gamma_{\boldsymbol f}
&=
\binom{k}{f_1,\ldots,f_t}
\sum_{j=1}^p
\sum_{\boldsymbol d\in\operatorname{Orb}(\boldsymbol f)}
\sum_{1\le s_1<\cdots<s_t\le p}
H_{j,s_1}^{d_1}\cdots H_{j,s_t}^{d_t}
\nonumber \\
&=
\frac{\binom{k}{f_1,\ldots,f_t}}{n(\boldsymbol f)}
\sum_{j=1}^p
\sum_{\substack{s_1,\ldots,s_t\in[p]\\
s_1,\ldots,s_t\ \mathrm{pairwise\ distinct}}}
H_{j,s_1}^{f_1}\cdots H_{j,s_t}^{f_t}. \label{eq:gammaf}
\end{align}
Indeed, for each unordered set of $t$ distinct indices, every distinct
assignment of the exponent multiset $\{f_1,\ldots,f_t\}$ is counted
$n(\boldsymbol f)$ times in the sum over ordered, pairwise distinct
$(s_1,\ldots,s_t)$.

For fixed $j\in[p]$ and
$\boldsymbol f=(f_1,\ldots,f_t)\in\mathcal P_t(k)$, define
$x=(x_1,\ldots,x_p):=H_{j,\cdot}$, the $j$th row of $\H$.
The inner sum in \prettyref{eq:gammaf} is then
\begin{equation}
\sum_{\substack{s_1,\ldots,s_t\in[p]\\
s_1,\ldots,s_t\ \mathrm{pairwise\ distinct}}}
x_{s_1}^{f_1}\cdots x_{s_t}^{f_t}.
    \label{eq:coeff}
\end{equation}
As opposed to summing this term by term in 
$O(p^t)$ time,  Algorithm~\ref{alg:dp} below computes this coefficient in
$O(t2^t p)$ time and $O(2^t)$ memory.  Summing over $j\in[p]$,  $\Gamma_{\boldsymbol f}$ for each fixed $\boldsymbol f$ can be computed in $O(t2^t p^2)$ time. Since the number of
pairs $(t,\boldsymbol f)$ depends only on $k$, the entire \prettyref{eq:correction1}
can be computed in $O_k(p^2)$ time.

Note that \prettyref{eq:coeff} 
is the coefficient of $y_1\cdots y_t$ in
\begin{equation}\label{eq:y_poly}
\prod_{i=1}^p
\left(
1+\sum_{a=1}^t y_a x_i^{f_a}
\right).
\end{equation}
To compute this coefficient (and those for other monomials in $y$) in \eqref{eq:y_poly}, Algorithm~\ref{alg:dp} uses the following standard dynamic programming approach: We use 
\begin{align*}
\texttt{dp}: \{0,\ldots,N-1\} \rightarrow \R
\end{align*}
to store the coefficients of the $N = 2^t$ monomials $y_1^{b_1}\ldots y_t^{b_t}$ with $b_s \in \{0,1\}$, $1\leq s\leq t$. Each $b \in \{0,\ldots,N-1\}$, which we call a ``state", is identified with its $t$-bit binary expansion, where the $j$th bit indexes whether $y_j$ is chosen -  e.g., $\texttt{dp}[(0,\ldots,0,0)]$, $\texttt{dp}[(0,\ldots,0,1)]$, $\texttt{dp}[(0,\ldots,0,1,0,1)]$, $\texttt{dp}[(1,\ldots,1)]$ will represent the coefficients of $1, y_1, y_1y_3, y_1\ldots y_t$ respectively. We initialize the states as
\begin{align*}
\texttt{dp}[0] = 1, \quad \texttt{dp}[index] = 0,\quad 1\leq index \leq N-1.
\end{align*}
We then update this coefficient map by sequentially processing the $p$ terms $1 + \sum_{a=1}^t y_a x_i^{f_a}$ for $1\leq i\leq p$. When processing the $i$th term, we cycle through the $N$ states and their $t$-bits, and the key update is: for each state $0\leq index\leq N-1$ and bit $0\leq j \leq t-1$ such that the $j$th bit of $index$ is $0$, 
\begin{align*}
\texttt{new\_dp}[new\_index] \mathrel{+}= \texttt{dp}[index] \cdot x_i^{\,f_{j+1}}, 
\end{align*}
where $new\_index$ is obtained by changing the $j$th bit of $index$ from $0$ to $1$. After $p$ steps, we output $\texttt{dp}[N-1]$ as the desired coefficient of $y_1\ldots y_t$ in (\ref{eq:y_poly}). Since for each step we need to cycle through the $N$ states and their $t$-bits, the entire algorithm has time complexity $O(t2^tp)$.

\begin{algorithm}[H]
\caption{Dynamic Programming for computing the Coefficient of $y_1y_2\cdots y_t$ in \prettyref{eq:y_poly}}
\label{alg:dp}
\begin{algorithmic}[1]
\State \textbf{Input:} A list $x = \{x_1,\dots,x_p\}$ and an exponent vector $d = (f_1,\dots,f_t)$.
\State \textbf{Output:} The coefficient of 
$
y_1y_2\cdots y_t \  \text{in} \ F(y_1,\dots,y_t)=\prod_{i=1}^{p}\Bigl(1+\sum_{a=1}^{t}y_a\,x_i^{f_a}\Bigr).
$
\State $N \gets 2^t$  \Comment{$N$ is the total number of states}
\State Initialize $\texttt{dp}[0],\ldots,\texttt{dp}[N-1]$ as $\texttt{dp}[0]\gets 1$ and $\texttt{dp}[index]\gets 0$ for $index\neq 0$ 
\For{$i=1$ to $p$}
    \State $\texttt{new\_dp} \gets \texttt{dp}$
    \For{each $index \in \{0,\dots,N-1\}$}
        \For{$j=0$ to $t-1$}
            \If{$ \text{the } j\text{th bit of } index \text{ is } 0$}
                \State $new\_index \gets index \; | \; (1\ll j)$ \Comment{new\_index equals index with its $j$th bit set to 1}
                \State $\texttt{new\_dp}[new\_index] \mathrel{+}= \texttt{dp}[index] \cdot x_i^{\,f_{j+1}}$
            \EndIf
        \EndFor
    \EndFor
    \State $\texttt{dp} \gets \texttt{new\_dp}$
\EndFor
\State \Return $\texttt{dp}[N-1]$
\end{algorithmic}
\label{alg:better}
\end{algorithm}

\section{Lower bounds and sharp sample complexity}\label{sec:lb}
In this section, we provide a statistical lower bound for learning the prior that matches the sample complexity upper bound of the EBMoM estimator, certifying its optimality. 
% \nbwu{I will add a warmup part on block design that explains the main difference with sequence models: in addition to convolution with noise, the new phenomenon of self-convolution. In other words, in order to learn the prior $g$, one needs to deconvolve $g$ from both the noise distribution and itself. This example can also explain the $n=p^{1-o(1)}$ sample complexity using CLT -- using moment-matching construction to achieve super-CLT rate of convergence.}

Let us start by explaining the major distinction between the linear model and the sequence model from a deconvolution point of view. 
As an instructive example, let us revisit the extremely 
correlated design matrix  \prettyref{eq:baddesign} in \prettyref{sec:ub},  with $\tilde\X =\I_n$. 
In other words, consider the following \textit{block design}, where
% whose rows all have unit norm and disjoint support like so:
$p = n m$ for some integer $m \geq 1$, and each of the $n$ rows has exactly $m$ non-zeros all equal to $\frac{1}{\sqrt{m}}$ and disjoint support:
\begin{equation}
    \X =
    \frac{1}{\sqrt{m}}
%     \left[
% \begin{array}{c}
% \underbrace{1 \ \cdots \ 1}_{m} \ 0 \ \cdots \ 0 \ \cdots \ 0 \\[6pt]
% 0 \ \cdots \ 0 \ \underbrace{1 \ \cdots \ 1}_{m} \ \cdots \ 0 \\[6pt]
% \ddots \\[6pt]
% 0 \ \cdots \ 0 \ 0 \ \cdots \ 0 \ \underbrace{1 \ \cdots \ 1}_{m}
% \end{array}
% \right]
\left[
\begin{array}{c}
1 \ \cdots \ 1 \ 0 \ \cdots \ 0 \ \cdots \ 0 \\[6pt]
0 \ \cdots \ 0 \ 1 \ \cdots \ 1 \ \cdots \ 0 \\[6pt]
\ddots \\[6pt]
0 \ \cdots \ 0 \ 0 \ \cdots \ 0 \ 1 \ \cdots \ 1
\end{array}
\right]
\label{eq:blockdesign}
\end{equation}
%  Specifically, assume that $p = n m$ for some integer $m \geq 1$. The $n\times p$ design $\X$ satisfies that, for the first row
% $X_{1j}=\frac{1}{\sqrt{m}}$ for $j=1,\ldots,m$ and 0 otherwise; for the second row, $X_{2j}=\frac{1}{\sqrt{m}}$ for $j=m+1,\ldots,2m$ and 0 otherwise; etc.
Thus, $\btheta = \X \bbeta$ has i.i.d.\ coordinates drawn from 
\begin{equation}
g^{(m)} := \text{Law}\pth{\frac{\beta_1+\ldots+\beta_m}{\sqrt{m}}}, 
\quad \beta_i \iiddistr g,    \label{eq:self-conv}
\end{equation}
which is the $m$-fold convolution of $g$ with itself rescaled by $1/\sqrt{m}$.
 Then $\y = \btheta+\beps$ has i.i.d.\ coordinates drawn from $g^{(m)} * \N(0,\sigma^2)$. As such, the problem reduces to a sequence model, where $g^{(m)}$ plays the role of the prior, and
 in order to learn the prior $g$, one needs to deconvolve $g$ from both the noise distribution and itself. This 
new phenomenon of \textit{self-deconvolution} is a major distinction between linear models and sequence models.

For dense designs, the situation is more complicated as all $p$ regression coefficients $\beta_1,\ldots,\beta_p$ will participate in the self-convolution and contribute to all responses; nevertheless, this simple example of block design is already enough to demonstrate the sample complexity lower bound $n=p^{1-o(1)}$. To see this, the main observation is that if $g$ is centered and has variance 1, due to the Central Limit Theorem (CLT), the $m$-fold self-convolution $g^{(m)}$ in \prettyref{eq:self-conv} approaches the standard normal as $m$ grows.
The typical convergence rate (Berry-Esseen) in CLT is $\frac{1}{\sqrt{m}}$, but this can be sped up arbitrarily by choosing $g$ to have many matching moments with the normal limit. The same idea can be used to construct a pair of priors $g,g'$ matching a constant $L$ number of moments, such that the moments of their self-convolutions $g^{(m)},g'^{(m)}$ are even closer, leading to the indistinguishability of the Gaussian convolutions 
$g^{(m)} * \N(0,\sigma^2)$ and $g'^{(m)} * \N(0,\sigma^2)$. This is formalized by the next result proved in \prettyref{sec:lower_bound-block}.

\begin{proposition}\label{prop:lower_bound-block}
Consider the block design \prettyref{eq:blockdesign}, with $m=\frac{p}{n} \in \naturals$.
For every $L\geq 2$, there exist some small universal constant $c > 0$ and a pair of distinct priors $g,g'$ on $[-c,c]$, such that
\begin{equation*}
\TV\big(P_g(\y), P_{g'}(\y)\big) \leq Cnm^{-\frac{L}{2}+1}
\end{equation*}
where $P_g(\y), P_{g'}(\y)$ denote the law of $\y$ (conditioning on the block design $\X$) under prior $g,g'$ respectively, and $C$ is a universal constant.
\end{proposition}

\prettyref{prop:lower_bound-block} provides a (Le Cam) two-point lower bound for estimating the prior.
Suppose that $n\leq p^{1-\delta}$ for any constant $\delta$. We can choose  $g \neq g'$ with $L=4/\delta$, such that $\TV(P_g(\y), P_{g'}(\y)) =o(1)$; that is, one cannot distinguish whether $g$ or $g'$ is the true prior based on $\y$.
% which implies the impossibility of consistently estimating the prior $g$. 
This shows that for the block design (\ref{eq:blockdesign}), consistent estimation of the prior requires sample size at least $n=p^{1-o(1)}$.

% first $L$ entries are all equal to $\frac{1}{L}$ and the rest are zero; for the second row, the 

% first $L$ entries are all equal to $\frac{1}{L}$ and the rest are zero; for the second row, the 

% A much more general result than \prettyref{prop:lower_bound-block} 
% is the following theorem that applies to general designs. 
The following theorem extends this  lower bound to more general designs. Compared with the block design example in \prettyref{prop:lower_bound-block}, the major technical difficulty is that $\y$ no longer has a product law. See \prettyref{sec:lower_bound-pf} for a proof.
\begin{theorem}\label{thm:lower_bound}
Suppose that $n\leq p^{1-\delta}$ for some fixed $\delta > 0$ and the design $\X$ satisfies 
%\nbwu{There is some notational clash with $\tilde\X$ in
%\prettyref{eq:baddesign}.}
\begin{align}\label{eq:design_assumption_lower}
\max_{j\in[p]}\pnorm{\tilde \X \e_j}{} \leq C_0 \Big(\frac{n}{p}\Big)^\alpha (\log p)^\gamma, \quad \text{ where } \tilde \X = (\X\X^\top + \sigma^2 \I_n)^{-1/2}\X
\end{align}
for some fixed $\alpha  > 0$ and $C_0,\gamma \geq 0$. Then there exists a pair of distinct 1-subgaussian distributions $g,g'$ such that, for all large enough $n,p$,
\begin{align*}
\dTV\big(P_g(\y), P_{g'}(\y)\big) \leq 0.1.
\end{align*}
\end{theorem}

Both Assumption \ref{assump:design} needed in Corollary \ref{cor:prior} and the condition (\ref{eq:design_assumption_lower})  are satisfied by a broad variety of design matrices, for which the tight sample complexity for estimating the prior is shown to be $n=p^{1-o(1)}$. For example, it can be readily checked that (\ref{eq:design_assumption_lower}) is satisfied by the block design (\ref{eq:blockdesign}) with $\alpha = 1/2$ and $\gamma = 0$. The following lemma (proved in  Section \ref{subsec:proof_lower_gaussian}) confirms that (\ref{eq:design_assumption_lower}) is satisfied
with high probability by various examples of random designs considered in Lemma \ref{lem:random_lsi}. 

\begin{lemma}\label{lem:lower_gaussian}
Suppose that $\X = \frac{\sigma}{\sqrt{n+p}}\bar \X$, where the rows of $\bar \X$ are independent and sub-Gaussian with constant $C_0$, and have mean zero and covariance $\bSigma$ satisfying $ \kappa^{-1} \I_p \prec \bSigma \prec \kappa\I_p $ for some $\kappa > 1$. Then with probability converging to $1$, the condition (\ref{eq:design_assumption_lower}) is satisfied with $\alpha = \gamma = 1/2$.
%\nbwu{The sentence above says ``$\alpha = 1/2$ and $\gamma = 0$''. Typo?}
\end{lemma}

\section{Experiments}
\label{sec:exp}

% \begin{figure}[ht]
%     \centering
%     \includegraphics[width=0.5\linewidth]{figs/W1_dmm_rademacher_vs_p_k4_T50_seed42_sigma1.0_osigma1.0_rho0.0.pdf}
%     \caption{Rademacher, iid design.}
%     \label{fig:rademacher-dmm}
% \end{figure}

% \begin{figure}[!ht]
%     \centering
%     \begin{subfigure}[t]{0.48\linewidth}
%         \centering
%         \includegraphics[width=\linewidth]{block_iid_m3_1.pdf}
%         \caption{Third moment $\hat m_3$}
%         \label{fig:m3}
%     \end{subfigure}
%     \hfill
%     \begin{subfigure}[t]{0.48\linewidth}
%         \centering
%         \includegraphics[width=\linewidth]{block_iid_m4_1.pdf}
%         \caption{Fourth moment $\hat m_4$}
%         \label{fig:m4}
%     \end{subfigure}
%     \caption{Histograms of the estimates $\hat m_3$ and $\hat m_4$ for i.i.d.~and correlated block design with $\rho = 0.9$}
%     \label{fig:m3m4}
% \end{figure}

\begin{figure}[ht]
     \centering
     \begin{subfigure}[b]{0.45\textwidth}
         \centering
         \includegraphics[width=\textwidth]{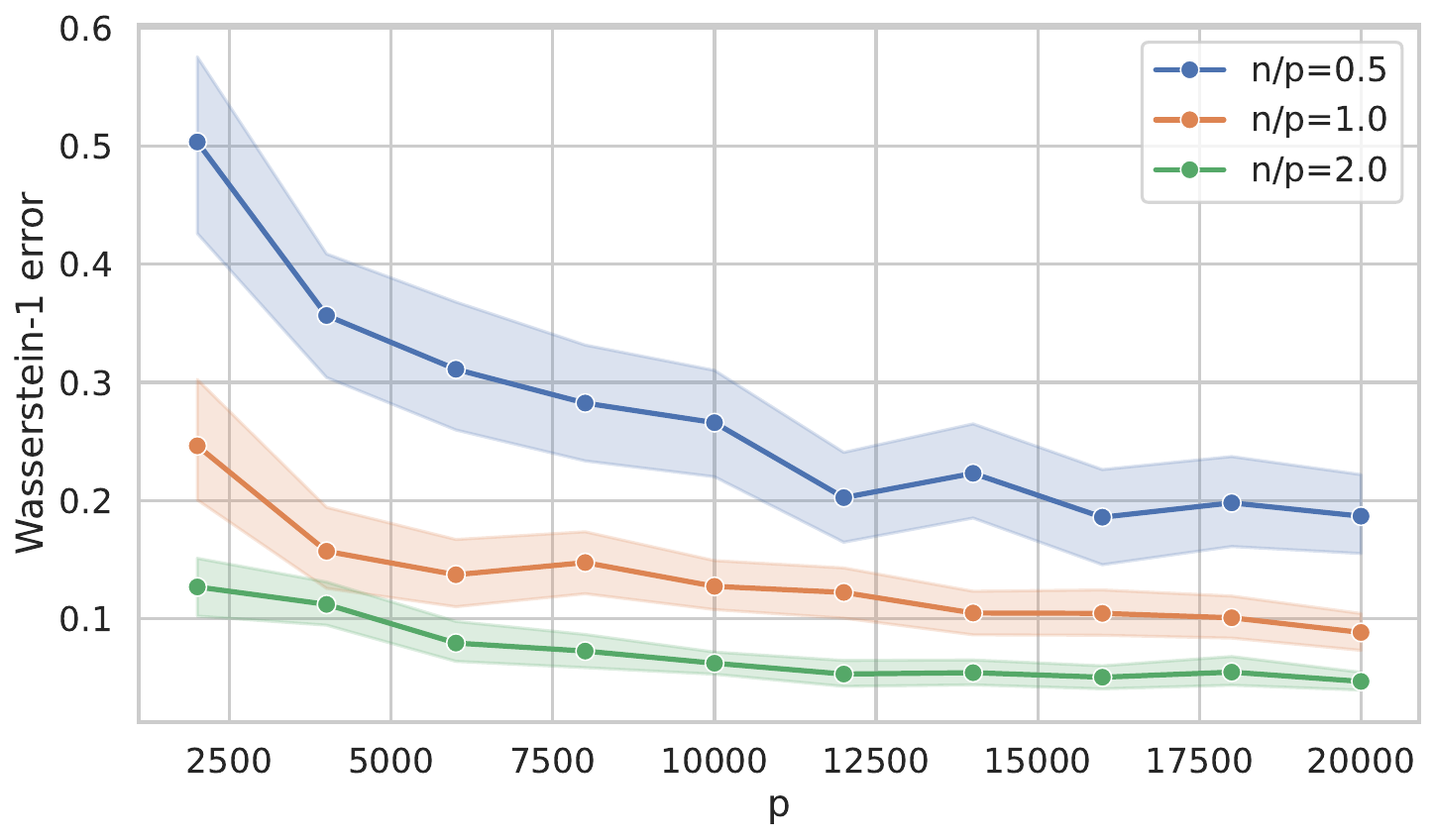}
         \caption{i.i.d.~design}
         \label{fig:rademacher-dmm_rho0.0}
     \end{subfigure}
     \hfill
     \begin{subfigure}[b]{0.45\textwidth}
         \centering
         \includegraphics[width=\textwidth]{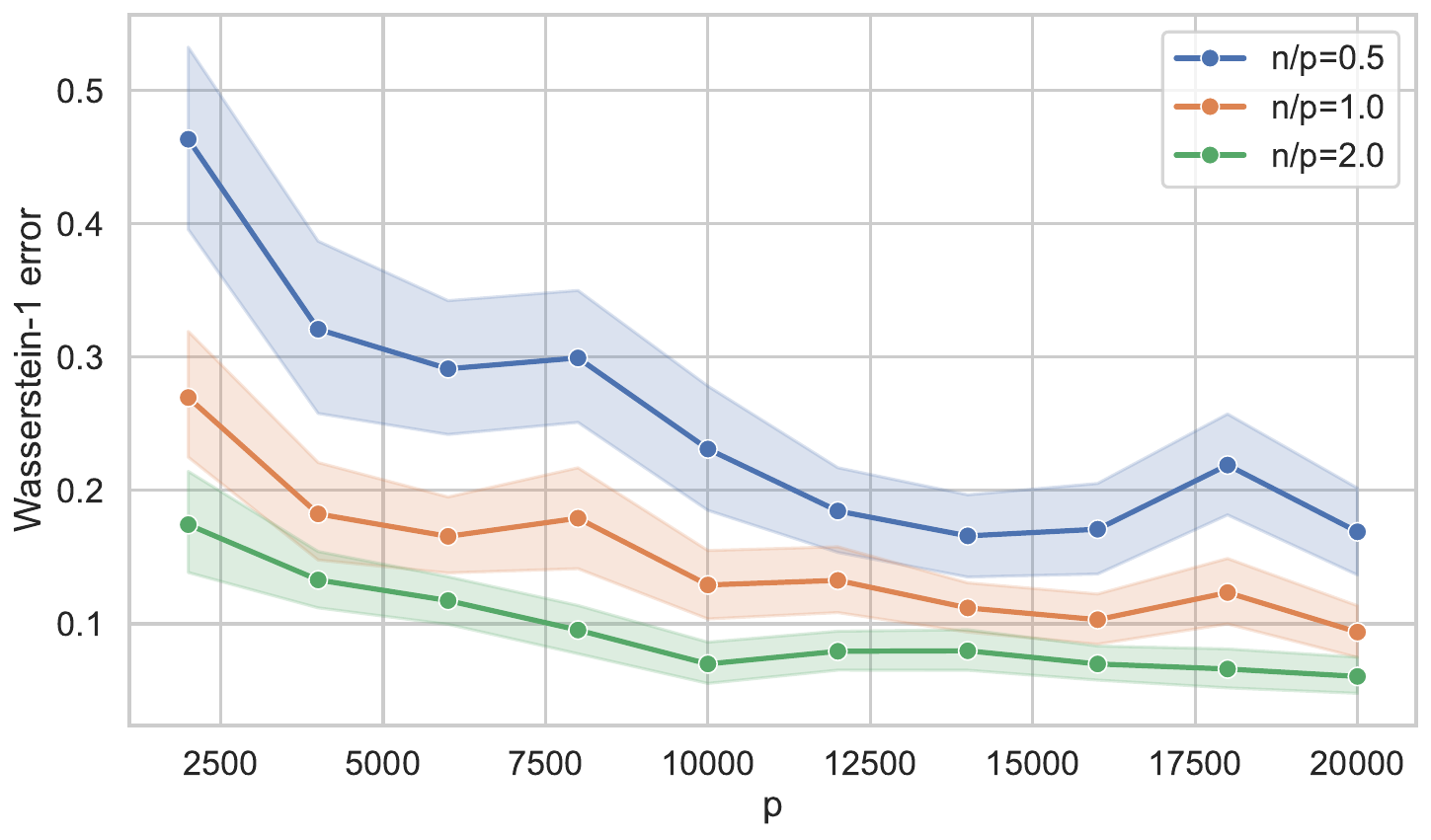}
         \caption{Correlated block design with $\rho = 0.9$}
         \label{fig:rademacher-dmm_rho0.2}
     \end{subfigure}
        \caption{Rademacher prior. Method: DMM.}
        \label{fig:rademacher-dmm}
\end{figure}

\begin{figure}[ht]
     \centering
     \begin{subfigure}[b]{0.45\textwidth}
         \centering
         \includegraphics[width=\textwidth]{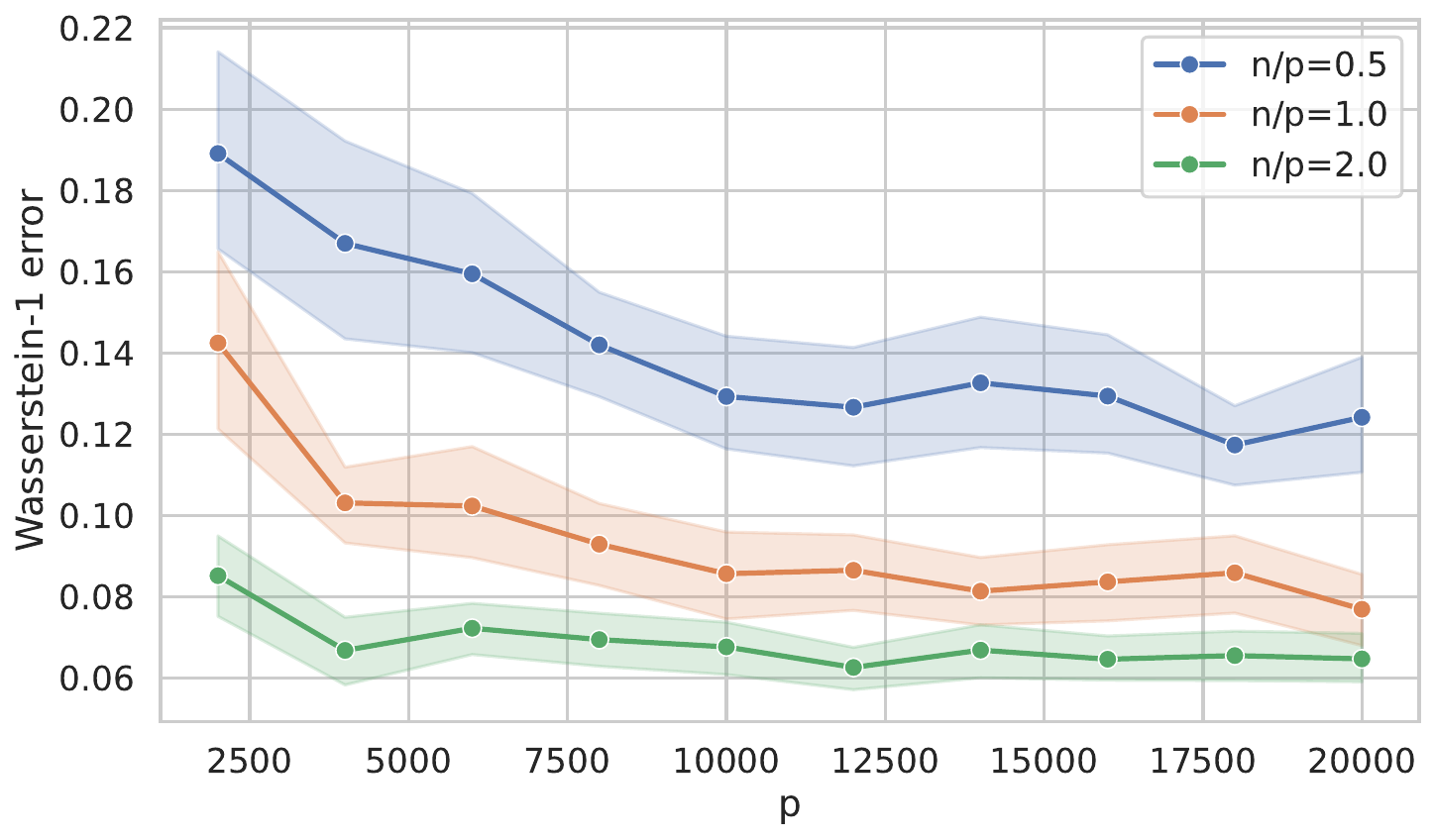}
         \caption{i.i.d.~design}
         \label{fig:sparse-dmm_rho0.0}
     \end{subfigure}
     \hfill
     \begin{subfigure}[b]{0.45\textwidth}
         \centering
         \includegraphics[width=\textwidth]{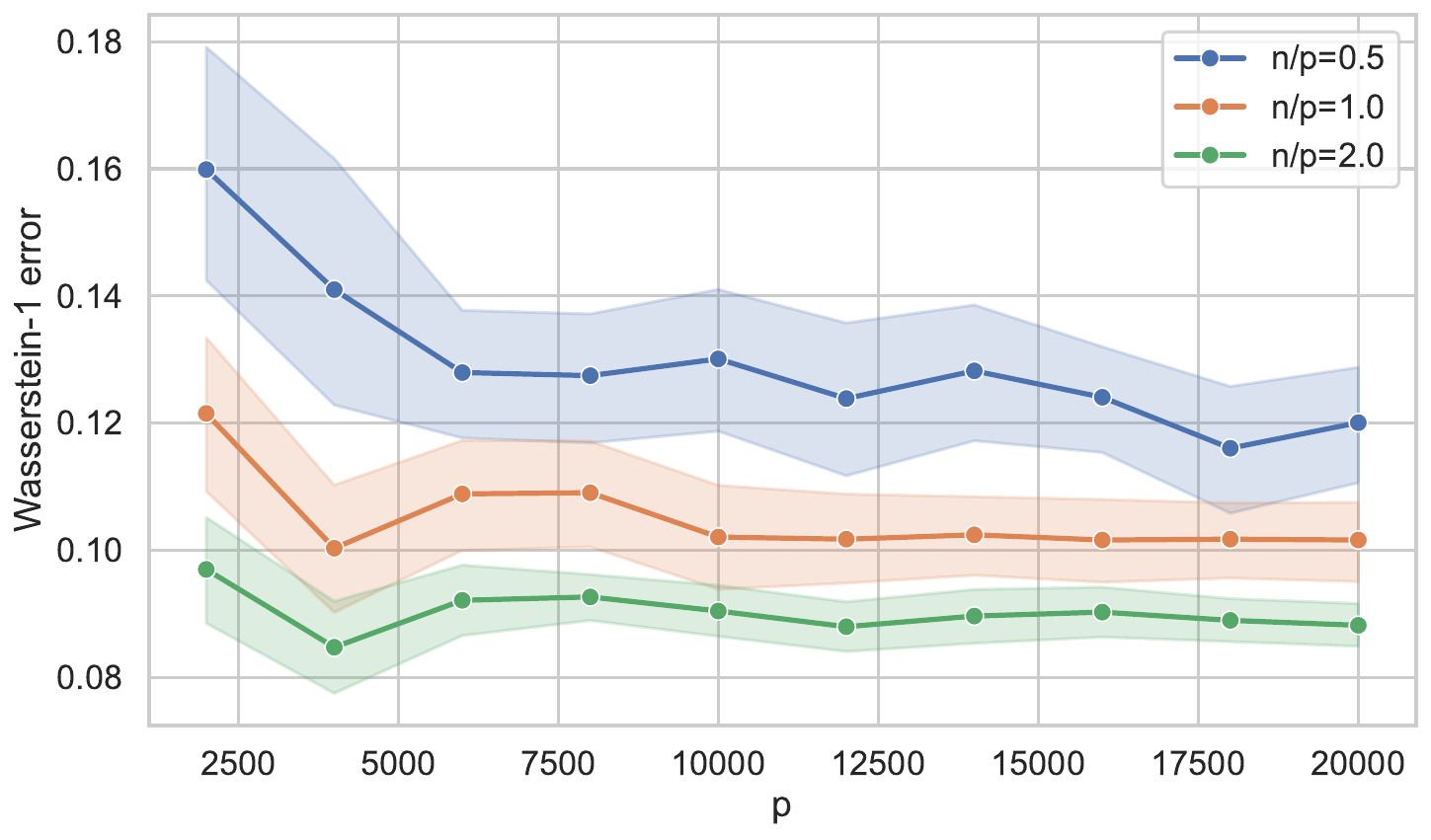}
         \caption{Correlated block design with $\rho = 0.9$}
         \label{fig:sparse-dmm_rho0.2}
     \end{subfigure}
        \caption{Sparse prior. Method: DMM.}
        \label{fig:sparse-dmm}
\end{figure}

\begin{figure}[ht]
     \centering
     \begin{subfigure}[b]{0.45\textwidth}
         \centering
         \includegraphics[width=\textwidth]{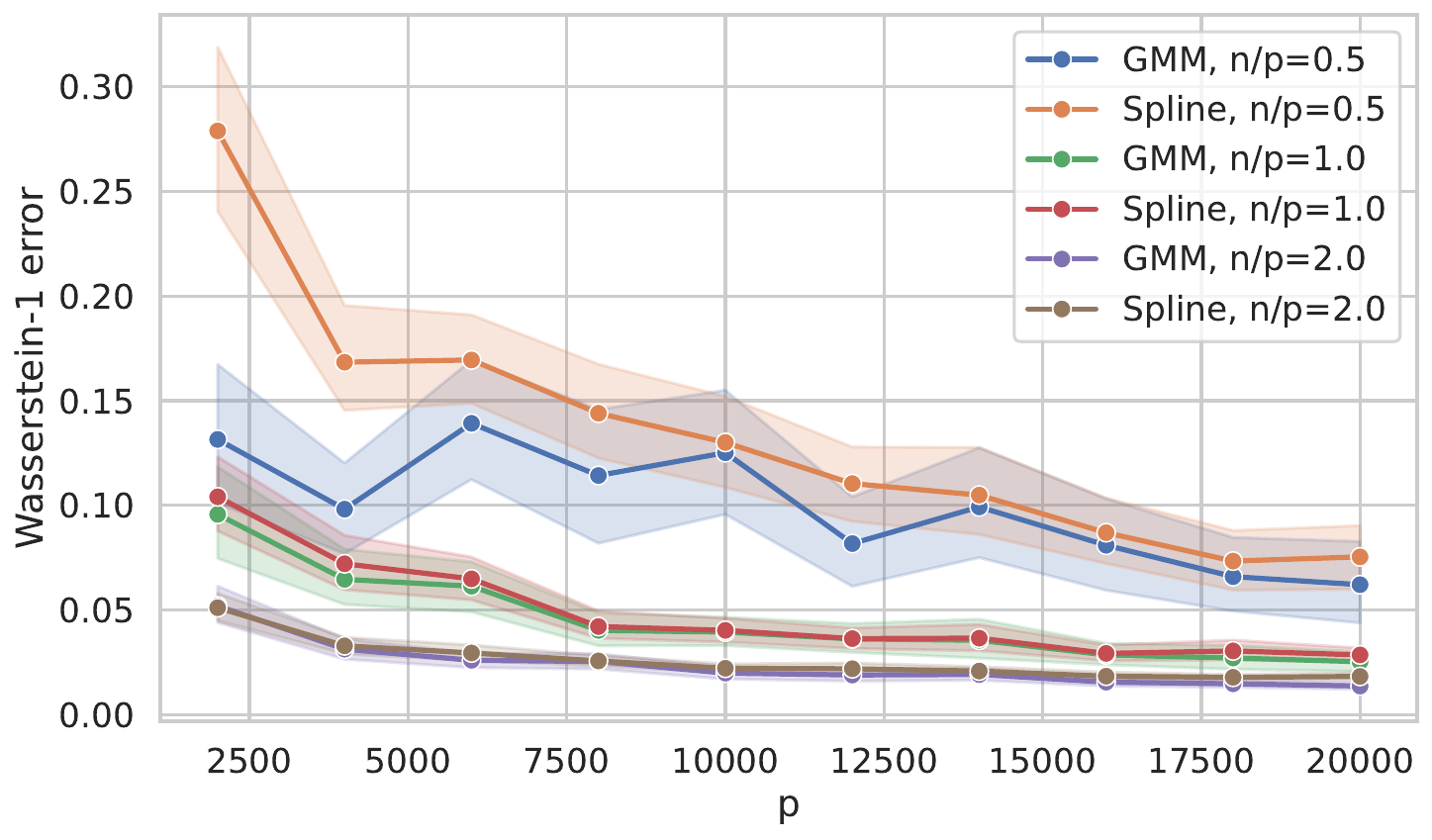}
         \caption{i.i.d.~Design }
         \label{fig:gaussian-dmm_rho0.0}
     \end{subfigure}
     \hfill
     \begin{subfigure}[b]{0.45\textwidth}
         \centering
         \includegraphics[width=\textwidth]{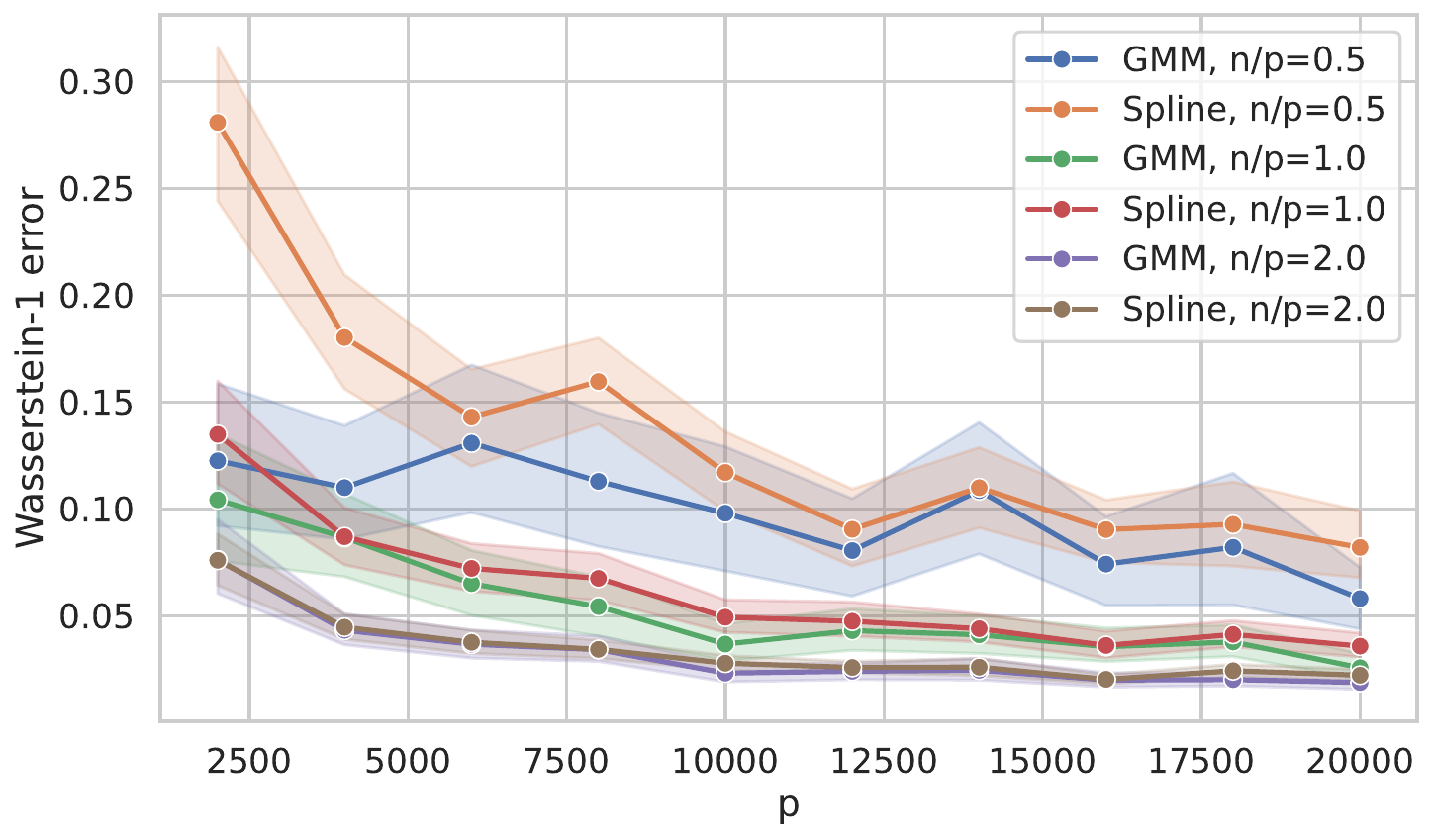}
         \caption{Correlated block design with $\rho = 0.9$}
         \label{fig:gaussian-dmm_rho0.2}
     \end{subfigure}
        \caption{Gaussian prior. Method: GMM and spline.}
        \label{fig:gaussian}
\end{figure}

\begin{figure}[ht]
     \centering
     \begin{subfigure}[b]{0.45\textwidth}
         \centering
         \includegraphics[width=\textwidth]{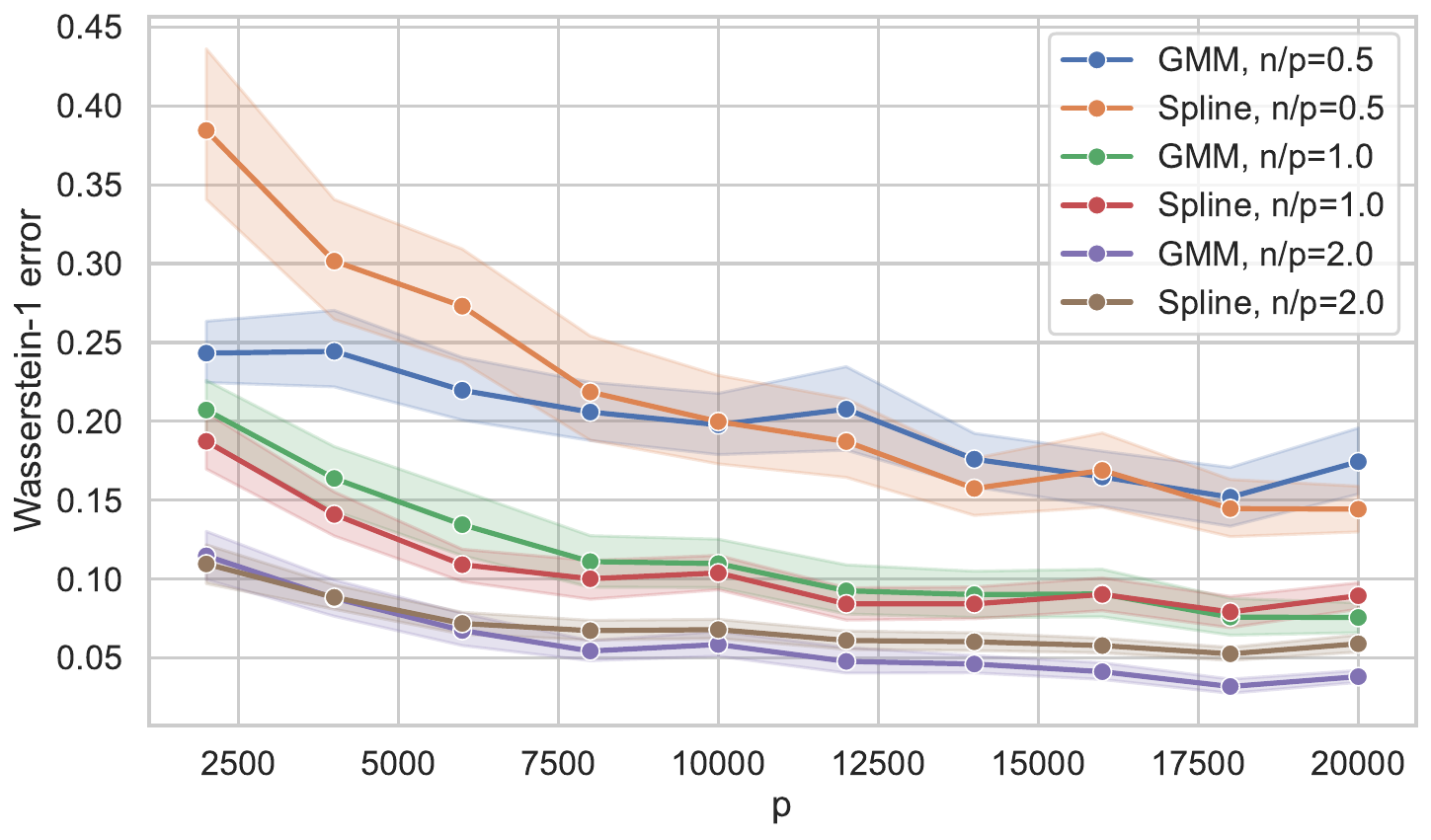}
         \caption{i.i.d.~design}
         \label{fig:bimodal-dmm_rho0.0}
     \end{subfigure}
     \hfill
     \begin{subfigure}[b]{0.45\textwidth}
         \centering         \includegraphics[width=\textwidth]{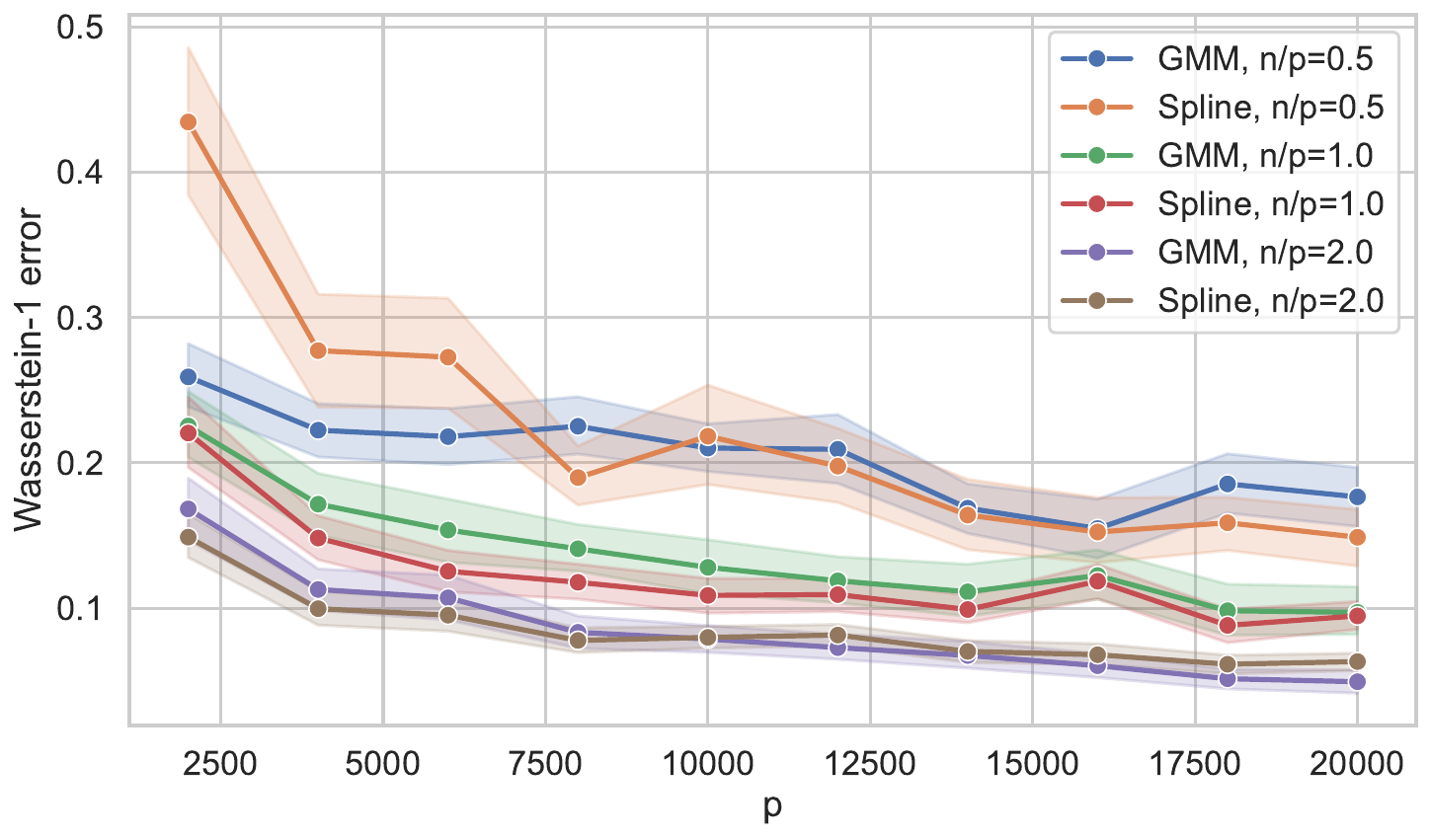}
         \caption{Correlated block design with $\rho = 0.9$}
         \label{fig:bimodal-dmm_rho0.2}
     \end{subfigure}
        \caption{Bimodal prior. Method: GMM and spline.}
        \label{fig:bimodal}
\end{figure}

\paragraph{Experimental setup.}
We evaluate the finite-sample performance of the proposed EBMoM procedure under both independent and correlated Gaussian designs. The independent design $\X_{\tiid}$ has i.i.d.~$\N(0,1/p)$ entries. 
Following \cite{fan2023gradient}, we consider the following  correlated block design: 
% $\X_{\block,\rho}$ is defined by
\begin{align*}
\X_{\block,\rho}=[\X_1,\X_2], \qquad \X_2=\rho \X_1+\sqrt{1-\rho^2}\,\X',
\end{align*}
with $\rho = 0.9$, where $\X_1,\X'\in\R^{n\times (p/2)}$ are independent matrices
with i.i.d.\ $\mathcal{N}(0,1/p)$ entries. Thus, corresponding columns
in the two blocks have correlation $\rho = 0.9$, while each entry retains
marginal variance $1/p$. %We use $\rho=0.2$ and $\rho = 0.9$ for the correlated block design. 
Throughout the experiments, we set the noise
level to $\sigma=1$.

We consider the following four types of prior distributions:
\begin{itemize}
\item Rademacher: $g = \frac{1}{2}\delta_{-1} + \frac{1}{2}\delta_1$;
\item Sparse: $g = 0.9 \delta_0 + 0.1 \delta_1$;
\item Gaussian: $g = \N(0,1)$;
\item Bimodal: $g = \frac{1}{2}\N(-1, 0.25) + \frac{1}{2}\N(1, 0.25)$.
\end{itemize}

We first apply EBMoM to estimate the prior moments, which are then converted to estimates for the prior. Specifically, for the Rademacher and sparse (both binary), we estimate the first three moments and then apply the  DMM method in (\ref{eq:DMM}) to denoise the noisy estimates and obtain a binary distribution $\hat g$. For the continuous Gaussian and bimodal priors, we fit  a Gaussian mixture model (GMM) using Lindsay's algorithm \cite{lindsay1989moment} (see also \cite[Algorithm 3]{wu2020optimal}), which fits a 2-compoment homoskedastic Gaussian mixture based on matching the first 4 moments. 
In addition, we also consider a more nonparametric spline-based fit, which solves the best moment fit penalized by a spline penalty (see Appendix \ref{sec:implementation} for the details).
% For the prior estimate, we report the Wasserstein-1 error averaged over \red{X} trials. 

% \paragraph{Moment estimates}
% We first evaluate the quality of third- and fourth-moment estimation by EBMoM, with the histograms shown in Figures~\ref{fig:m3}
% and~\ref{fig:m4}. In these experiments, we take $n=p=3000$ and true prior $g=\mathcal{N}(0,1)$, where the true moments are $m_3 = 0$ and $m_4 = 3$. We consider both the i.i.d.~design and correlated block design with $\rho = 0.9$, and conduct $T=300$ trials.

% \begin{table}[h]
% \centering
% \label{tab:moment-mean-std}
% \begin{tabular}{lcccc}
% \toprule
% & \multicolumn{2}{c}{$\hat m_3$ (truth $=0$)}
% & \multicolumn{2}{c}{$\hat m_4$ (truth $=3$)} \\
% \cmidrule(lr){2-3}\cmidrule(lr){4-5}
% Design & Mean & Std & Mean & Std \\
% \midrule
% i.i.d.~& $0.004$ & $0.298$ & $2.942$ & $0.891$ \\
% Correlated ($\rho=0.9$)          & $-0.026$ & $0.340$ & $2.969$ & $1.200$ \\
% \bottomrule
% \end{tabular}
% \caption{Mean and standard deviation of $\hat m_3, \hat m_4$.}
% \end{table}

\paragraph{Prior estimate}

In Figures~\ref{fig:rademacher-dmm}--\ref{fig:bimodal}, we fix three aspect ratios $n/p \in \{0.5,1,2\}$, and the dimension $p$ ranges from $2000$ to $20000$. The Wasserstein-1 error is averaged over $T= 50$ trials. As seen in all four figures, the correlated block design brings more variability to the EBMoM estimates as expected %\zf{show results for $\rho=0.9$ instead of 0.2? Is there a bigger difference for 0.9?}, 
%\nbwu{I have added 4 figures for 0.9. Please update. My impression is they appear similar.}
and the estimation error decreases as $n,p$ grows at a proportional speed. This validates the main finding of Theorem \ref{thm:main_moment} that a sub-linear sample complexity is sufficient for consistent estimation of the prior. 

\paragraph{Comparison with EBflow.} We next compare the EBMoM estimates with ones obtained via a maximum marginal likelihood approach. Specifically, we compare EBMoM estimates using both GMM and nonparametric splines for the Gaussian and bimodal priors with the nonparametric EBflow method of \cite{fan2023gradient}, and the discrete EBMoM fits using DMM for the Rademacher and sparse priors with the parametric EBflow method; we describe details of the EBflow procedure in Appendix \ref{sec:EBflow}.

As the marginal log-likelihood optimization landscape is in general non-convex, and results of EBflow depend on algorithm initialization, we test EBflow initialized both with an uninformative guess for the prior and with the prior estimate of EBMoM. The latter corresponds to using EBflow to perform a likelihood-based refinement of the method-of-moments estimate, rather than to perform ab initio estimation.

Figures~\ref{fig:compare_ebflow_rho0}
and~\ref{fig:compare_ebflow_rho9} compare the Wasserstein-1 errors of EBMoM and EBflow, the latter using both uninformative and EBMoM initializations, for the i.i.d.~design and correlated block design with $\rho = 0.9$. We observe that EBMoM is competitive with EBflow, and in fact yields lower estimation error than EBflow with an uninformative initialization in many settings. This error is notably lower for the sparse and bimodal priors, where we hypothesize that the dynamics of EBflow often become trapped in a sub-optimal local minimum of the optimization landscape. In these settings, we observe that the estimation error of EBflow is substantially improved when initialized with the EBMoM estimate; furthermore, in most (but not all) tested settings, EBflow refinement yielded a further reduction of error over the EBMoM estimate.

% For the Gaussian prior, the errors of the GMM and spline
% reconstructions decrease steadily with $p$, and both are below the
% error of EBflow over the displayed range. For the bimodal prior, the
% GMM reconstruction gives the best overall performance, while the
% spline method also improves with dimension. In contrast, EBflow
% exhibits a larger residual error, especially under the strongly
% correlated design.

% The difference is more pronounced for the discrete priors. For both
% the Rademacher and sparse priors, the DMM error decreases as $p$
% increases, whereas the EBflow error remains nearly constant over the
% displayed dimensions. In particular, the gap is substantial for the
% Rademacher prior. These results indicate that the EBMoM--DMM pipeline
% is particularly well adapted to recovering atomic distributions and
% that its advantage persists in the presence of strong column
% correlation.

Overall, this comparison shows that EBMoM can perform favorably in comparison to ab initio estimation using maximum marginal-likelihood methods, with lower computational cost and less parameter tuning, and can also serve as a warm start for likelihood-based refinement.

\begin{figure}[!p]
\centering
\begin{subfigure}[t]{0.78\textwidth}
    \centering
    \includegraphics[width=\textwidth,trim={0 3.5cm 0 0}]{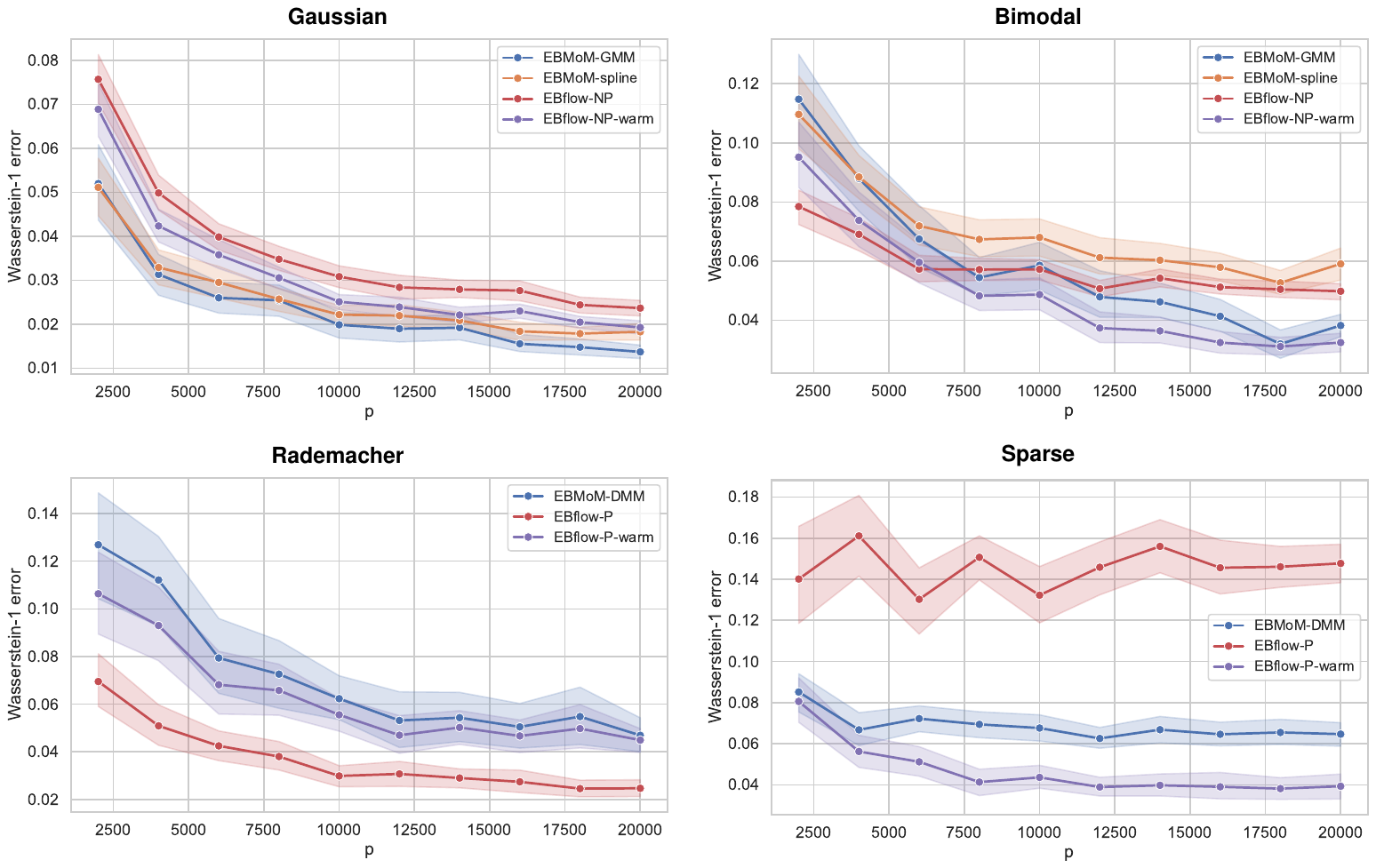}
    \caption{i.i.d.\ Gaussian design.}
    \label{fig:compare_ebflow_rho0}
\end{subfigure}

\medskip
\begin{subfigure}[t]{0.78\textwidth}
    \centering
    \includegraphics[width=\textwidth,trim={0 3.5cm 0 0}]{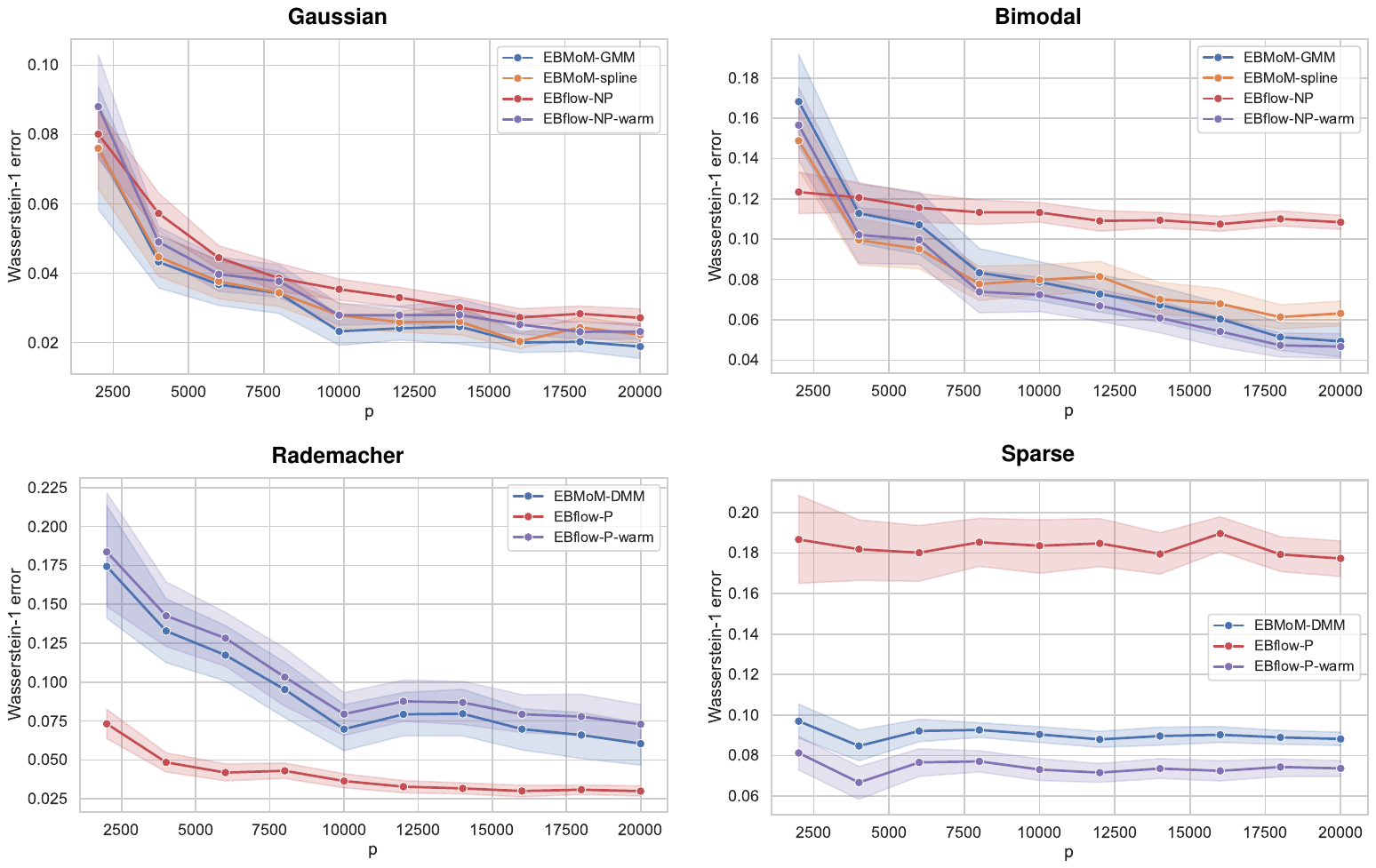}
    \caption{Correlated block design with $\rho=0.9$.}
    \label{fig:compare_ebflow_rho9}
\end{subfigure}

\caption{Comparison with EB flow under i.i.d.\ Gaussian and correlated block designs.}
\label{fig:compare_ebflow}
\end{figure}

% \begin{figure}[ht]
% \centering
% \includegraphics[width=0.9\linewidth]{figs/combine_ebflowall_rho0.0.pdf}   
% \caption{Comparison with EB flow for i.i.d. Gaussian design}
% \label{fig:compare_ebflow_rho0}
% \end{figure}

% \begin{figure}[ht]
% \centering
% \includegraphics[width=0.8\linewidth]{figs/combine_ebflowall_rho0.9.pdf}   
% \caption{Comparison with EB flow for correlated block design with $\rho = 0.9$}
% \label{fig:compare_ebflow_rho9}
% \end{figure}

\paragraph{Posterior mean estimation.}
We finally examine whether the EBMoM prior estimate is accurate enough for
downstream coefficient estimation. As a benchmark for posterior mean estimation, we run classical approximate message
passing (AMP) with the scalar Bayes denoiser %\zf{meaning --- AMP to do posterior mean estimation?} 
associated with either the true
prior or the EBMoM estimate \cite{bayati2011dynamics,barbier2020mutual}.  For
this experiment we use the canonical AMP normalization
$X_{ij}\stackrel{\mathrm{i.i.d.}}{\sim}\N(0,1/n)$, fix $n/p=0.5$ and
$\sigma=1$.
 %and take $p\in\{1000,2000,3000,4000,5000\}$.  
The correlated
design uses the block construction above, rescaled to marginal entry variance
$1/n$, with $\rho=0.9$. 
The moment-to-prior methodology is the same as before: 
For the Rademacher and sparse priors, the AMP denoiser uses the
DMM reconstruction from the first three centered moments; for the Gaussian
and bimodal priors, it uses a two-component Gaussian mixture with common
within-component variance, fitted using the first four moments. Both oracle-prior and EBMoM-prior AMP converged in all 300
simulated instances for each design: no damping was needed for the i.i.d.
design, while damping $0.7$ stabilized every run for the correlated block design with $\rho=0.9$.\footnote{We note that formal AMP theory \cite{bayati2011dynamics,barbier2020mutual} does not directly extend to correlated block designs, and thus we treat the correlated-design results as an empirical comparison without recognizing AMP as a certified posterior mean estimator.}

\prettyref{fig:amp-ebmom-np0p5} reports the normalized coefficient MSE
$p^{-1}\|\widehat{\bbeta}-\bbeta\|_2^2$.  In addition to oracle-prior and
EBMoM-prior AMP, we include ridge regression with its penalty selected by
generalized cross-validation and Lasso with three-fold
cross-validation.  For the Gaussian prior we
also report the exact posterior mean
$(\X^\top\X+\I_p)^{-1}\X^\top\y$, which corresponds to the ridge estimator with penalty level equal to $1$.  Under the i.i.d.~design, AMP with the
EBMoM prior performs better than ridge for the sparse and Rademacher prior, and roughly the same as ridge for the bimodal prior. For the Gaussian prior, the exact posterior
mean gives the smallest risk, while GCV-tuned ridge remains close to it and
retains a small advantage over EBMoM-prior AMP, as expected. The same
qualitative ordering persists under the strongly
correlated design.    Thus, in
this experiment, estimating the prior by EBMoM preserves most of the
estimation gain available from prior-aware AMP, empirically even
outside the standard i.i.d.~design setting.

\begin{figure}[!p]
    \centering
    \begin{subfigure}[t]{0.78\textwidth}
        \centering
        \includegraphics[width=\textwidth]{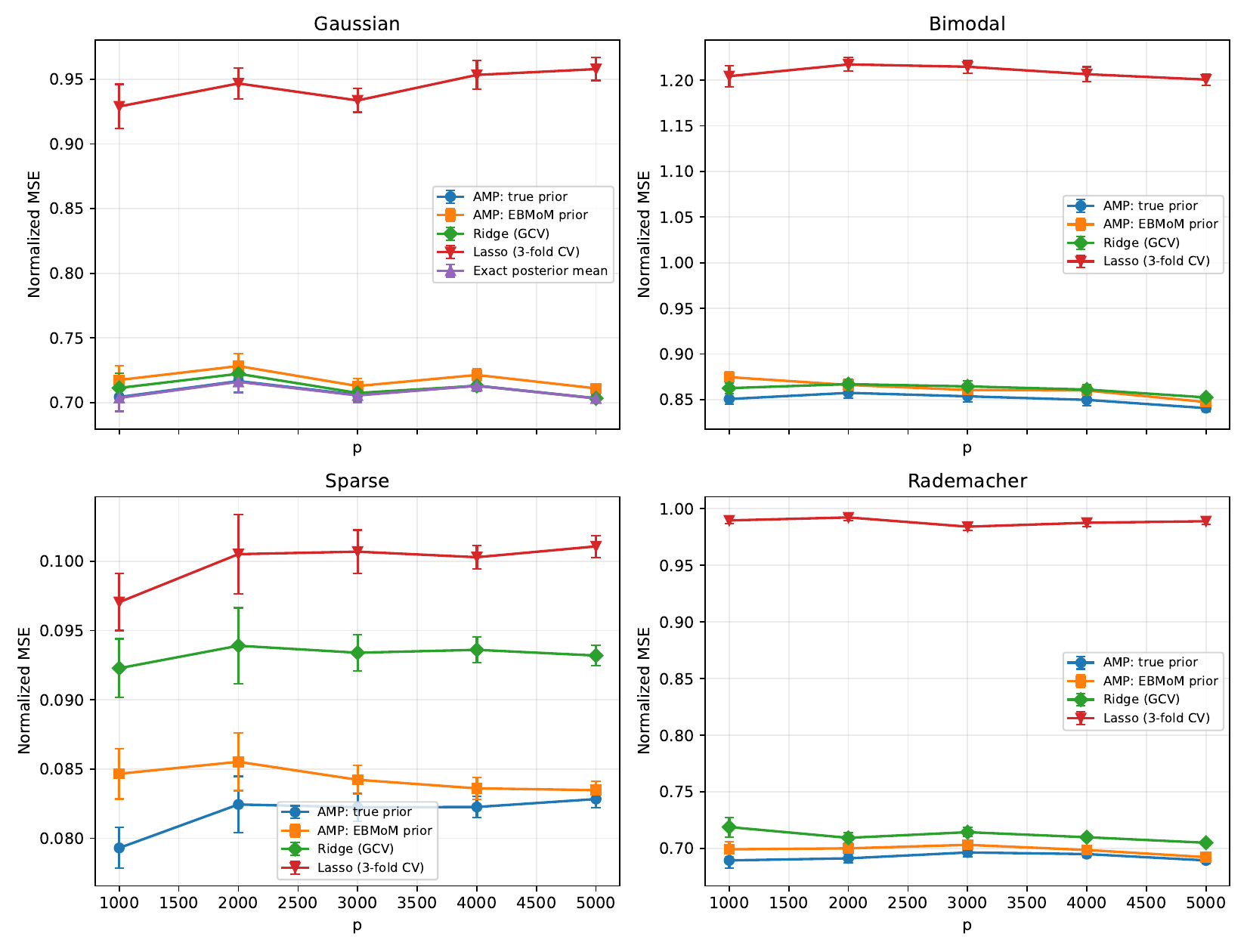}
        \caption{i.i.d.~Gaussian design.}
        \label{fig:amp-ebmom-np0p5-rho0}
    \end{subfigure}

    \medskip
    \begin{subfigure}[t]{0.78\textwidth}
        \centering
        \includegraphics[width=\textwidth]{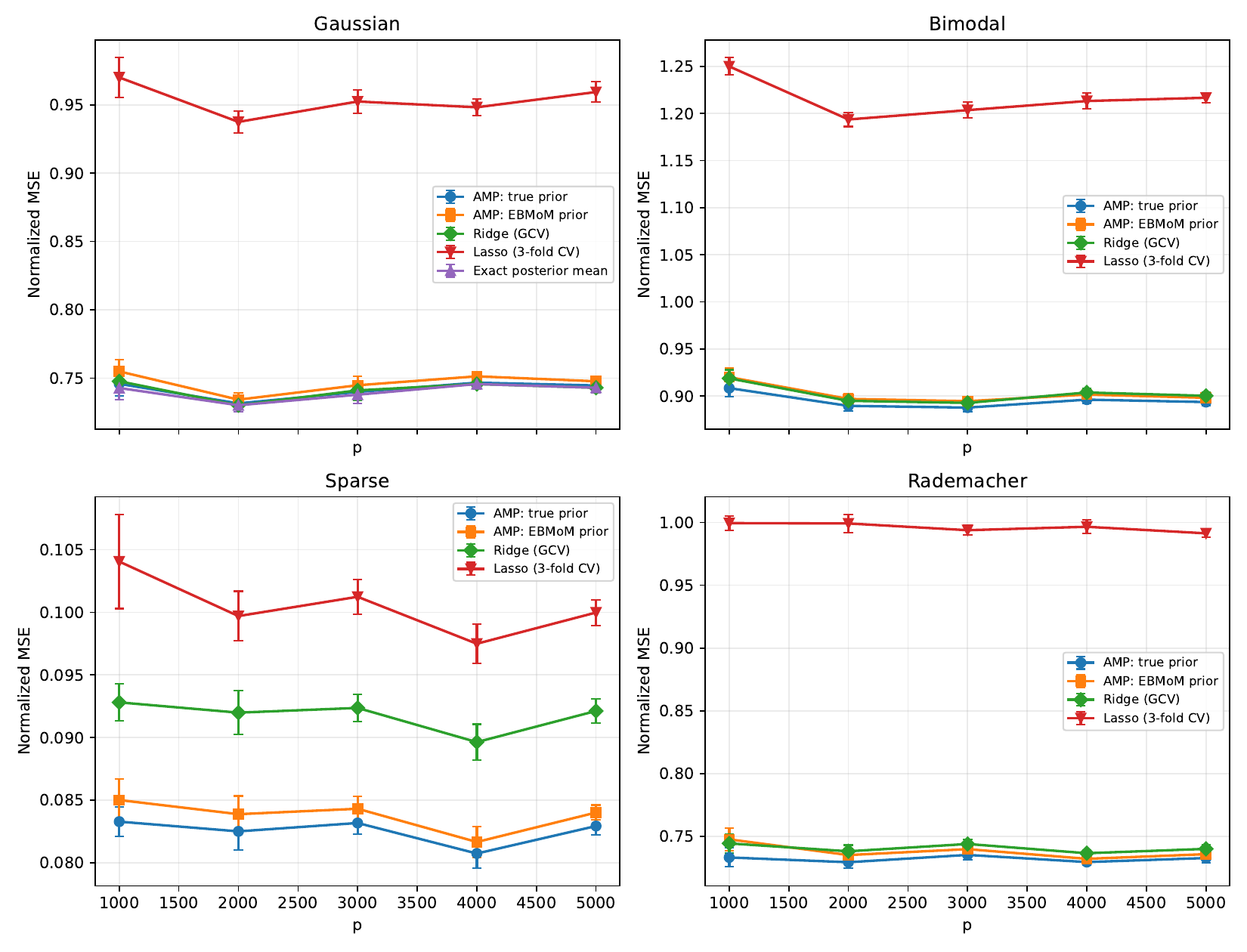}
        \caption{Correlated block design with $\rho=0.9$.}
        \label{fig:amp-ebmom-np0p5-rho0p9}
    \end{subfigure}
    \caption{Normalized coefficient MSE for posterior mean estimation at
    $n/p=0.5$ and $\sigma=1$.  Points and error bars show the mean and standard error over 15 trials.}
    \label{fig:amp-ebmom-np0p5}
\end{figure}

%\begin{remark}
% We show that
% \begin{align*}
% 0 \leq \E\pnorm{\bbeta - \E[\bbeta]}{}^2 - \E\pnorm{\bbeta - \E[\bbeta|\y]}{}^2 \leq C\pnorm{\X}{F}^2, 
% \end{align*}
% where $\E\pnorm{\bbeta - \E[\bbeta]}{}^2 \asymp p$ is the naive MSE. Consequently, if $\pnorm{\X}{F}^2 = o(p)$, the Bayes estimator essentially has no improvement over the naive mean estimator, and this happens, e.g., for iid Gaussian design $\N(0,1/p)$ with high probability when $n = o(p)$. To this, note that
% \begin{align*}
% \E\pnorm{\bbeta - \E[\bbeta]}{}^2 - \E\pnorm{\bbeta - \E[\bbeta|\y]}{}^2 = \E\pnorm{\E[\bbeta] - \E[\bbeta\y]}{}^2 = \sum_{j=1}^p \E[(\E[\beta_j] - \E[\beta_j|\y])]^2 \leq C\sum_{j=1}^p I(\beta_j;\y), 
% \end{align*}
% where the last step follows from Tao's inequality. Let $\tilde\y = \x_j \beta_j + \beps$, then by data processing inequality and rotational invariance of the Gaussian distribution,
% \begin{align*}
% I(\beta_j;\y) \leq I(\beta_j;\tilde\y + \beps) = I(\beta_j; \pnorm{\x_j}{}\beta + N(0,1)).
%  \end{align*}
% It is well-known that in the AWGN channel $Y = \sqrt{\gamma}X + Z$ with $X \perp Z$ and $Z\sim \N(0,1)$, $I(X;Y) \leq \frac{1}{2}\gamma \Var(X)$, hence we conclude that $I(\beta_j;\y) \leq C\pnorm{\x_j}{}^2$, which further implies the desired claim. 
% \end{remark}

\section{Discussion}\label{sec:discussion}
In this paper, we propose a computationally efficient method that iteratively estimates the moments of the prior distribution for the Bayesian linear regression model. For a broad class of design matrices, we show that these moment estimates can be combined into an estimator of the prior that is consistent given $n = p^{1-o(1)}$ samples for any $o(1)$ factor as $n\wedge p\rightarrow\infty$, establishing that consistent prior estimation is possible in the sub-linear regime $n \ll p$. We complement this with a matching information-theoretic lower bound: for a broad class of design matrices and any fixed $\eps > 0$, there exist two distinct priors whose induced observations are indistinguishable from $n = p^{1-\eps}$ samples.

We conclude with some further discussion on assumptions and open questions.

\paragraph{Unknown noise variance}
The current EBMoM estimator (\ref{eq:F_k}) assumes that the noise variance $\sigma^2$ is given. On the one hand, certain assumptions are needed to estimate $\sigma^2$ because 
if $\X\X^\top \propto \I_n$, then $\sigma^2$ is not identifiable when the prior is Gaussian. 
On the other hand, under mild assumptions, $\sigma^2$ can be estimated at a much lower sample complexity than that for estimating higher moments or the prior itself. More precisely, similar to Rao's MIVQUE, we may choose two quadratic forms $\y^\top \Q\y$ with $\Q=\I_n$ and $\Q=\X\X^\top$ and solve for the prior and noise variances, resulting in the unbiased estimator
\[\hat{\sigma}^2 = \frac{t_2 \|\y\|^2- t_1\|\X^\top\y\|^2}{nt_2^2-t_1^2},    \quad t_1\equiv\Fnorm{\X}^2, \quad t_2\equiv\Fnorm{\X\X^\top}^2\]
Under the assumption of \prettyref{prop:MSE_first_two} and the extra condition that the spectrum of $\X$ is not too spiky, namely
$\Fnorm{
\X\X^\top
-
\frac{\Tr(\X\X^\top)}{n}I_n
}^2
\gg \opnorm{\X}^4$,
similar calculation shows that this $\hat \sigma^2$ is consistent when $n\gg \sqrt{p}$.
Note that, for isotropic Gaussian design $X_{ij} \iiddistr \N(0,1/p)$, using
$t_1 \approx n$ and 
$t_2 \approx n(p+n+1)/p$ results in the simplified version:
\[\hat{\sigma}^2 = \frac{(p+n+1)\|\y\|^2- p\|\X^\top\y\|^2}{n(n+1)},\]
which was proposed and shown to be consistent whenever $n\gg \sqrt{p}$ in \cite{Dicker2014Variance}. As shown by Theorems \ref{thm:main_moment} and \ref{thm:lower_bound}, this sample complexity is much lower than what is required by EBMoM for estimating the full prior. 

We note that the EB regression model assumes that $\bbeta$ has i.i.d.\ coordinates. If $\bbeta$ is treated as a fixed vector, a substantially stronger condition on the design and the sample complexity is needed for estimating $\|\bbeta\|^2$ and $\sigma^2$. For example, even for random design $\X$ where the rows are drawn independently from a well-conditioned $\bSigma$, unless $\bSigma$ is known, consistent estimation of $\sigma^2$ requires nearly linear sample complexity $n=p^{1-o(1)}$ much higher than $\sqrt{p}$ \cite{Verzelen2018Adaptive,kong2018estimating}.

\paragraph{Sample complexity of estimating parametric priors} In this paper, we have focused on the nonparametric class of priors where $g$ is an arbitrary subgaussian distribution, and the optimal sample complexity for consistently estimating $g$ is $n=p^{1-o(1)}$. For parametric priors, the sample complexity is conceivably much lower but this remains an open question.

Specifically, suppose that $g$ is identified within the parametric family by its first $k$ moments, and, for simplicity, the design $\X$ has i.i.d.~$\N(0,1/p)$ entries. Our result shows that for $k\geq 3$, the EBMoM method requires $n \gg p^{1-\frac{1}{1+2k^2}}$ for consistency, while our current lower bound argument (for nonparametric priors) can be adapted to yield a lower bound of $n \gg p^{1-\frac{2}{k}}$. We conjecture that neither is tight for fixed $k \geq 3$ %\nbwu{Confirm? $k=1,2$ are resolved and the conjecture is not true per Proposition 1?}, 
and the optimal sample complexity is $n \gg p^{1-\frac{1}{k}}$, which is the threshold above which the signal strength parameter $A_k$ in Assumption \ref{assump:design} diverges. 
In other words, for discrete prior $g$ with $m$ atoms (which is identified by $k=2m-1$ moments), the conjectured sample complexity is $n \gg p^{\frac{2m-2}{2m-2}}$.
In the special case of $k=3$ and i.i.d.\ Gaussian design, we present in Appendix \ref{app:m3} a more refined analysis of EBMoM that achieves the conjectured $p^{2/3}$ threshold, but the matching lower bound and more general results are still open.

\paragraph{Regret in EB regression}
A central quantity in EB and compound decision theory is the \textit{regret}, which measures how close a data-driven estimator approaches the optimal Bayes estimator (posterior mean) in terms of risk. For a linear model with a fixed design $\X \in \R^{n\times p}$ and a class of prior distributions $\calG$, the EB regret is defined by 
\begin{equation}
    \mathsf{Regret}(\X,\calG):=
    \inf_{\hat\bbeta}\sup_{g \in \calG}\Big\{
    \Expect[\|\bbeta-\hat\bbeta\|^2] - 
    \Expect[\|\bbeta-
    \Expect[\bbeta|\X,\y]\|^2]\Big\}
\end{equation}
where $\bbeta \iiddistr g \in \calG$, the infimum is taken over all estimators $\hat\bbeta = \hat\bbeta(\X,\y)$, and $\E[\bbeta|\X,\y]$ is the (oracle) Bayes estimator. 
Of special interest is to characterize the conditions under which $\mathsf{Regret}(\X,\calG)=o(p)$ holds, so that the average regret of estimating each regression coefficient is amortized in high dimensions. Estimators achieving such sublinear regret are referred to as \emph{asymptotically optimal} in the empirical Bayes literature \cite{robbins1956empirical,zhang2003compound}.

% It is worth emphasizing that this objective is substantially more demanding than consistent estimation of the prior itself. As shown in Sections \ref{sec:method}--\ref{sec:lb}, the prior can be estimated consistently with sublinear sample size $n=p^{1-o(1)}$, whereas achieving sublinear regret requires the linear sample size $n=\Theta(p)$ even for simple iid design (see \prettyref{app:eblb} for a proof). Thus, accurately learning the prior does not by itself guarantee that one can approach the oracle Bayes risk with vanishing regret per coordinate.

% It is worth distinguishing consistent estimation of the prior from nontrivial estimation of the regression coefficients. 
One subtle point worth emphasizing is the distinction between consistent estimation of the prior and nontrivial estimation of the regression coefficients. As shown in Sections \ref{sec:method}--\ref{sec:lb}, the prior can already be estimated consistently with sublinear sample complexity $n=p^{1-o(1)}$. In contrast, it is only meaningful to ask whether an empirical Bayes estimator can achieve the Bayes error up to sublinear regret when the Bayes error itself is nontrivially smaller than the trivial MSE, which in general requires a linear sample size.\footnote{To see this, note that for $g=N(0,1)$, the Bayes squared error of estimating $\bbeta$ equals $p - \sum \frac{s_i^2}{1+s_i^2} \geq p-n$, where $s_i$ is the singular value of $\X$. Thus the variance reduction is trivial unless $n=\Theta(p)$. This can be extended to other priors by applying the mutual information method \cite[Chapter 30]{PW-it}.} This makes the linear-sample regime $n=\Theta(p)$ the natural setting for studying regret in EB regression.

The classical literature on empirical Bayes mostly focuses on the sequence model corresponding to $\X=\I$ and $n=p$. In this special case, the behavior of the optimal regret is relatively well-understood. For example, for the nonparametric class of priors with a compact support or subgaussian tails, this regret is not just sub-linear but in fact \emph{poly-logarithmic} in $p$, attained by either NPMLE \cite{jiang2009general} or kernel methods \cite{li2005convergence} (see  \cite{polyanskiy2021sharp} 
 for lower bound.) 
 
 For linear models, the recent work \cite{fan2023gradient} showed that, 
 under appropriate conditions on the design $\X$ that is satisfied by well-conditioned random design with $n=\Theta(p)$, 
for the class $\calG$ of compactly support priors, 
the NPMLE attains a sublinear regret 
$\mathsf{Regret}(\X,\calG)=o(p)$; 
 see \cite[Corollaries 3.3 and 3.11]{fan2023gradient}.
 However, due to the non-separability of the likelihood in the linear model, the NPMLE is difficult to compute. It remains a challenging problem to determine how the optimal regret scales with $p$ and how to attain it with computationally efficient means.
 
%  It remains a challenging problem to understand the behavior of the optimal regret 
% % $\mathsf{Regret}(\X,\calG)$ 
% even for random designs and $n$ proportional to $p$.

% Finally, we clarify a subtle point: 
% As shown in Sections \ref{sec:method}--\ref{sec:lb}, the prior can be estimated consistently at a sublinear sample complexity of $n=p^{1-o(1)}$; nevertheless, achieving a sublinear regret requires linear sample complexity of $n=\Theta(p)$.

\section{Proofs of upper bounds}\label{sec:proof_ub}

\subsection{Proof of Proposition \ref{prop:MSE_first_two}}\label{sec:proof_first_two}
Simple calculations yield that $\hat m_1$ is unbiased and
\begin{align*}
\Var(\hat m_1) = \frac{\sigma^2 \|\X\ones\|^2 + s^2 \|\X^\top\X\ones\|^2}{\|\X\ones\|^4}
\end{align*}
where $s^2=m_2-m_1^2$ is the variance of the prior. Using $\pnorm{\X^\top \X\bm{1}}{}^2 \leq \pnorm{\X}{\op}^2\pnorm{\X\bm{1}}{}^2$ and $\pnorm{\X}{\op} \geq c\sigma$, we have
\begin{align*}
\Var(\hat m_1) \leq \pnorm{\X}{\op}^2\Big(s^2 + \frac{\sigma^2}{\pnorm{\X}{\op}^2}\Big)\frac{1}{\pnorm{\X\bm{1}}{}^2} \leq C\frac{\pnorm{\X}{\op}^2}{\pnorm{\X\bm{1}}{}^2}, 
\end{align*}
and the conclusion follows from the condition on $\pnorm{\X\bm{1}}{}^2$.

Next we bound $\E[(\hat\mu_2 - \mu_2)^2]$. Using $\bar\mu_2 = F_2(\bar\y) = \frac{\|\X^\top \bar{\y}\|^2-\sigma^2\Fnorm{\X}^2}{\Fnorm{\X^\top \X}^2}$ with $\bar \y = \y - m_1\X\bm{1}$ as an intermediary, we have
\begin{align*}
\E[(\hat \mu_2 - \mu_2)^2] \leq 2\underbrace{\E[(\bar \mu_2 - \mu_2)^2]}_{(I)} + 2\underbrace{\E[(\hat\mu_2 - \bar\mu_2)^2]}_{(II)}. 
\end{align*}
To bound $(I)$, note that $\E[\bar\mu_2] = \mu_2$, and 
    $\|\X^\top \bar\y\|^2 = \|\X^\top \X\bar\bbeta\|^2 + 
    \|\X^\top \beps\|^2 + 2 \iprod{\X^\top\X \bar\bbeta}{\X^\top \beps}
    $. 
    Recalling that $\H=\X^\top\X$, we have    
     $\var(\|\X^\top \bar\y\|^2) = \var(\|\H \bar\bbeta\|^2) + 
    \var(\|\X^\top \beps\|^2) + 4 \var(\iprod{\H \bar\bbeta}{\X^\top \beps})
    $, because odd moments of $\beps$ vanish. It can be verified that
    \[
    \var\Big(\iprod{\H \bar\bbeta}{\X^\top \beps}\Big)
    = \sigma^2 \mu_2  \tr(\H^3), 
    \quad
    \var(\|\X^\top \beps\|^2)
    = 2 \sigma^4 \tr(\H^2).
    \]
    Finally, with $\S \equiv \H^2$,
 \begin{align*}
\Expect[\|\H \bar\bbeta\|^4]
= & \sum_{a,b,c,d=1}^p 
S_{ab} S_{cd} \Expect[\bar\beta_a\bar\beta_b\bar\beta_c\bar\beta_d] \\
= & 
\mu_4 \sum_{a=1}^p S_{aa}^2
+ \mu_2^2  \sum_{a\neq c=1}^p S_{aa}S_{cc}
+ 2 \mu_2^2  \sum_{a \neq b=1}^p S_{ab}^2\\
= & 
(\mu_4 -3\mu_2^2) \underbrace{\sum_{a=1}^p S_{aa}^2}_{\leq \Fnorm{\S}^2 =\tr(\H^4)}
+ \mu_2^2  \underbrace{\sum_{a, c=1}^p S_{aa}S_{cc}}_{\tr(\H^2)^2}
+ 2 \mu_2^2  \underbrace{\sum_{a,b=1}^p S_{ab}^2}_{\tr(\H^4)}.
 \end{align*}
Applying $\Expect[\|\H \bar\bbeta\|^2]=\mu_2 \Fnorm{\H}^2$, we get
\begin{align*}
\Var(\|\H \bar\bbeta\|^2) = (\mu_4 - 3\mu_2^2)\sum_{a=1}^p S_{aa}^2 + 2\mu_2^2 \Tr (\H^4) \leq (\mu_4  + 2\mu_2^2)   \Tr(\H^4).
\end{align*}
Overall, we get 
\begin{equation}\label{eq:m2var}
\begin{aligned}
\var(\bar \mu_2) &= \frac{(\mu_4 -3\mu_2^2) \sum_{a=1}^p S_{aa}^2 + 2 \mu_2^2 \tr(\H^4) + 4  \sigma^2 \mu_2  \tr(\H^3) + 2 \sigma^4 \tr(\H^2) }{\tr(\H^2)^2}     \\
& \leq C\frac{\tr(\H^4) + \sigma^2  \tr(\H^3)
+  \sigma^4 \tr(\H^2)}{\tr(\H^2)^2}\\
&\leq C\frac{\Tr(\H^4) + \sigma^4 \Tr (\H^2)}{\Tr(\H^2)^2},   
\end{aligned}
\end{equation}
where in the last step we apply $\Tr (\H^3)^2 \leq \pnorm{\H}{F}^2\pnorm{\H^2}{F}^2 = \Tr(\H^2)\Tr(\H^4)$. Now applying $\Tr(\H^4) \leq (n\wedge p)\pnorm{\X}{\op}^8$ and $\Tr(\H^2)\leq (n\wedge p)\pnorm{\X}{\op}^4$ along with the lower bound of $\pnorm{\H}{F}^2$ and $\|\X\|_\op/\sigma$ in (\ref{eq:first_two_assumption}), we have
 \begin{align*}
 (I) = \Var (\bar\mu_2) \leq C\frac{\pnorm{\X}{\op}^8\Big(1 + \frac{\sigma^4}{\pnorm{\X}{\op}^4}\Big)(n\wedge p)}{\pnorm{\H}{F}^4} \leq \frac{C}{n\wedge p}.
 \end{align*}
 Next to bound $(II)$, using $\hat\y - \bar\y = -(\hat m_1 - m_1)\X\bm{1}$, we have 
 \begin{align*}
 \hat\mu_2 - \bar\mu_2 &= \frac{\pnorm{\X^\top\hat\y}{}^2 - \pnorm{\X^\top\bar\y}{}^2}{\pnorm{\X^\top \X}{F}^2} = \frac{\langle \X^\top (\hat\y - \bar\y), \X^\top (\hat\y + \bar\y)\rangle}{\pnorm{\X^\top \X}{F}^2}\\
 &= -(\hat m_1 - m_1)\Big[\frac{\langle \X^\top \X\bm{1}, \X^\top (\X(2\bbeta - (\hat m_1 + m_1)\bm{1}) + 2\beps) \rangle}{\pnorm{\X^\top \X}{F}^2}\Big],
 \end{align*}
 which implies that $(II) = \E[(\hat\mu_2 - \bar\mu_2)^2] \leq C((II_1) + (II_2))$, where
 \begin{align*}
 (II_1) &=\frac{1}{\pnorm{\X^\top \X}{F}^4} \E\Big[(\hat m_1 - m_1)^2\langle \X^\top \X\bm{1}, \X^\top \beps\rangle^2\Big],\\
 (II_2) &= \frac{1}{\pnorm{\X^\top \X}{F}^4} \E\Big[(\hat m_1 - m_1)^2\langle \X^\top \X\bm{1}, \X^\top \X(2\bbeta - (\hat m_1 + m_1)\bm{1})\rangle^2\Big].
 \end{align*}
 To bound $(II_1)$, we use $\hat m_1 - m_1 = \frac{\bm{1}^\top \X^\top \bar\y}{\pnorm{\X\bm{1}}{}^2}$ to get
 \begin{align*}
 (II_1) &= \frac{1}{\pnorm{\X^\top \X}{F}^4\pnorm{\X \bm{1}}{}^4}\E\Big[(\bm{1}^\top \X\bar\y)^2 \langle\X^\top \X\bm{1}, \X^\top \beps\rangle^2\Big]\\
 &\leq \frac{1}{\pnorm{\X^\top \X}{F}^4\pnorm{\X \bm{1}}{}^4} \E^{1/2}[(\bm{1}^\top \X\bar\y)^4]\E^{1/2}\Big[\langle\X^\top \X\bm{1}, \X^\top \beps\rangle^4\Big].
 \end{align*}
 Since $\bm{1}^\top \X^\top \bar\y = \bm{1}^\top \X^\top\X \bar\bbeta + \bm{1}^\top \X^\top \beps$, where $\bm{1}^\top \X^\top\X \bar\bbeta$ is subgaussian with constant $2M\pnorm{\X^\top \X\bm{1}}{}$ and $\bm{1}^\top \X^\top \beps \sim \N(0, \sigma^2\pnorm{\X\bm{1}}{}^2)$, we have $\E[(\bm{1}^\top \X\bar\y)^4] \leq C(M^4 \pnorm{\X^\top \X\bm{1}}{}^4 + \sigma^4\pnorm{\X\bm{1}}{}^4)$. Similarly, $\langle \X^\top \X\bm{1}, \X^\top \beps\rangle \sim \N(0,\sigma^2\pnorm{\X\X^\top \X\bm{1}}{}^2)$, so $\E[\langle\X^\top \X\bm{1}, \X^\top \beps\rangle^4] \leq C\sigma^4\pnorm{\X\X^\top \X\bm{1}}{}^4$. Combining the estimates yields 
\begin{align*}
(II_1) &\leq C\frac{M^2\sigma^2\pnorm{\X^\top \X\bm{1}}{}^2\pnorm{\X\X^\top\X\bm{1}}{}^2 + \sigma^4\pnorm{\X\bm{1}}{}^2\pnorm{\X\X^\top\X\bm{1}}{}^2}{\pnorm{\X^\top \X}{F}^4\pnorm{\X\bm{1}}{}^4}\\
&\leq C\frac{\sigma^2\pnorm{\X}{\op}^6 + \sigma^4\pnorm{\X}{\op}^4}{\pnorm{\X^\top\X}{F}^4} \leq \frac{C}{(n\wedge p)^2},
\end{align*}
where we apply the lower bound of $\pnorm{\X^\top \X}{F}^2$ and $\|\X\|_\op/\sigma$ in (\ref{eq:first_two_assumption}). A similar argument yields the same bound for $(II_2)$, concluding that $(II) \leq \frac{C}{(n\wedge p)^2}$. Combining the bounds of $(I)$ and $(II)$ concludes the proof.

\subsection{Proof of Theorem \ref{thm:main_moment}}\label{subsec:proof_upper_main}
The proof strategy consists of the following two steps:
\begin{enumerate}[(i)]
\item We consider a centered model $ \overline{\y} = \X(\bbeta - m_1 \bm{1}) + \beps$ assuming the true mean $m_1$ is known. Based on observations $(\X,\bar{\y})$, we apply the 
recipe in \eqref{eq:moment_def} to estimate the centered moments $\mu_k := \E[(\beta_1 - m_1)^k]$. More precisely, let $\{\bar \mu_\ell\}_{\ell\geq 2}$ be recursively defined by
\begin{align}\label{eq:centered_moment_ideal}
\bar\mu_\ell := \frac{1}{A_\ell}\Big[F_\ell(\overline{\y}) - \sum_{t=2}^\ell \sum_{d_1 + \ldots + d_t = \ell: d_j \geq 2} \bar\mu_{d_1} \ldots \bar\mu_{d_t} \sum_{s_1,\ldots,s_t = 1}^p \tilde{T}^{(d_1,\ldots,d_t)}_{s_1,\ldots,s_t}\Big].
\end{align}
The truncated MSE bound of $\bar\mu_\ell$ for $\mu_\ell$ is detailed in Section \ref{subsection:MSE_centered}. Because $\bar{\y}$ uses the true mean for centering, $\bar\mu_\ell$ is 
henceforth referred to as an \textit{oracle estimator} which is only used as a proof device.
% We emphasize that, because $\bar{\y}$ uses the true mean for centering, $\bar\mu_\ell$ is not an actual estimator and is only used as a proof device.
\item Since $\hat{m}_1$ given in  (\ref{eq:m1hat}) is consistent for $m_1$, we expect $\hat\y = \X (\bbeta - \hat{m}_1 \bm{1}) + \beps $ to be a good proxy for $\bar{\y}$, and therefore $\hat\mu_\ell$ in \eqref{eq:moment_def}, namely, 
\begin{align}\label{eq:centered_moment_est}
\hat\mu_\ell := \frac{1}{A_\ell}\Big[F_\ell(\hat\y) - \sum_{t=2}^\ell \sum_{d_1 + \ldots + d_t = \ell: d_j \geq 2} \hat\mu_{d_1} \ldots \hat\mu_{d_t} \sum_{s_1,\ldots,s_t = 1}^p \tilde{T}^{(d_1,\ldots,d_t)}_{s_1,\ldots,s_t}\Big],
\end{align}
to be close to the oracle $\bar\mu_\ell$. We quantify this approximation error in Section \ref{subsection:perturbation_noncentered}. When combined with the previous step, this yields the truncated MSE of $\hat\mu_\ell$.
% $\{\hat\mu_\ell\}_{\ell\geq 2}$ for $\{\mu_\ell\}_{\ell\geq 2}$.
%\item Since $m_\ell = \sum_{k=0}^\ell {\ell\choose k}\mu_km_1^{\ell-k}$ by the binomial formula, we can bound the MSE of $\hat m_\ell$ in (\ref{eq:moment_def}) via the above MSE of $\hat m_1$ and $\{\hat\mu_\ell\}_{\ell\geq 2}$.
\end{enumerate}
To simplify notations, 
% we use $\mu_\ell = \E[(\beta_1 - m_1)^\ell]$ to denote the centered moments. With $\hat{m}_1$ given by (\ref{eq:centered_moment_est}), 
let 
\begin{align}
\hat{\bbeta} := \bbeta - \hat{m}_1 \bm{1}_p, \qquad \overline{\bbeta} := \bbeta - m_1 \bm{1}_p
\end{align}
where $\hat{m}_1$ is given in (\ref{eq:m1hat}). In the following, unless specified otherwise $C > 0$ will denote a universal constant whose value will change from place to place.

%Note that both $ \hat{\mu}_\ell $ and $ \overline{m}_\ell $ are estimators for the centered moment $ \mu_\ell$, not the original moment $m_\ell$.

\subsubsection{MSE for centered model}\label{subsection:MSE_centered}

The following result bounds the MSE of the oracle estimator $\bar\mu_k$ in (\ref{eq:centered_moment_ideal}) for the centered moments $\mu_k$ truncated on a high probability event. For each $k\geq 2$, define the event
\begin{align}\label{def:event_Ek}
\cE_k := \{|\bar\mu_\ell - \mu_\ell| \leq K_\ell, \ell = 2,\ldots,k\}, \quad \text{ where } K_\ell := M^\ell \ell^{\ell/2}.
\end{align}
For convenience, let $\cE_1$ denote the entire sample space.
% (so that $\cE_1$ always holds).

\begin{proposition}\label{prop:MSE_centered}
Suppose that Assumption \ref{assump:design} holds for some $c_0 > 0$ up to some integer $k_0\geq 2$, and Assumption \ref{assump:prior} holds for some $M > 0$. Then for any $2\leq k\leq k_0$, it holds that 
\begin{align*}
\E[(\bar\mu_k - \mu_k)^2\bm{1}_{\cE_{k-1}}] \leq \Delta_k,
\end{align*}
where, for some constant $C_0  = C_0(M,c_0) > 0$,
\begin{align}\label{eq:MSE_centered}
\Delta_k = (C_0 k)^{5k^2}  \frac{1}{n \wedge p}\Big( \frac{p}{n} \vee 1 \Big)^{2k^2}.
\end{align}
Moreover, event $\cE_k$ holds with probability at least $1-\Delta_k$.
%\nbwu{It should be $\E[(\bar\mu_k - \mu_k)^2\bm{1}_{\cE_{k-1}}]$ on the LHS? Also, missing probability bound of the event?} 
\end{proposition}

We first state a few preliminary results for the proof of Proposition \ref{prop:MSE_centered}. The next two results are useful for controlling the bias.

\begin{lemma}\label{prop:Bias}
For any positive integers $d_1,\ldots,d_t$, let $\gamma^{(d_1,\ldots,d_t)} = \sum_{s_1,\ldots,s_t=1}^p \tilde T_{s_1,\ldots,s_t}^{(d_1,\ldots,d_t)}$, where $\tilde\T^{(d_1,\ldots,d_t)}$ is given by (\ref{eq:T_diag_free}). Then for any $k\geq 2$, with $A_k$ given by (\ref{eq:Ak}), 
\begin{align}
\E[F_k(\bar{\y})] = A_k\mu_k + \sum_{t=2}^k \sum_{d_1+\cdots+d_t = k: d_j \geq 2} \gamma^{(d_1,\ldots,d_t)} \mu_{d_1} \cdots \mu_{d_t}.
\end{align}
\end{lemma}
%We are now ready to prove Proposition \ref{prop:Bias}.
\begin{proof}%[Proof of Proposition \ref{prop:Bias}]
First note that if $\overline{\y} = \btheta + \beps$ with $\beps \sim \N(0, \sigma^2 I_n)$, then for any $\a \in \bS^{n-1}$, $\E_{\beps} H_k(\a^\top \bar{\y}/\sigma) = (\a^\top \btheta/\sigma)^k$. In our case, we have $ \btheta = \X \bar{\bbeta}
$. Moreover, with $\T^{(k)}$ and $\tilde{\T}^{(d_1,\ldots,d_t)}$ defined in (\ref{eq:T_tensor}) and (\ref{eq:T_diag_free}), it can be readily checked that for any vector $\bbeta\in\R^p$,
\begin{align}\label{eq:T_expansion}
\Tr (\X^\top \X\bbeta)^{\otimes k} = \iprod{\T^{(k)}}{\bbeta^{\otimes k}} = \sum_{t=1}^k \sum_{d_1,\ldots,d_t = k: d_j \geq 1} \sum_{s_1,\ldots,s_t = 1}^p \tilde{T}^{(d_1,\ldots,d_t)}_{s_1,\ldots,s_t}\beta^{d_1}_{s_1}\ldots \beta^{d_t}_{s_t},
\end{align}
where the summation over $(d_1,\ldots,d_t)$ is ordered. Applying the above two facts, the first term in the definition (\ref{eq:moment_def}) satisfies
\begin{align}\label{eq:Fk_expectation}
\notag\E[F_k(\bar{\y})] &=\E_{\bbeta}\E_{\beps}\Big[\sum_{j=1}^p (\sigma\pnorm{\x_j}{}^k)H_k\Big(\frac{\x_j^\top \bar{\y}}{\sigma\pnorm{\x_j}{}}\Big)\Big] = \E_{\bbeta}\Big[\sum_{j=1}^p (\x_j^\top \btheta)^k\Big] = \E_{\bbeta}[\Tr (\X^\top \X\bar\bbeta)^{\otimes k}]\\
\notag&= \E_{\bbeta}\bigg[\sum_{t=1}^k \sum_{d_1,\ldots,d_t = k: d_j \geq 1} \sum_{s_1,\ldots,s_t = 1}^p \tilde{T}^{(d_1,\ldots,d_t)}_{s_1,\ldots,s_t} \bar{\beta}^{d_1}_{s_1}\ldots \bar{\beta}^{d_t}_{s_t}\bigg]\\
&\notag= \sum_{t=1}^k \sum_{d_1+\cdots+d_t = k: d_j \geq 2} \gamma^{(d_1,\ldots,d_t)} \mu_{d_1} \cdots \mu_{d_t}\\
&= A_k\mu_k + \sum_{t=2}^k \sum_{d_1+\cdots+d_t = k: d_j \geq 2} \gamma^{(d_1,\ldots,d_t)} \mu_{d_1} \cdots \mu_{d_t},
\end{align}
as desired.
\end{proof}

\begin{lemma}\label{lem:gamma}
Let $\gamma^{(d_1,\ldots,d_t)} = \sum_{s_1,\dots,s_t = 1}^p \tT^{(d_1,\dots,d_t)}_{s_1,\dots,s_t}$. For any positive integers $t,k$, and $ d_1,\dots,d_t $ satisfying $ d_j \geq 2 $ and $ d_1 + \cdots + d_t = k $, it holds that
\begin{equation}
|\gamma^{{(d_1,\dots,d_t)}}| \leq \binom{k}{d_1,\dots,d_t} (n \wedge p) \pnorm{\X}{\op}^{2k} \leq k! (n \wedge p) \pnorm{\X}{\op}^{2k}.
\end{equation}
%\ys{This can be improved to $\pnorm{\X}{\op}^{2k-2}\pnorm{\X}{F}^2$.}
% \nb{I haven't worked it out for if we relax the assumptions to $ d_j \geq 1 $, which is needed for the non-centered case. I'm not sure if an additional assumption $ \|H \1\|_\infty <= C\opnorm{X} $ can solve the problem.}
\end{lemma}

\begin{proof}
Recall from (\ref{eq:T_diag_free}) that $\tT^{(d_1,\dots,d_t)}_{s_1,\dots,s_t} = \binom{k}{d_1,\dots,d_t} \sum_{i=1}^p H_{i,s_1}^{d_1} \cdots H_{i,s_t}^{d_t}$ for $1\leq s_1 < \ldots < s_t \leq p$ and zero otherwise.
We have
\begin{align*}
|\gamma^{{(d_1,\dots,d_t)}}| &\leq \binom{k}{d_1,\dots,d_t} \abs{\sum_{i=1}^p \sum_{1 \leq s_1 < \cdots < s_t \leq p}  H_{i,s_1}^{d_1} \cdots H_{i,s_t}^{d_t}} \\
&\leq \binom{k}{d_1,\dots,d_t} \sum_{i=1}^p \abs{ \opnorm{\H}^{d_1 - 2} \cdots \pnorm{\H}{\op}^{d_t - 2} \sum_{1 \leq s_1 < \cdots < s_t \leq p}  H_{i,s_1}^{2} \cdots H_{i,s_t}^{2}} \\
&\leq \binom{k}{d_1,\dots,d_t} \pnorm{\H}{\op}^{k - 2t} \sum_{i=1}^p \abs{ \sum_{1 \leq s_1 < \cdots < s_{t-1}} H_{i,s_1}^2 \cdots H_{i,s_{t-1}}^2 \sum_{s_t = 1}^p H_{i,s_t}^2 }\\
&\leq \binom{k}{d_1,\dots,d_t} \pnorm{\H}{\op}^{k - 2t} \|\H\|_\op^2 \sum_{i=1}^p \abs{ \sum_{1 \leq s_1 < \cdots < s_{t-1}} H_{i,s_1}^2 \cdots H_{i,s_{t-1}}^2  } \\
&\leq \ldots \leq \binom{k}{d_1,\dots,d_t} \pnorm{\H}{\op}^{k - 2t} \|\H\|_\op^{2(t-1)} \sum_{i=1}^p \norm{\H_{i\cdot}}^2\\
&= \binom{k}{d_1,\dots,d_t} \pnorm{\X}{\op}^{2k - 4}\pnorm{\H}{F}^2\leq \binom{k}{d_1,\dots,d_t} \pnorm{\X}{\op}^{2k} (n \wedge p),
\end{align*}
as desired.
\end{proof}

% \subsection{Variance bound for centered model}\label{subsection:variance_centered}

%Next we bound $\Var(\bar{m}_k)$, for which it suffices to bound $\Var(F_k(\bar{\y}))$ and $\Var(\sum_{t=2}^\ell \sum_{d_1 + \ldots + d_t = \ell: d_j \geq 2} \overline{m}_{d_1}^\circ \ldots \overline{m}_{d_t}^\circ \sum_{s_1,\ldots,s_t = 1}^p \tilde{T}^{(d_1,\ldots,d_t)}_{s_1,\ldots,s_t})$ separately. We first consider the typical term in the summation \eqref{eq:moment_def}, i.e. $ \sum_{d_1+\cdots+d_t = k, d_j \geq 2} \gamma^{(d_1,\dots,d_t)} \tm_{d_1} \cdots \tm_{d_t} $. Thanks to the uniform bound Lemma \ref{lem:gamma}, it suffices to estimate $ \Var(\tm_{d_1} \cdots \tm_{d_t}) $. An important observation is that the truncated estimator $ \tm_d $ is bounded, which facilitates us to apply the following lemma.

The next few results are useful for bounding the variance of $\bar\mu_k$.

\begin{lemma}\label{lem:variance_bound}
Suppose that Assumption \ref{assump:prior} holds for some $M > 0$ and $\pnorm{\X}{\op}\geq c_0\sigma$ for some $c_0 > 0$. Then for any $k\geq 2$ and some $C_0 = C_0(c_0,M) > 0$, 
\begin{align}
\Var(F_k(\bar\y)) \leq (C_0k)^{5k}(n\wedge p)\pnorm{\X}{\op}^{4k}.
\end{align}
\end{lemma}
\begin{proof}
Note that $\Var(F_k(\bar{\y})) =  \Var(\E[F_k(\bar{\y})|\bar{\bbeta}]) + \E[\Var(F_k(\bar{\y})|\bar{\bbeta})]$. For the first term, recall from (\ref{eq:Fk_expectation}) that $\E[F_k(\bar{\y}) | \bar{\bbeta}] = \iprod{\T^{(k)}}{\bar{\bbeta}^{\otimes k}}$, so by Lemma \ref{lem:var_of_exp} below, 
\begin{align}\label{eq:Fk_var_1}
\Var(\E[F_k(\bar{\y})|\bar{\bbeta}]) = \Var(\langle \T^{(k)},\bar{\bbeta}^k\rangle)\leq 8^k k^{5k}M^{2k}(n\wedge p)\pnorm{\X}{\op}^{4k}.
\end{align}
On the other hand, by Lemma \ref{lem:exp_of_var}, we have
\begin{align}\label{eq:Fk_var_2}
\E[\Var(F_k(\bar{\y})|\bar{\bbeta})] \leq \Big(\frac{CMk}{c_0}\Big)^{2k+2}  (n \wedge p) \opnorm{\X}^{4k}
\end{align}
for some universal $C > 0$. Combining these yields the conclusion.
\end{proof}

\begin{lemma}\label{lem:Product_Variance}
Suppose $X_1,\dots,X_n$ are bounded random variables for some positive integer $n$. Then
\begin{equation}
\Var(X_1 \cdots X_n) \leq n 2^{n-1} \max_{i\in[n]} \|X_i\|_\infty^{2(n-1)} \cdot \max_{i\in[n]} \Var(X_i),
\end{equation}
where $\pnorm{X}{\infty}$ denotes the essential supremum of $X$.
\end{lemma}
\begin{proof}
We prove by induction. For the base case $ n=2 $, we have
\begin{align}\label{eq:Product_Variance}
\Var(X_1 X_2) &\leq 2 \Var((X_1 - \E[X_1]) X_2) + 2 \Var(\E[X_1] X_2) \nonumber\\
&\leq 2 \E[((X_1 - \E[X_1])X_2)^2] + 2 \|X_1\|_\infty^2 \Var(X_2) \nonumber\\
&\leq 2 \|X_2\|_\infty^2 \E[(X_1 - \E[X_1])^2] + 2 \|X_1\|_\infty^2 \Var(X_2) \nonumber\\
&= 2 \|X_2\|_\infty^2 \Var(X_1) + 2 \|X_1\|_\infty^2 \Var(X_2)\\
\notag&\leq 4 \max(\pnorm{X_1}{\infty}^2, \pnorm{X_2}{\infty}^2)\cdot \max(\Var(X_1),\Var(X_2)),
\end{align}
which agrees with the desired bound. Let $A:= \max_{i\in[n]} \|X_i\|_\infty$.
Suppose the result holds for $ n-1 $, applying \eqref{eq:Product_Variance} yields
\begin{align*}
\Var(X_1 \cdots X_n) &\leq 2 \|X_1\|_\infty^2 \Var(X_2 \cdots X_n) + 2 \|X_2 \cdots X_n\|_\infty^2 \Var(X_1)\\
&\leq 2 \|X_1\|_\infty^2 \Big[(n-1) 2^{n-2} A^{2(n-2)} \max_{2\leq j \leq n} \Var(X_j)\Big] + 2 A^{2(n-1)} \Var(X_1)\\
&\leq n 2^{n-1} A^{2(n-1)} \cdot \max_{j\in[n]} \Var(X_j),
\end{align*}
which proves the desired bound.
\end{proof}

% In the context of variance estimation for $ \checkm_k $, we apply Lemma \ref{lem:Product_Variance} to the term $ \tm_{d_1} \cdots \tm_{d_t} $ and recall the truncation $ |\tm_d| \leq 2M^d $. As a consequence, it results in

\begin{lemma}\label{lem:var_of_exp}
Suppose that $\bbeta = (\beta_1,\ldots,\beta_p)$ has i.i.d.\ entries distributed as $g$, where $g$ is centered and satisfies Assumption \ref{assump:prior} for some $M > 0$. Then for any $k\geq 2$, we have
\begin{align*}
% \Var \pth{\iprod{T^{(k)}}{\beta^{\otimes k}}} \leq k^{4k+2} M^{2k} p \opnorm{X}^{4k}
\Var \Big(\iprod{\T^{(k)}}{\bbeta^{\otimes k}}\Big) \leq 8^k k^{5k}M^{2k}(n\wedge p)\pnorm{\X}{\op}^{4k}.
\end{align*}
\end{lemma}
\begin{proof}
Using the diagonal-free decomposition (\ref{eq:T_expansion}) of the tensor $\T^{(k)} $ and then expanding using $\beta_{s_j}^{d_j} = \beta_{s_j}^{d_j} - m_{d_j} + m_{d_j}$ (where $m_{d_j}$ denotes the $d_j$th moment of $\beta_1$), we have
\begin{align}\label{eq:T_beta_expansion}
\notag\iprod{\T^{(k)}}{\bbeta^{\otimes k}} &= \sum_{t=1}^k \sum_{d_1 + \ldots + d_t = k: d_j\geq 1} \sum_{s_1,\ldots,s_t = 1}^p \tilde{T}^{(d_1,\ldots,d_t)}_{s_1,\ldots,s_t} \beta^{d_1}_{s_1}\ldots \beta^{d_t}_{s_t}\\
&= \sum_{t=1}^k \sum_{d_1+\cdots+d_t = k: d_j \geq 1} \sum_{\cI \subseteq [t]}\sum_{s_1,\dots,s_t = 1}^p \tT_{s_1 \dots s_t}^{(d_1,\dots,d_t)} \prod_{i \in \cI}(\beta_{s_{i}}^{d_{i}} - m_{d_i}) \cdot \prod_{i\notin \cI} m_{d_i}. %\label{eq:Tensor_T_expansion}
\end{align}
We proceed to bound the variance of the innermost summation, which, up to a relabeling, is of the following form: for some $u\equiv |\cI|\in[t]$, 
\begin{align}\label{eq:var_typical}
\notag&\Var \Big(\sum_{s_1,\dots,s_t = 1}^p \tT_{s_1 \dots s_t}^{(d_1,\dots,d_t)} (\beta_{s_{1}}^{d_{1}} - m_{d_1}) \cdots (\beta_{s_{u}}^{d_{u}} - m_{d_u}) m_{d_{u+1}}\cdots m_{d_t}\Big)\\
\notag&= m_{d_{u+1}}^2 \cdots m_{d_t}^2 \Var \Big(\sum_{s_1,\dots,s_u=1}^p \Big(\sum_{s_{u+1},\dots,s_t=1}^p \tT_{s_1 \dots s_t}^{(d_1,\dots,d_t)}\Big) (\beta_{s_1}^{d_1} - m_{d_1}) \cdots (\beta_{s_u}^{d_u} - m_{d_u})\Big) \\
&= \Var(\beta_{1}^{d_1}) \cdots \Var(\beta_{1}^{d_u}) \, m_{d_{u+1}}^2 \cdots m_{d_t}^2 \sum_{s_1,\dots,s_u=1}^p \Big(\sum_{s_{u+1},\dots,s_t=1}^p \tT_{s_1 \dots s_t}^{(d_1,\dots,d_t)}\Big)^2,
\end{align}
and we note that, since $m_1\equiv 0$, this term is 0 unless $d_{u+1},\ldots, d_t \geq 2$.
%where the last line follows from Lemma \ref{lem:diag_free_variance}.
% \nb{We need the assumption $ m_1 = 0 $ to obtain the desired variance bound.}
% Recall the notation $ H^{(d)} := H^{\circ d} $ and the property $ \opnorm{H^{(d)}} \leq \opnorm{H}^d = \opnorm{X}^{2d} $. It suffices to control
% \[
% \sum_{s_1,\dots,s_u} \pth{\sum_{s_{u+1},\dots,s_t} \tT_{s_1 \dots s_t}^{(d_1,\dots,d_t)}}^2 = \sum_{1 \leq s_1 < \cdots < s_u} \pth{ \sum_{i=1}^p H^{(d_1)}_{i s_1} \cdots H^{(d_u)}_{i s_u} \sum_{(s_u) < s_{u+1} < \cdots < s_t <= p} H^{(d_{u+1})}_{is_{u+1}} \cdots H^{(d_t)}_{is_t} }^2.
% \]
%Thanks to the assumption $ m_1 = 0 $, only terms with $ d_{u+1}, \dots, d_t \geq 2 $ make non-trivial contribution to the variance. 
Next note that, the $d$th Hadamard product $\H^{(d)} = \H\circ\ldots\circ \H$ satisfies\footnote{A more general result than \prettyref{eq:Hadamardop} is $\opnorm{A\circ B} \leq \opnorm{A}\opnorm{B}$. Indeed, for any unit vector $x$,
$x^\top (A\circ B) x = 
\text{vec}(\diag(x))^\top (A\otimes B) \text{vec}(\diag(x))$, where $A\otimes B$ is the Kronecker product, so $\opnorm{A\circ B} \leq \opnorm{A\otimes B} = \opnorm{A}\opnorm{B}$.
} 
\begin{equation}
    \pnorm{\H^{(d)}}{\op}\leq \pnorm{\H}{\op}^d.
    \label{eq:Hadamardop}
\end{equation}
Therefore
\begin{align*}
&{k\choose d_1,\ldots,d_t}^{-2}\sum_{s_1,\dots,s_u} \Big(\sum_{s_{u+1},\dots,s_t} \tT_{s_1 \dots s_t}^{(d_1,\dots,d_t)}\Big)^2\\
&= \sum_{1 \leq s_1 < \cdots < s_u} \Big( \sum_{i=1}^p  \sum_{(s_u) < s_{u+1} < \cdots < s_t \leq p} H^{d_1}_{i,s_1} \cdots H^{d_u}_{i,s_u}H^{d_{u+1}}_{i,s_{u+1}} \cdots H^{d_t}_{i,s_t}\Big)^2\\
&= \sum_{1 \leq s_1 < \ldots < s_u}\Big(\sum_{i=1}^p H^{ d_1}_{i,s_1}\ldots H^{d_u}_{i,s_u}\underbrace{\sum_{(s_u) < s_{u+1} < \ldots < s_t\leq p} H^{d_{u+1}}_{i,s_{u+1}} \ldots H^{ d_{t}}_{i,s_{t}}}_{\lambda_{i;s_u}}\Big)^2\\
&\leq \sum_{s_1,\ldots,s_u = 1}^p \Big(\sum_{i=1}^p \lambda_{i;s_u}H^{ d_1}_{i,s_1}\ldots H^{d_u}_{i,s_u} \Big)^2 \\
&\stackrel{(*)}{\leq} \sum_{i=1}^p \max_{j\in[p]}\lambda_{i;j}^2 \cdot \pnorm{\H^{(d_1)}}{\op}^2\ldots \pnorm{\H^{(d_u)}}{\op}^2 \leq \pnorm{\X}{\op}^{4(d_1+\ldots+d_u)} \cdot \sum_{i=1}^p \max_{j\in[p]}\lambda_{i;j}^2,
\end{align*}
%\nbwu{what's semicolon in $\lambda_{i;j}$? Check everywhere}
%\hw{This could be also bounded by $ \opnorm{\H^{(d_1)}}^2 \pnorm{\H^{(d_2)}}{\infty}^2 \dots \pnorm{\H^{(d_u)}}{\infty}^2 \sum_{i=1}^p \max_{j \in [p]} \lambda_{i;j}^2 $. Note that $ \pnorm{\H^{(d)}}{\infty}^2 = \max_{i \in [p]} \sum_{j=1}^p \H_{ij}^{2d} \leq \max_{i \in [p]} (\sum_{j=1}^p \H_{ij}^2)^d = \max_{i \in [p]} \norm{\H_{i \cdot}}^{2d} = \pnorm{\H}{\infty}^{2d} $. Thus, the term above can be further bounded by $ \opnorm{\H}^{2d_1} \pnorm{\H}{\infty}^{2(d_2+\dots+d_u)} \sum_{i=1}^p \max_{j\in[p]}\lambda_{i;j}^2 $}
where $(*)$ follows from Lemma \ref{lem:Matrix_Product}, and the $j$ in $\lambda_{i;j}$ denotes the lower range of the summation over $s_{u+1},\ldots,s_t$. It remains to notice that: for $d_{u+1},\ldots,d_t \geq 2$, we have the following bound uniform in $j$:
\begin{align}\label{eq:lambda_bound}
\notag|\lambda_{i;j}| &= \Big|\sum_{j< s_{u+1}<\ldots < s_t} H^{d_{u+1}}_{i,s_{u+1}}\ldots H_{i,s_t}^{d_t}\Big|\\
\notag&\leq \pnorm{\X}{\op}^{2((d_{u+1}-2)+\ldots+(d_t-2))}\sum_{j< s_{u+1}<\ldots < s_t} H^2_{i,s_{u+1}}\ldots H_{i,s_t}^2\\
\notag&\leq \pnorm{\X}{\op}^{2((d_{u+1}-2)+\ldots+(d_t-2))} \pnorm{\X}{\op}^4 \cdot \sum_{j < s_{u+1} < \ldots < s_{t-1}}H^2_{i,s_{u+1}}\ldots H_{i,s_{t-1}}^2\\
&\leq \ldots \leq \pnorm{\X}{\op}^{2(d_{u+1}+\ldots + d_t-2)}\sum_{j<s_{u+1}} H_{i,s_{u+1}}^2 \leq \pnorm{\X}{\op}^{2(d_{u+1}+\ldots + d_t-2)} \pnorm{\H_{i,\cdot}}{}^2,
\end{align}
%\hw{This could also be bounded by $\pnorm{\H}{\infty}^{2(d_{u+1} + \dots + d_{t})}$}
which implies that
\begin{align*}
&\sum_{s_1,\dots,s_u} \Big(\sum_{s_{u+1},\dots,s_t} \tT_{s_1 \dots s_t}^{(d_1,\dots,d_t)}\Big)^2 \leq {k\choose d_1,\ldots,d_t}^2 \pnorm{\X}{\op}^{4k-8}\sum_{i=1}^p \pnorm{\H_{i,\cdot}}{}^4\\
&\leq {k\choose d_1,\ldots,d_t}^2 \pnorm{\X}{\op}^{4k-4}\sum_{i=1}^p \pnorm{\H_{i,\cdot}}{}^2 = {k\choose d_1,\ldots,d_t}^2 \pnorm{\X}{\op}^{4k-4} \pnorm{\H}{F}^2 \leq k^{2k}(n\wedge p)\pnorm{\X}{\op}^{4k}.
\end{align*}
Recall that $K_d = M^d d^{d/2}$ and $m_d(g) \leq K_d$. Then using $\Var(\beta_1^d)\leq m_{2d}(g) \leq K_{2d}$, the above display further implies that the typical term in (\ref{eq:var_typical}) can be bounded by
\begin{align*}
&\Var \Big(\sum_{s_1,\dots,s_t = 1}^p \tT_{s_1 \dots s_t}^{(d_1,\dots,d_t)} (\beta_{s_{1}}^{d_{1}} - m_{d_1}) \cdots (\beta_{s_{u}}^{d_{u}} - m_{d_u}) m_{d_{u+1}}\cdots m_{d_t}\Big) \\
&\leq \prod_{i=1}^u M^{2d_i}(2d_i)^{d_i} \cdot \prod_{i = u+1}^t M^{2d_i}d_i^{d_i} \cdot k^{2k} (n\wedge p)\pnorm{\X}{\op}^{4k}\leq 2^kk^{3k}M^{2k}(n\wedge p)\pnorm{\X}{\op}^{4k},
\end{align*}
where we apply the fact that $\max_{d_1+\ldots + d_t \leq k} d_1^{d_1}\ldots d_t^{d_t} = (k/t)^k \leq k^k$. 
% Note that for $n$ not necessarily independent variables $X_1,\ldots,X_n$, we have
% \begin{align*}
% \Var\Big(\sum_{i=1}^n X_i\Big) = \E \Big(\sum_{i=1}^n (X_i - \E[X_i])\Big)^2 \leq \E\Big[n \cdot \sum_{i=1}^n (X_i - \E[X_i])^2\Big] = n \cdot\sum_{i=1}^n \Var(X_i).
% \end{align*}

Finally, applying Jensen's inequality to (\ref{eq:T_beta_expansion}) and noting that the total number of summands satisfies $\sum_{t=1}^k \sum_{d_1+\cdots+d_t = k, d_j \geq 1} \sum_{\cI \subseteq [t]} 1 \leq k^k 2^k$, we have
\begin{align*}
\Var\Big(\iprod{\T^{(k)}}{\bbeta^k}\Big) \leq (k^k 2^k)^2 \cdot 2^kk^{3k}M^{2k}(n\wedge p)\pnorm{\X}{\op}^{4k} = 8^k k^{5k}M^{2k}(n\wedge p)\pnorm{\X}{\op}^{4k},
\end{align*}
as desired.
\end{proof}

%\begin{lemma}\label{lem:diag_free_variance}
%For $ i \in [p] $, let $ (u_i^{(1)}, \dots, u_i^{(m)}) $ be a set of iid mean-zero random vectors in $ \R^m $. Let $ T \in (\R^p)^{\otimes m} $ be an upper triangular tensor (i.e., $T_{s_1,\ldots,s_m} = 0$ unless $1 \leq s_1 < \ldots < s_m \leq p$). Then,
%\begin{align*}
%\Var \Big(\sum_{i_1,\dots,i_m = 1}^p T_{i_1,\dots,i_m} u_{i_1}^{(1)} \cdots u_{i_m}^{(m)}\Big) = \sum_{i_1,\dots,i_m = 1}^p T_{i_1,\dots,i_m}^2 \prod_{j=1}^m \E[(u_1^{(j)})^2].
%\end{align*}
%\end{lemma}
%\begin{proof}
%By the upper-triangular assumption, the expression is mean-zero and the variance equals the second moment.
%\begin{align*}
%\Var \Big(\sum_{i_1,\dots,i_m = 1}^p T_{i_1,\dots,i_m} u_{i_1}^{(1)} \cdots u_{i_m}^{(m)}\Big) = \E \Big[\Big(\sum_{i_1,\dots,i_m = 1}^p T_{i_1,\dots,i_m} u_{i_1}^{(1)} \cdots u_{i_m}^{(m)}\Big)^2\Big].
%\end{align*}
%By independence of the random vectors, we further obtain
%\begin{align*}
%\E \Big[\Big(\sum_{i_1,\dots,i_m = 1}^p T_{i_1,\dots,i_m} u_{i_1}^{(1)} \cdots u_{i_m}^{(m)}\Big)^2\Big] = \sum_{i_1,\dots,i_m = 1}^p T_{i_1,\dots,i_m}^2 \prod_{j=1}^m \E[(u_1^{(j)})^2],
%\end{align*}
%as desired.
%\end{proof}

\begin{lemma}\label{lem:Matrix_Product}
Let $ A^{(1)},\dots,A^{(t)} $ be $m\times p$ matrices. For any set of real numbers $\{\lambda_{ij}$, $i\in[p], j\in[m]\}$, it holds that
\begin{align*}
\sum_{s_1,\dots,s_t=1}^m \Big(\sum_{i=1}^p \lambda_{i,s_t} A^{(1)}_{s_1,i} \cdots A^{(t)}_{s_t,i}\Big)^2 \leq \pnorm{A^{(1)}}{\op}^2 \cdots \pnorm{A^{(t)}}{\op}^2 \sum_{i=1}^p \max_{j \in [m]} \lambda_{ij}^2.
\end{align*}
\end{lemma}
\begin{proof}
For any $ s_2,\dots,s_t \in [m] $, let $ \Lambda^{(s_2,\dots,s_t)} = (\{\lambda_{i,s_t} A^{(2)}_{s_2, i} \cdots A^{(t)}_{s_t,i}\}_{i=1}^p) \in \R^{p} $. Then,
\begin{align*}
\sum_{s_1,\dots,s_t=1}^m \Big(\sum_{i=1}^p \lambda_{i,s_t} A^{(1)}_{s_1,i} \cdots A^{(t)}_{s_t,i}\Big)^2 &= \sum_{s_2,\dots,s_t = 1}^m \biggpnorm{\sum_{i=1}^p \lambda_{i,s_t}A^{(2)}_{s_2,i} \cdots A^{(t)}_{s_t,i} \cdot A^{(1)}_{\cdot i}}{}^2\\
&= \sum_{s_2,\dots,s_t = 1}^m \pnorm{A^{(1)} \Lambda^{(s_2,\dots,s_t)} }{}^2\\
&\leq \pnorm{A^{(1)}}{\op}^2 \sum_{i=1}^p \sum_{s_2,\dots,s_t = 1}^m  \lambda_{i,s_t}^2 (A^{(2)}_{s_2,i} \cdots A^{(t)}_{s_t,i})^2\\
&= \pnorm{A^{(1)}}{\op}^2 \sum_{i=1}^p \pnorm{A^{(2)}_{\cdot i}}{}^2 \cdots \pnorm{A^{(t-1)}_{ \cdot i}}{}^2 \sum_{s_t=1}^m \lambda_{i,s_t}^2 (A^{(t)}_{s_t,i})^2\\
&\leq \pnorm{A^{(1)}}{\op}^2 \cdots \pnorm{A^{(t)}}{\op}^2 \sum_{i=1}^p \max_{j \in [m]} \lambda_{ij}^2,
\end{align*}
as desired.
%\hw{This could also be bounded by $ \opnorm{A^{(1)}}^2 \pnorm{A^{(2)}}{\infty}^2 \dots \pnorm{A^{(t)}}{\infty}^2 \sum_{i=1}^p \max_{j \in [m]} \lambda_{ij}^2 $}
\end{proof}

\begin{lemma}\label{lem:exp_of_var}
Recall $F_k(\cdot)$ in (\ref{eq:F_k}) and assume $\y = \X\bbeta + \beps$ with $\beps\sim \N(0,\sigma^2\I_n)$. For any integer $k\geq 2$, it holds that
\begin{align*}
\Var(F_k(\y)|\bbeta) \leq \sum_{\ell=1}^{k-1} {k\choose\ell}^2\ell!\sigma^{2\ell}\pnorm{\X}{\op}^{2\ell}\sum_{j=1}^p (\H\bbeta)_j^{2(k-\ell)} + k!\sigma^{2k}(n\wedge p)\pnorm{\X}{\op}^{2k}. 
\end{align*}
Moreover, if $\beta_1,\ldots,\beta_p$ are i.i.d.\ centered and $M$-subgaussian for some $M > 0$ and $\pnorm{\X}{\op}\geq c_0\sigma$ for some $c_0 > 0$, then
\begin{align*}
\E_{\bbeta}[\Var(F_k(\y)|\bbeta)]\leq \Big(\frac{CMk}{c_0}\Big)^{2k+2}(n\wedge p)\pnorm{\X}{\op}^{4k},
\end{align*}
where $C > 0$ is universal.
\end{lemma}
\begin{proof}
By Lemma \ref{lem:hermite_cov}, we have
\begin{align*}
\Var(F_k(\y) | \bbeta) &= \sum_{j,j'=1}^p (\sigma \pnorm{\x_j}{})^k (\sigma\pnorm{\x_j'}{})^k \Cov\Big(H_k\Big(\frac{\x_j^\top \y}{\sigma\pnorm{\x_j}{}}\Big),H_k\Big(\frac{\x_{j'}^\top \y}{\sigma\pnorm{\x_{j'}}{}}\Big)\Big)\\
&= \sum_{j,j'=1}^p \sum_{\ell=1}^k {k\choose \ell}^2 \ell! (\sigma \pnorm{\x_j}{})^k (\sigma\pnorm{\x_j'}{})^k\Big(\frac{\x_j^\top \btheta}{\sigma\pnorm{\x_j}{}}\Big)^{k-\ell}\Big(\frac{\x_{j'}^\top \btheta}{\sigma\pnorm{\x_{j'}}{}}\Big)^{k-\ell}\Big(\frac{\x_j^\top \x_{j'}}{\pnorm{\x_j}{}\pnorm{\x_{j'}}{}}\Big)^\ell\\
&= \sum_{\ell=1}^k{k\choose \ell}^2 \ell!\sigma^{2\ell} \cdot \sum_{j,j'=1}^p (\x_j^\top \btheta)^{k-\ell}(\x_{j'}^\top\btheta)^{k-\ell} (\x_j^\top \x_{j'})^\ell.
\end{align*}
Let us bound the inner summation for each $\ell\in[k]$. If $\ell = k \geq 2$, we have 
\begin{align*}
\big|\sum_{j,j'=1}^p (\x_j^\top \x_{j'})^k\big| \leq \pnorm{\X}{\op}^{2(k-2)} \sum_{j,j'=1}^p (\x_j^\top \x_{j'})^2 = \pnorm{\X}{\op}^{2(k-2)} \pnorm{\H}{F}^2 \leq (n\wedge p)\pnorm{\X}{\op}^{2k}. 
\end{align*}
For $\ell \leq k-1$, $(\H\bbeta)^{\circ \alpha}\in\R^p$ denotes the $\alpha$th Hadamard product of $\H\bbeta$, then using $\x_j^\top \btheta = (\H \bbeta)_j$,
\begin{align*}
&\Big|\sum_{j,j'=1}^p (\x_j^\top \btheta)^{k-\ell}(\x_{j'}^\top\btheta)^{k-\ell} (\x_j^\top \x_{j'})^\ell\Big|= \Big|\sum_{j,j'=1}^p [(\H\bbeta)^{\circ(k-\ell)}]_j [(\H\bbeta)^{\circ(k-\ell)}]_{j'} [\H^{\circ \ell}]_{j,j'} \Big|\\
&=(\H\bbeta)^{\circ(k-\ell)\top}  \H^{\circ \ell} (\H\bbeta)^{\circ(k-\ell)}\leq \pnorm{\H^{\circ \ell}}{\op}\pnorm{(\H\bbeta)^{\circ (k-\ell)}}{}^2 \leq \pnorm{\X}{\op}^{2\ell}\sum_{j=1}^p (\H\bbeta)_j^{2(k-\ell)},
\end{align*}
leading to the desired bound for $\Var(F_k(\y)|\bbeta)$. To bound its expectation over $\bbeta$, note that for each $j\in[p]$, $(\H\bbeta)_j$ is subgaussian with constant $CM\pnorm{\H_{j\cdot}}{}$ for some universal $C > 0$ (in the following, $C > 0$ denotes a universal constant whose value changes from instance to instance). Hence the subgaussian assumption yields
\begin{align*}
\sum_{j=1}^p \E (\H \bbeta)_j^{2(k-\ell)} \leq (CM)^{2(k-
\ell)}\pnorm{\X}{\op}^{4(k-\ell-1)}\pnorm{\H}{F}^2 (2k)^k\leq (CM)^{2k}(2k)^k(n\wedge p)\pnorm{\X}{\op}^{4(k-\ell)},
\end{align*}
which further implies that 
%We claim that for each $\ell\in[k]$, $\Big|\sum_{j,j'=1}^p (\x_j^\top \btheta)^{k-\ell}(\x_{j'}^\top\btheta)^{k-\ell} (\x_j^\top \x_{j'})^\ell\Big| \leq p\pnorm{\X^\top\btheta}{\infty}^{2(k-\ell)}\pnorm{\X}{\op}^{2\ell}$. To see this, if $\ell = 1$, then
%\begin{align*}
%\Big|\sum_{j,j'=1}^p (\x_j^\top \btheta)^{k-\ell}(\x_{j'}^\top\btheta)^{k-\ell} (\x_j^\top \x_{j'})^\ell\Big| &= \biggpnorm{\sum_{j=1}^p (\x_j^\top\btheta)^{k-1} \x_j}{}^2 = \biggpnorm{\X\Big(\sum_{j=1}^p (\x_j^\top\btheta)^{k-1}\e_j\Big)^2}{}\\
%&\leq \pnorm{\X}{\op}^2 \sum_{j=1}^p (\x_j^\top\btheta)^{2(k-1)} \leq p\pnorm{\X}{\op}^2\pnorm{\X^\top\btheta}{\infty}^{2(k-1)}.
%\end{align*}
%If $\ell \geq 2$, then
%\begin{align*}
%\Big|\sum_{j,j'=1}^p (\x_j^\top \btheta)^{k-\ell}(\x_{j'}^\top\btheta)^{k-\ell} (\x_j^\top \x_{j'})^\ell\Big| &\leq \sum_{j,j'=1}^p |\x_j^\top \btheta|^{k-\ell}|\x_{j'}^\top \btheta|^{k-\ell}|\x_j^\top \x_{j'}|^2\pnorm{\x_j}{}^{\ell-2}\pnorm{\x_{j'}}{}^{\ell-2}\\
%&\leq \pnorm{\X^\top \btheta}{\infty}^{2(k-\ell)}\pnorm{\X}{\op}^{2(\ell-4)} \pnorm{\X^\top \X}{F}^2 \leq p\pnorm{\X^\top\btheta}{\infty}^{2(k-\ell)}\pnorm{\X}{\op}^{2\ell},
%\end{align*}
%proving the claim. 
\begin{align*}
\E_{\bbeta}\big[\Var(F_k|\bbeta)\big] &\leq (n\wedge p)(CM)^{2k} (2k)^k \sum_{\ell=1}^k {k\choose \ell}^2 \ell! \sigma^{2\ell} \pnorm{\X}{\op}^{4k-2\ell}\\
&\leq (n\wedge p)(CM)^{2k}k^{2k}8^k\sum_{\ell=1}^k \sigma^{2\ell}\pnorm{\X}{\op}^{4k-2\ell}\\
&\leq (n\wedge p)(CM)^{2k}k^{2k}8^k\pnorm{\X}{\op}^{4k}\frac{\sigma^2}{\pnorm{\X}{\op}^2}\sum_{\ell=0}^{k-1}\Big(\frac{\sigma}{\pnorm{\X}{\op}}\Big)^{2\ell}\\
&\leq (n\wedge p)(CM)^{2k}k^{2k+1}8^k\sigma^2\pnorm{\X}{\op}^{4k-2}\Big(1\vee \frac{\sigma}{\pnorm{\X}{\op}}\Big)^{2(k-1)}\\
&\leq (n\wedge p)\Big(\frac{CMk}{c_0}\Big)^{2k+2}\pnorm{\X}{\op}^{4k},
\end{align*}
where we apply Assumption \ref{assump:design} in the last step.
%\hw{We can also write the bound as
%\[
%\E_{\bbeta}\Var(F_k|\bbeta) \leq (CM)^{2k} k^{3k + 1} (n %\wedge p) \opnorm{\X}^{4k} \frac{\sigma^2}{\opnorm{\X}^2} \Big( \frac{\sigma}{\opnorm{\X}} \vee 1 \Big)^{2k-2},
%\]
%so that when $ \sigma = 0 $ we do have $ %\E_{\bbeta}\Var(F_k|\bbeta) = 0 $.
%}
%If $\bbeta_1,\ldots,\bbeta_p$ are iid $M$-subgaussian (\ys{used centeredness here}), then each $(\X^\top \btheta)_j = (\X^\top \X \bbeta)_j$, $j\in[p]$ is $M'$-subgaussian with $M' \leq C\pnorm{(\X^\top \X)_{j\cdot}}{} \leq C\pnorm{\X}{\op}^2$, hence $\E\pnorm{\X^\top\btheta}{\infty}^{2k} \leq C^k\pnorm{\X}{\op}^{4k}(\log p)^k$, yielding the final conclusion $\E_{\beta}[\Var(F_k|\bbeta)] \leq p(C\log p)^k k!(1+\sigma^2)^k(\pnorm{\X}{\op}^{2k} \vee \pnorm{\X}{\op}^{4k})$.
\end{proof}

\begin{lemma}[Hermite polynomials of correlation Gaussians]\label{lem:hermite_cov}
Let $H_k$ denote the monic degree-$k$ Hermite polynomial. Let $(Z_1,Z_2)$ 
be joint normally distributed with $Z_1,Z_2\sim \N(0,1)$  and $\Expect[Z_1Z_2]=\rho \in [-1,1]$. Then for $k,k'\geq 1$,
\begin{align}
\Cov(H_k(Z_1),H_{k'}(Z_2)) = \bm{1}_{k=k'}\rho^kk!,
\label{eq:cov-Hermite1}
\end{align}
and for any $\theta_1,\theta_2\in\R$,
%\nbwu{Note that the sum below starts from 1. The old version starts from 0.}
\begin{align}
\Cov(H_k(\theta_1+Z_1),H_{k}(\theta_2+Z_2)) = \sum_{\ell=1}^k {k\choose \ell}^2 \ell! \cdot \theta_1^{k-\ell}\theta_2^{k-\ell}\rho^\ell.
\label{eq:cov-Hermite2}
\end{align}
\end{lemma}
\begin{proof}
We first prove \prettyref{eq:cov-Hermite1}.
Let $f_\rho(x,y)$ denote the joint density of $(Z_1,Z_2)$ and $f_0(x,y)=\varphi(x)\varphi(y)$, where $\varphi(\cdot)$ is the standard normal density.
It is well-known that the likelihood ratio $\frac{f_\rho}{f_0}$ (Mehler kernel) admits the following spectral decomposition:
\[
K(x,y) = 
\frac{f_\rho(x,y)}{f_0(x,y)}
= \sum_{k\geq 0} \frac{\rho^k}{k!} H_k(x)H_k(y).
\]
Let $W_1,W_2$ be i.i.d.\ standard normal. Then
\begin{align*}
 \Cov(H_k(Z_1),H_{k'}(Z_2)) 
 =   & ~ \Expect[H_k(Z_1)H_{k'}(Z_2)] \\
 =   & ~ \Expect[H_k(W_1)H_{k'}(W_2)K(W_1,W_2)] \\
 =   & ~ 
 \sum_{\ell\geq 0} \frac{\rho^\ell}{\ell!} \Expect[(H_k(W_1) H_\ell(W_1)] \Expect[H_{k'}(W_2)H_\ell(W_2)] \\
 = & ~ \begin{cases}
     \rho^k k! & k=k'\\
     0 & k\neq k'\\
 \end{cases},
\end{align*}
where we used the orthogonality 
$\Expect[H_k(W_1) H_\ell(W_1)] = \bm{1}_{k=\ell} k!$.

Finally, \prettyref{eq:cov-Hermite2} follows  \prettyref{eq:cov-Hermite1} and the identity that $H_k(x+y) = \sum_{\ell=0}^k {k\choose \ell}x^{k-\ell}H_{\ell}(y)$,
which is a consequence of $H_k(x) = \Expect_{W\sim\N(0,1)}[(x+ i W)^k]$.    
% Next to show \prettyref{eq:cov-Hermite2}, recall the identity that $H_k(x+y) = \sum_{\ell=0}^k {k\choose \ell}x^{k-\ell}H_{\ell}(y)$ 
% which is a consequence of $H_k(x) = \Expect_{W\sim\N(0,1)}[(x+ i W)^k]$.    
% Then
% \prettyref{eq:cov-Hermite2} follows applying  \prettyref{eq:cov-Hermite1} to
% \[
% \Cov(H_k(\theta_1+Z_1),H_{k}(\theta_2+Z_2)) 
% % = \Expect[(H_k(\theta_1+Z_1)-\theta_1^k)(H_{k}(\theta_2+Z_2)-\theta_2^k)]
% = \Cov\pth{\sum_{a=1}^k \binom{k}{a} \theta^{k-a} H_a(Z_1),
% \sum_{b=1}^k \binom{k}{b} \theta^{k-b} H_b(Z_2)}.
% \]
\end{proof}

\begin{lemma}\label{lem:estimate_L2}
There exists some $C_1 = C_1(M,c_0) > 0$ such that $\E[F_k^4(\bar\y)] \leq (C_1 k)^{6k}p^4\pnorm{\X}{\op}^{8k}$.
\end{lemma}
\begin{proof}
Using the explicit expansion of the Hermite polynomial
\begin{align}\label{eq:hermite_explicit}
H_k(x) = k!\sum_{m=0}^{\floor{k/2}}\frac{(-1)^m}{m!(k-2m)!2^m}x^{k-2m},
\end{align}
it is easy to obtain the bound
\begin{align*}
|H_k(x)| \leq (Ck)^k (1+|x|^k)
\end{align*}
for some universal $C >0 $. Hence using $\bar\y = \X\bar\bbeta + \beps = \x_j\bar\beta_j + \X_{-j}\bar\bbeta_{-j} + \beps$, we have
\begin{align*}
|F_k(\bar\y)| &= \Big|\sum_{j=1}^p (\sigma\pnorm{\x_j}{})^k H_k\Big(\frac{\x_j^\top \bar\y}{\sigma\pnorm{\x_j}{}}\Big)\Big| \leq (Ck)^k\sum_{j=1}^p \big((\sigma\pnorm{\x_j}{})^k + |\x_j^\top \bar\y|^k\big)\\
&\leq (Ck)^k \sum_{j=1}^p (\pnorm{\X}{\op}^{2k} + \pnorm{\X}{\op}^{2k}|\beta_j|^k + |\x_j^\top \X_{-j}\bbeta_{-j}|^k + |\x_j^\top \beps|^k),
\end{align*}
where we apply $\pnorm{\X}{\op} \geq c_0\sigma$ via Assumption \ref{assump:design}. This implies
\begin{align*}
\E[F_k^4(\bar\y)] \leq 64(Ck)^{4k}\Big[p^4\pnorm{\X}{\op}^{8k} + \pnorm{\X}{\op}^{8k} \E \Big(\sum_{j=1}^p |\bar\beta_j|^k\Big)^4 + \E \Big(\sum_{j=1}^p |\x_j^\top \X_{-j}\bar\bbeta_{-j}|^k\Big)^4 + \E\Big(\sum_{j=1}^p |\x_j^\top\beps|^k\Big)^4\Big].
\end{align*}
Using $\E|\bar\beta_j|^{4k} \leq M^{4k}(4k)^{2k}$, we have  
\begin{align*}
\E \Big(\sum_{j=1}^p |\bar\beta_j|^k\Big)^4 \leq p^3\sum_{j=1}^p \E[|\beta_j|^{4k}] \leq p^3 M^{4k}(4k)^{2k}. 
\end{align*}
Next, we note that $\u^\top \bar\bbeta$ is subgaussian with constant $CM\pnorm{\u}{}$ for some universal $C > 0$. Applying this fact yields that $\x_j^\top \X_{-j}\bar\bbeta_{-j}$ is subgaussian with constant $CM\pnorm{\X_{-j}^\top \x_j}{} \leq CM\pnorm{\X}{\op}^2$, so we have
\begin{align*}
\E \Big(\sum_{j=1}^p |\x_j^\top \X_{-j}\bar\bbeta_{-j}|^k\Big)^4 \leq p^3\sum_{j=1}^p \E |\x_j^\top \X_{-j}\bar\bbeta_{-j}|^{4k} \leq Cp^4\pnorm{\X}{\op}^{8k} M^{4k}(4k)^{2k}.
\end{align*}
A similar estimate yields $\E\Big(\sum_{j=1}^p |\x_j^\top\beps|^k\Big)^4 \leq Cp^4\pnorm{\X}{\op}^{4k}\sigma^{4k} (4k)^{2k} \leq C(4k/c_0^2)^{2k} p^4 \pnorm{\X}{\op}^{8k}$, using Assumption \ref{assump:design}. Putting together the estimates finishes the proof.
\end{proof}

Equipped with these auxiliary results, we are now ready to prove Proposition \ref{prop:MSE_centered} by induction. 

\begin{proof}[Proof of Proposition \ref{prop:MSE_centered}]
~
Recall that $K_\ell = M^\ell \ell^{\ell/2}$. We will prove via induction the following two hypotheses: 
\begin{align}
\E [(\bar\mu_k - \mu_k)^2 \Indc_{\cE_{k-1}}] & \leq \Delta_k,
    \label{eq:induction-hyp1}\\
\Prob(\cE_k^c) & \leq \Delta_k,
    \label{eq:induction-hyp2}
\end{align}
for all $2\leq k\leq k_0$, with $k=2$ being the base case. Recall  $\Delta_k$  defined in 
\prettyref{eq:MSE_centered}, namely
\begin{align*}
\Delta_k = (C_0 k)^{5k^2}  \frac{1}{n \wedge p}\Big( \frac{p}{n} \vee 1 \Big)^{2k^2}
\end{align*}
for some sufficiently large constant $C_0$ to be chosen later depending only on $M$ and $c_0$.

\paragraph{Base case} For $k=2$,  
% $\Delta_2 = (\frac{2CM}{c_0})^{27}\big(\frac{p}{n}\vee 1\big)^4\frac{1}{n\wedge p}$, 
 $\bar\mu_2 = \frac{F_2(\bar\y)}{A_2}$ as defined in (\ref{eq:centered_moment_ideal}) is unbiased and satisfies the variance bound in (\ref{eq:m2var}), namely:  
\begin{align*}
\Var(\bar\mu_2) \leq \frac{
(\mu_4 + 2 \mu_2^2)   \tr(\H^4)
+ 4  \sigma^2 \mu_2  \tr(\H^3)
+ 2 \sigma^4 \tr(\H^2) }{\tr(\H^2)^2}. 
\end{align*}
We now apply
\begin{align*}
\Tr(\H^4) \leq (n\wedge p)\pnorm{\X}{\op}^8, \quad \Tr(\H^3) \leq (n\wedge p)\pnorm{\X}{\op}^6, \quad \Tr(\H^2) \leq (n\wedge p)\pnorm{\X}{\op}^4,
\end{align*}
along with Assumption \ref{assump:design}, which implies $A_2 = \Tr(\H^2) \geq c_0(n\wedge p)\pnorm{\X}{\op}^4 $. Applying the subgaussian bound $\mu_k \leq (CM)^kk^{k/2}$ yields the desired result $\E[(\bar\mu_2 - \mu_2)^2] = \Var(\bar\mu_2)\leq \Delta_2$, provided that $C_0=C_0(c_0,M)$ is chosen sufficiently large.

\paragraph{Induction}
Suppose that the hypotheses (\ref{eq:induction-hyp1})-(\ref{eq:induction-hyp2}) hold up to some $k-1$. We now show that they also hold for $k$. We first prove (\ref{eq:induction-hyp1}), which can be decomposed as
%\nbwu{I think we should clearly state what the induction hypothesis is as well as which is applied where. I think it consists of two parts and below I added reference to which one at appropriate places. The statement is still incomplete. }
%\nbwu{This paragraph seems obsolete?}By Propositions \ref{prop:Bias} and \ref{prop:variance_bound}, we have
%\begin{align*}
%\E [(\bar\mu_k - \mu_k)^2] &= |\E[\bar\mu_k] - \mu_k|^2 + \Var(\bar\mu_k)\\
%&\leq \Big(\frac{2Mk}{c_0}\Big)^{5k+2} \Big(\frac{p}{n} \vee 1\Big)^{2k-2} \Delta_{k-1} + \Big(\frac{CMk}{c_0}\Big)^{2k^2}\Big(\frac{p}{n} \vee 1\Big)^{2k-2}\Delta_{k-1}\\
%&\quad + \Big(\frac{CMk}{c_0} \Big)^{4k} \frac{1}{n \wedge p} \Big(\frac{p}{n}\vee 1\Big)^{2k-2}\\
%&\leq \Big(\frac{CMk}{c_0}\Big)^{2k^2}\Big(\frac{p}{n} \vee 1\Big)^{2k-2}\Delta_{k-1} + \Big(\frac{CMk}{c_0} \Big)^{4k} \frac{1}{n \wedge p} \Big(\frac{p}{n}\vee 1\Big)^{2k-2}\\
%&\leq \Delta_k, 
%\end{align*}
%where the last step follows from the expression $\Delta_k = \Big( \frac{CMk}{c_0} \Big)^{(k+1)^3}  \frac{1}{n \wedge p}\Big( \frac{p}{n} \vee 1 \Big)^{k^2}$ in (\ref{eq:MSE_centered}). The proof is complete.
\begin{align}\label{eq:bias_var_decomp}
\E[(\bar\mu_k - \mu_k)^2\bm{1}_{\cE_{k-1}}] = \Big(\E[\bm{1}_{\cE_{k-1}}(\bar\mu_k - \mu_k)]\Big)^2 + \Var\Big(\bm{1}_{\cE_{k-1}}(\bar\mu_k - \mu_k)\Big). 
\end{align}
For the bias term in (\ref{eq:bias_var_decomp}), we have
\begin{align*}
\E[\bm{1}_{\cE_{k-1}}\bar\mu_k] &= \frac{1}{A_k}\E\Big[\bm{1}_{\cE_{k-1}}F_k(\bar\y) - \sum_{t=2}^k \sum_{d_1,\ldots,d_t:d_\ell\geq 2}\gamma^{(d_1,\ldots,d_t)}\bar\mu_{d_1}\ldots \bar{\mu}_{d_t}\bm{1}_{\cE_{k-1}}\Big]\\
&= \frac{1}{A_k}\E\Big[F_k(\bar\y) - \sum_{t=2}^k \sum_{d_1,\ldots,d_t:d_\ell\geq 2}\gamma^{(d_1,\ldots,d_t)}\bar\mu_{d_1}\ldots \bar{\mu}_{d_t}\bm{1}_{\cE_{k-1}}\Big] - \frac{1}{A_k}\E[\bm{1}_{\cE_{k-1}^c}F_k(\bar\y)]\\
&\overset{(*)}{=} \mu_k + \frac{1}{A_k}\sum_{t=2}^k \sum_{d_1,\ldots,d_t:d_\ell\geq 2}\gamma^{(d_1,\ldots,d_t)} \big(\mu_{d_1}\ldots\mu_{d_t} - \E[\bar\mu_{d_1}\ldots\bar\mu_{d_t}\bm{1}_{\cE_{k-1}}]\big) - \frac{1}{A_k}\E[\bm{1}_{\cE_{k-1}^c}F_k(\bar\y)]\\
&= \mu_k + \Prob(\cE_{k-1}^c)\cdot\frac{1}{A_k}\sum_{t=2}^k \sum_{d_1,\ldots,d_t:d_\ell\geq 2}\gamma^{(d_1,\ldots,d_t)} \mu_{d_1}\ldots\mu_{d_t} - \frac{1}{A_k}\E[\bm{1}_{\cE_{k-1}^c}F_k(\bar\y)]\\
&\quad + \frac{1}{A_k}\sum_{t=2}^k \sum_{d_1,\ldots,d_t:d_\ell\geq 2}\gamma^{(d_1,\ldots,d_t)}  \E[(\mu_{d_1}\ldots\mu_{d_t} - \bar\mu_{d_1}\ldots\bar\mu_{d_t})\bm{1}_{\cE_{k-1}}],
\end{align*}
where $(\ast)$ follows from the expression of $\E[F_k(\bar\y)]$ in Lemma \ref{prop:Bias}. This further implies that
\begin{align*}
\E[\bm{1}_{\cE_{k-1}}(\bar\mu_k - \mu_k)] &= \underbrace{\Prob(\cE_{k-1}^c)\cdot\frac{1}{A_k}\sum_{t=1}^k \sum_{d_1,\ldots,d_t:d_\ell\geq 2}\gamma^{(d_1,\ldots,d_t)} \mu_{d_1}\ldots\mu_{d_t}}_{(I)} - \underbrace{\frac{1}{A_k}\E[\bm{1}_{\cE_{k-1}^c}F_k(\bar\y)]}_{(II)}\\
&\quad + \underbrace{\frac{1}{A_k}\sum_{t=2}^k \sum_{d_1,\ldots,d_t:d_\ell\geq 2}\gamma^{(d_1,\ldots,d_t)}  \E[(\mu_{d_1}\ldots\mu_{d_t} - \bar\mu_{d_1}\ldots\bar\mu_{d_t})\bm{1}_{\cE_{k-1}}]}_{(III)},
\end{align*}
and hence
\begin{align}\label{eq:main_bias}
\big(\E[\bm{1}_{\cE_{k-1}}(\bar\mu_k - \mu_k)]\big)^2 \leq 3((I)^2 + (II)^2 + (III)^2).
\end{align}
For $(I)$, by the induction hypothesis \prettyref{eq:induction-hyp2} we have $\Prob(\cE_{k-1}^c)\leq \Delta_{k-1} \leq \sqrt{\Delta_{k-1}}$. 
%\nbwu{If the induction hypothesis is indeed $\Prob(\cE_{k-1}^c)  \leq \Delta_{k-1}$, this below is not right for example $M$ is missing.
%}
Applying the upper bound of $|\gamma^{(d_1,\ldots,d_t)}|$ in Lemma \ref{lem:gamma}, we have 
\begin{align*}
|\sum_{t=1}^k \sum_{d_1,\ldots,d_t:d_\ell\geq 2} \gamma^{(d_1,\ldots,d_t)}\mu_{d_1}\ldots\mu_{d_t}| \leq C^kk^{3k/2}(n\wedge p)\pnorm{\X}{\op}^{2k}.
\end{align*}
%\zf{Missing factor $C^k$ from $\mu_d$? Also, here and throughout the proof below, I'm not sure where the various $k^k$ factors are coming from. Here, the number of ordered partitions of $k$ into positive integers $d_1,\ldots,d_t$ is $2^{k-1}$, and $\mu_{d_1}\ldots \mu_{d_t}$ gives $k^{k/2}$, so why is this not $k^{3k/2}$?}. 
So using the $A_k$ lower bound in Assumption \ref{assump:design} yields
\begin{align*}
(I) \leq \frac{\sqrt{\Delta_{k-1}} \cdot k^{3k/2}(n\wedge p)\pnorm{\X}{\op}^{2k}}{c_0^k  \pnorm{\X}{\op}^{2k}p \big(\frac{n}{p}\wedge 1\big)^k} \leq \frac{1}{3}\sqrt{\Delta_k}
\end{align*}
provided $C_0=C_0(c_0,M)$ is sufficiently large. 
% (In this case $C_0>1/c_0$ works; below we do not spell out the conditions explicitly.)
For $(II)$, we have $\E[\bm{1}_{\cE_{k-1}^c}F_k(\bar\y)] \leq \sqrt{\Prob(\cE_{k-1}^c) \E[F_k^2(\bar\y)]}$, so using Lemma \ref{lem:estimate_L2}, a similar bound as above yields %\zf{Lemma \ref{lem:estimate_L2} gives $(C_1k)^{3k/2}$ below?}
\begin{align*}
(II) \leq \frac{\sqrt{\Prob(\cE_{k-1}^c) \E[F_k^2(\bar\y)]}}{A_k} \leq \frac{\sqrt{\Delta_{k-1}}(C_1 k)^{3k/2}p\pnorm{\X}{\op}^{2k}}{c_0^k  \pnorm{\X}{\op}^{2k}p \big(\frac{n}{p}\wedge 1\big)^k}\leq \frac{1}{3}\sqrt{\Delta_k},
\end{align*}
provided that $C_0$ is sufficiently large.
The term $(III)$ can be telescoped as
\begin{align*}
\E\Big[(\bar\mu_{d_1} \cdots \bar\mu_{d_t} - \mu_{d_1} \cdots \mu_{d_t})\bm{1}_{\cE_{k-1}}\Big] = \E \Big[ \sum_{j=1}^t \bar\mu_{d_1} \cdots \bar\mu_{d_{j-1}} (\bar\mu_{d_j} - \mu_{d_j}) \mu_{d_{j+1}} \cdots \mu_{d_t} \bm{1}_{\cE_{k-1}}\Big].
\end{align*}
On the event $\cE_{k-1}$, we have $\bar\mu_{d_j}\leq 2K_{d_j}$ (with $K_d = M^d d^{d/2}$), so %\zf{$\bm{1}_{\cE_{d_j}}$ below show be $d_{j-1}$?}
\begin{align}\label{eq:bias_of_prod}
\notag\Big|\E\Big[(\bar\mu_{d_1} \cdots \bar\mu_{d_t} - \mu_{d_1}\ldots\mu_{d_t})\bm{1}_{\cE_{k-1}}\Big]\Big| &\leq \E \Big[\Big|\sum_{j=1}^t \bar\mu_{d_1} \cdots \bar\mu_{d_{j-1}} (\bar\mu_{d_j} - \mu_{d_j}) \mu_{d_{j+1}} \cdots \mu_{d_t}\Big|\bm{1}_{\cE_{k-1}}\Big]\\
\notag&\leq \E \Big[ \Big(\sum_{j=1}^t (\bar\mu_{d_j} - \mu_{d_j})^2\bm{1}_{\cE_{d_j-1}}\Big)^{1/2} \Big(\sum_{j=1}^t (\bar\mu_{d_1} \cdots \bar\mu_{d_{j-1}} \mu_{d_{j+1}} \cdots \mu_{d_t})^2\bm{1}_{\cE_{k-1}}\Big)^{1/2}\Big]\\
% &\leq (t 2^{2t} M^{2k})^{\frac{1}{2}} \EE \qth{\pth{\sum_{j=1}^t (\tm_{d_j} - m_{d_j})^2}^{\frac{1}{2}}}\\
\notag&\leq (t 2^{2t}M^{2k}k^k)^{\frac{1}{2}} \Big(\E\sum_{j=1}^t (\bar\mu_{d_j} - \mu_{d_j})^2\bm{1}_{\cE_{d_j-1}}\Big)^{1/2}\\
% &\leq \sqrt{t} 2^{t} M^{k} \pth{\EE \qth{\sum_{j=1}^t (\checkm_{d_j} - m_{d_j})^2}}^{1/2}\\
%\notag&\leq \sqrt{t} 2^tM^kk^{k/2} \Big(\E\sum_{j=1}^t (\bar\mu_{d_j} - \mu_{d_j})^2\bm{1}_{\cE_{d_j-1}}\Big)^{1/2}\\
% &\leq t 2^{t} M^{k} \Delta_{k-1}^{\frac{1}{2}}.
\notag&\leq  \sqrt{t} 2^tM^kk^{k/2} \Big(\sum_{j=1}^t \Delta_{d_j} \Big)^{1/2}\\
&\leq t 2^{t}M^kk^{k/2} \Delta_{k-1}^{\frac{1}{2}},
\end{align}
where the penultimate step applies the induction hypothesis \prettyref{eq:induction-hyp1}.
Combining  this with the bound on $\gamma^{(d_1,\ldots,d_t)}$ in Lemma \ref{lem:gamma}, we have
\begin{align*}
(III) &\leq \frac{1}{A_k}\sum_{t=2}^k \sum_{d_1,\ldots,d_t:d_\ell\geq 2}|\gamma^{(d_1,\ldots,d_t)}|  \Big|\E[(\mu_{d_1}\ldots\mu_{d_t} - \bar\mu_{d_1}\ldots\bar\mu_{d_t})\bm{1}_{\cE_{k-1}}]\Big|\\
% &\leq \frac{1}{|A_k|} k^k (k^k p \opnorm{X}^{2k}) (k 2^{k} M^{k} \Delta_{k-1}^{\frac{1}{2}})
\notag&\leq \frac{1}{A_k} 2^k(k^k (n \wedge p) \opnorm{\X}^{2k}) (k 2^{k} M^kk^{k/2}\Delta_{k-1}^{\frac{1}{2}})\\
&= \frac{1}{A_k} C^k k^{\frac{3k}{2}+1} (n \wedge p) \opnorm{\X}^{2k} \Delta_{k-1}^{\frac{1}{2}} \leq \frac{1}{3}\sqrt{\Delta_k},
\end{align*}
provided that $C_0$ is sufficiently large.
Plugging the bounds of $(I)$-$(III)$ into (\ref{eq:main_bias}) yields
\begin{align}\label{eq:truncate_bias}
\big(\E[\bm{1}_{\cE_{k-1}}(\bar\mu_k - \mu_k)]\big)^2 \leq \frac{1}{2}\Delta_k.
\end{align}

Next we bound the variance term $\Var\big(\bm{1}_{\cE_{k-1}}(\bar\mu_k - \mu_k)\big)$ in (\ref{eq:bias_var_decomp}). First we have $\Var\big(\bm{1}_{\cE_{k-1}}(\bar\mu_k - \mu_k)\big) \leq 2 \Var\big(\bm{1}_{\cE_{k-1}}\bar\mu_k\big) + 2 \Var\big(\bm{1}_{\cE_{k-1}}\mu_k\big)$. Applying the induction hypothesis \prettyref{eq:induction-hyp2}, we have
\begin{align}\label{eq:var_bound_1}
\Var\big(\bm{1}_{\cE_{k-1}}\mu_k\big) \leq \mu_k^2 \Prob(\cE_{k-1}^c) \leq M^{2k}k^k \Delta_{k-1} \leq \frac{1}{10}\Delta_k,
\end{align}
with the last inequality due to  $C_0$ being sufficiently large.
Next we have
\begin{align*}
\Var\big(\bm{1}_{\cE_{k-1}}\bar\mu_k\big) &= \Var\Big(\frac{\bm{1}_{\cE_{k-1}}}{A_k}\Big(F_k(\bar\y) - \sum_{t=2}^k \sum_{d_1+\ldots+d_t=k:d_\ell\geq 2}\gamma^{(d_1,\ldots,d_t)}\bar\mu_{d_1}\ldots\bar\mu_{d_t}\Big)\Big)\\
&\leq 2\underbrace{A_k^{-2}\Var(\bm{1}_{\cE_{k-1}}F_k(\bar\y))}_{(I)} + 2\underbrace{A_k^{-2}\Var\Big(\bm{1}_{\cE_{k-1}}\sum_{t=2}^k \sum_{d_1+\ldots+d_t=k:d_\ell\geq 2}\gamma^{(d_1,\ldots,d_t)}\bar\mu_{d_1}\ldots\bar\mu_{d_t}\Big)}_{(II)}.
\end{align*}
To bound $(I)$, note that for any random variable $X$ and event $\cE$, we have
\begin{align*}
\Var(\bm{1}_{\cE}X) &= \E\big[\big(\bm{1}_{\cE}X - \E[\bm{1}_{\cE} X]\big)^2\big]\\
&\leq 2\E\Big[\bm{1}_{\cE} (X - \E[X])^2\Big] + 2\E\Big[\big(\bm{1}_{\cE} \E[X] - \E[\bm{1}_{\cE}X]\big)^2\Big]\\
&\leq 2\E\Big[(X - \E[X])^2\Big] + 4\E\Big[(\bm{1}_{\cE} - \Prob(\cE))^2(\E[X])^2\Big] + 4 \big(\E[\bm{1}_{\cE}X] - \Prob(\cE)\E[X]\big)^2\\
&= 2\Var(X) + 4\Var(\bm{1}_{\cE})(\E[X])^2 + 4 \Cov(\bm{1}_{\cE}, X)^2\\
&\leq 2\Var(X) + 4\Var(\bm{1}_{\cE})(\E[X])^2 + 4\Var(\bm{1}_{\cE})\Var(X).
\end{align*}
Applying this fact yields
\begin{align*}
\Var(\bm{1}_{\cE_{k-1}}F_k(\bar\y)) \leq 2\Var(F_k(\bar\y)) + 8\Prob(\cE_{k-1}^c)\E[F_k^2(\bar\y)].
\end{align*}
By Lemma \ref{lem:variance_bound} and the lower bound of $A_k$ in Assumption \ref{assump:design}, we have
\begin{align*}
\frac{\Var(F_k(\bar\y))}{A_k^2} \leq \frac{1}{500}\Delta_k.
\end{align*}
Next using the induction hypothesis $\Prob(\cE^c_{k-1})\leq \Delta_{k-1}$ and the bound of $\E[F_k^2(\bar\y)]$ in Lemma \ref{lem:estimate_L2}, we have
\begin{align*}
\frac{\Prob(\cE^c_{k-1})\E[F_k^2(\bar\y)]}{A_k^2} \leq \frac{1}{500}\Delta_k.
\end{align*}
Combining the two estimates yields $(I) \leq \frac{1}{10}\Delta_k$. To bound $(II)$, note that $|\bar\mu_{d_\ell}\bm{1}_{\cE_{k-1}}| \leq |\mu_{d_\ell}| + K_{d_\ell} \leq 2K_{d_\ell}$, so by Lemma \ref{lem:Product_Variance},
\begin{align*}
\Var(\bm{1}_{\cE_{k-1}}\bar\mu_{d_1} \cdots \bar\mu_{d_t}) &\leq 2^{3k}M^{2k}k^{k} \max_{1\leq j\leq t} \Var(\bar\mu_{d_j}\bm{1}_{\cE_{k-1}}).
\end{align*}
For each $1\leq j\leq t$, the induction hypothesis yields that
\begin{align*}
\Var(\bar\mu_{d_j}\bm{1}_{\cE_{k-1}}) &\leq \E[(\bar\mu_{d_j}\bm{1}_{\cE_{k-1}} - \mu_{d_j})^2] \leq 2\E[\bm{1}_{\cE_{k-1}}(\bar\mu_{d_j} - \mu_{d_j})^2] + 2\Prob(\cE_{k-1}^c)\mu_{d_j}^2\\
&\leq 2\Delta_{d_j} + 2\Delta_{k-1}K_{d_j}^2 \leq 4\Delta_{k-1}(M^{k-1}(k-1)^{k-1/2})^2.
\end{align*}
Consequently, by the lower bound on $A_k$ in Assumption \ref{assump:design} and the upper bound of $|\gamma^{(d_1,\ldots,d_t)}|$ in Lemma \ref{lem:gamma}, we have
\begin{align*}
(II) &\leq \frac{k^k}{A_k^2} \sum_{t=2}^k \sum_{d_1+\cdots+d_t = k, d_j \geq 2} |\gamma^{(d_1,\dots,d_t)}|^2 \Var\pth{\bm{1}_{\cE_{k-1}}\bar\mu_{d_1} \cdots \bar\mu_{d_t}} \\
% &\leq c^{-k} \qth{\pth{\frac{p}{n}}^{2k} \vee 1} k^{4k+1} 8^{k-1} M^{2(k-1)} \Delta_{k-1}
&\leq \Big(\frac{8}{c_0^2}\Big)^{k}k^{5k} M^{4k} \Big(\frac{p}{n} \vee 1\Big)^{2k-2}  \Delta_{k-1}\\
&\leq \Big(\frac{CMk}{c_0}\Big)^{5k}\Big(\frac{p}{n} \vee 1\Big)^{2k-2}\Delta_{k-1} \leq \frac{1}{20}\Delta_k,
\end{align*}
where the last step holds for  sufficiently large   $C_0$.
Combining the bounds of $(I)$ and $(II)$ yields that
\begin{align*}
\Var(\bm{1}_{\cE_{k-1}}\bar\mu_k) \leq \frac{1}{10}\Delta_k,
\end{align*}
which, along with (\ref{eq:var_bound_1}), yields that
\begin{align*}
\Var\big(\bm{1}_{\cE_{k-1}}(\bar\mu_k - \mu_k)\big) \leq \frac{1}{2}\Delta_k.
\end{align*}
Together with the bias bound (\ref{eq:truncate_bias}), we have $\E[\bm{1}_{\cE_{k-1}}(\bar\mu_k - \mu_k)^2] \leq \Delta_k$, establishing the induction for (\ref{eq:induction-hyp1}).

Finally we lower bound the probability $\Prob(\cE_k)$. By definition, we have
\begin{align*}
\Prob(\cE_{k-1}) - \Prob(\cE_k) &= \Prob(\cE_{k-1}) - \Prob\Big(\cE_{k-1} \cap |\bar\mu_k - \mu_k| \leq M^kk^{k/2}\Big)\\
&=\Prob\Big(\cE_{k-1} \cap |\bar\mu_k - \mu_k| > M^kk^{k/2}\Big)\\
&= \Prob\Big(|\bar\mu_k - \mu_k|\bm{1}_{\cE_{k-1}} > M^kk^{k/2}\Big)\\
&\leq \frac{1}{M^{2k}k^k}\E\big[(\bar\mu_k - \mu_k)^2\bm{1}_{\cE_{k-1}}\big] \leq \frac{\Delta_k}{M^{2k}k^k},
\end{align*}
which, along with the induction hypothesis $\Prob(\cE_{k-1})\geq 1-\Delta_{k-1}$, implies
\begin{align*}
\Prob(\cE_k) \geq \Prob(\cE_{k-1}) - \frac{\Delta_k}{M^{2k}k^k} \geq 1 - \Delta_{k-1} - \frac{\Delta_k}{M^{2k}k^k} \geq 1 - \Delta_k.
\end{align*}
Here the last step follows from the definition of $\Delta_k$ and choosing 
  $C_0$ sufficiently large.
This concludes the induction for (\ref{eq:induction-hyp2}) and the proof of the proposition.
\end{proof}

\subsubsection{Perturbation analysis for debiased model}\label{subsection:perturbation_noncentered}

Now we focus on the difference between the moment estimators of the approximately centered model $\hat{\y} := \X \hat{\bbeta} + \beps $ and the true centered model $\bar{\y} := \X \bar{\bbeta} + \beps$, where $ \hat{\bbeta} := \bbeta - \hat{m}_1 \bm{1}_p $ with $\hat{m}_1$ given in (\ref{eq:m1hat}), and $\bar{\bbeta} = \bbeta - m_1 \bm{1}_p$. The following result bounds the mean squared difference between the true estimator $\hat\mu_k$ in (\ref{eq:centered_moment_est}) and the oracle estimator $\bar\mu_k$ in (\ref{eq:centered_moment_ideal}) truncated on a high-probability event $\cF_{k-1}$ defined as follows:  
Let $\cF_1$ denote the entire sample space and for $k \geq 2$, 
\begin{align*}
\cF_k := \{|\hat\mu_\ell - \mu_\ell| \leq K_\ell, |\bar\mu_\ell - \mu_\ell| \leq K_\ell, \ell = 2,\ldots,k\}.
\end{align*}

% Also let $\cF_1$ denote the entire space (so that $\cF_1$ always holds).

\begin{proposition}\label{prop:perturbation_bound}
Suppose that Assumption \ref{assump:design} holds for some $c_0 > 0$ up to some integer $k_0\geq 2$, and Assumption \ref{assump:prior} holds for some $M > 0$. Then for any $2\leq k\leq k_0$, it holds that 
\begin{align*}
\E[(\hat\mu_k - \bar\mu_k)^2\bm{1}_{\cF_{k-1}}] \leq \Delta_k,
\end{align*}
where $\Delta_k$ is given by (\ref{eq:MSE_centered}). Moreover, event $\cF_k$ holds with probability at least $1-\Delta_k$.
%\nbwu{Missing the lower bound on probability of event?}
%\begin{equation}
%\delta_\ell = \Big(\frac{CM\ell}{c_0^2}\Big)^{3\ell^2} \frac{1}{n \wedge p} \Big(\frac{p}{n} \vee 1\Big)^{2\ell^2}
%\end{equation}
%for some universal $C > 0$.
\end{proposition}

\begin{proof}
We will prove via induction the following two hypotheses: 
\begin{align}
\E [(\hat\mu_k - \bar\mu_k)^2 \Indc_{\cF_{k-1}}] & \leq \Delta_k,
    \label{eq:induction-hyp3}\\
\Prob(\cF_k^c) & \leq \Delta_k,
    \label{eq:induction-hyp4}
\end{align}
for all $2\leq k\leq k_0$, with $k=2$ being the base case.

The base case $k=2$ follows from Proposition \ref{prop:MSE_first_two}. Assuming \prettyref{eq:induction-hyp3}--\prettyref{eq:induction-hyp4}
hold for $2,\ldots,k-1$ with $k\geq 3$, we will bound separately the two terms 
\begin{align*}
\frac{\bm{1}_{\cF_{k-1}}}{A_k}(F_k(\hat\y) - F_k(\bar\y)) \quad\text{ and } \quad\frac{\bm{1}_{\cF_{k-1}}}{A_k} \sum_{t=2}^k \sum_{d_1 + \ldots + d_t = k: d_j \geq 2} \gamma^{{(d_1,\dots,d_t)}} (\hat{\mu}_{d_1} \ldots \hat{\mu}_{d_t} - \bar\mu_{d_1} \dots \bar\mu_{d_t}),
\end{align*}
where we recall $\gamma^{(d_1,\ldots,d_t)} = \sum_{s_1,\ldots,s_t=1}^p \tilde T^{(d_1,\ldots,d_t)}_{s_1,\ldots,s_t}$.

\paragraph{Bounding the first term} Recall that
$F_k(\hat{\y}) := \sum_{j=1}^p (\sigma\pnorm{\x_j}{})^k H_k\Big(\frac{\x_j^\top \hat{\y}}{\sigma\pnorm{\x_j}{}}\Big)$. Denoting $ \u_j := \x_j / \|\x_j\| $ and $ \delta := |\hat{m}_1 - m_1| $, we have
\begin{align*}
\frac{1}{A_k} |F_k(\hat{\y}) - F_k(\bar{\y})| &\leq \frac{1}{A_k} \sum_{j=1}^p \sigma^k \|\x_j\|^k \left| H_k(\u_j^\top \hat{\y} / \sigma) - H_k(\u_j^\top \bar{\y} / \sigma) \right|\\
&\leq \frac{1}{A_k} \sum_{j=1}^p \sigma^k \|\x_j\|^k |H_k'(\xi_j)| |\u_j^\top (\hat{\y} - \bar{\y}) / \sigma|\\
&= \frac{\delta}{A_k} \sum_{j=1}^p \sigma^{k-1} \|\x_j\|^k |H_k'(\xi_j)| |\u_j^\top \X \bm{1}|,
\end{align*}
for some $ \xi_j $ between $ \u_j^\top \hat{\y} / \sigma $ and $ \u_j^\top \bar{\y} / \sigma $.
Note that $ H_k'(x) = \sum_{\ell=0}^{k-1} a_{k,\ell} x^\ell$ is a polynomial of degree $k-1 $. For each term $x^\ell$ with $ x \in [a,b] $, we have $ |x|^\ell \leq |a|^\ell + |b|^\ell$. Then we have
\begin{align*}
\frac{1}{A_k} |F_k(\hat{\y}) - F_k(\bar{\y})| &\leq \frac{\delta}{A_k} \sum_{j=1}^p \sigma^{k-1} \|\x_j\|^k \sum_{\ell=0}^{k-1} |a_{k,\ell}| \Big(\Big|\frac{\u_j^\top \hat{\y}}{\sigma}\Big|^\ell + \Big|\frac{\u_j^\top \bar{\y}}{\sigma}\Big|^\ell\Big) |\u_j^\top \X \bm{1}|\\
&\leq \frac{\delta}{A_k} \sum_{\ell=0}^{k-1} |a_{k,\ell}| C^\ell \sigma^{k-\ell-1} \sum_{j=1}^p \|\x_j\|^{k-\ell-1} \Big(|\x_j^\top \X \hat\bbeta|^\ell + |\x_j^\top \X \bar{\bbeta}|^\ell + |\x_j^\top \beps|^\ell\Big) |\x_j^\top \X \bm{1}|\\
&= \frac{\delta}{A_k} \sum_{\ell=0}^{k-1} |a_{k,\ell}| C^\ell \sigma^{k-\ell-1} \sum_{j=1}^p \|\x_j\|^{k-\ell-1} \Big(|\x_j^\top \X (\hat\bbeta - \bar\bbeta +\bar\bbeta)|^\ell + |\x_j^\top \X \bar\bbeta|^\ell + |\x_j^\top \beps|^\ell\Big) |\x_j^\top \X \bm{1}|\\
&\leq \frac{\delta}{A_k} \sum_{\ell=0}^{k-1} |a_{k,\ell}| C^\ell \sigma^{k-\ell-1} \sum_{j=1}^p \|\x_j\|^{k-\ell-1} \Big(|\x_j^\top \X \bar\bbeta|^\ell + \delta^\ell|\x_j^\top \X \bm{1}|^\ell + |\x_j^\top \beps|^\ell\Big) |\x_j^\top \X \bm{1}|.
\end{align*}
Consequently, using $\max_{0\leq \ell\leq k-1}|a_{k,\ell}|\leq (Ck)^k$ for some universal $C > 0$,
which follows from the explicit formula of Hermite polynomial  (\ref{eq:hermite_explicit}) and the fact that $H_k'(x) = kH_{k-1}(x)$,
we have
\begin{align*}
&\E[(F_k(\hat\y) - F_k(\bar\y))^2]\\
&\leq \E\Big[\delta^2 \Big(\sum_{\ell=0}^{k-1}C^\ell (Ck)^k(\sigma\pnorm{\X}{\op})^{k-\ell-1}\Big(\sum_{j=1}^p \big(|\x_j^\top \X \bar\bbeta|^\ell + \delta^\ell|\x_j^\top \X \bm{1}|^\ell + |\x_j^\top \beps|^\ell\big) |\x_j^\top \X \bm{1}|\Big)\Big)^2\Big]\\
&\leq (Ck)^{2k}k \E\Big[\delta^2 \sum_{\ell=0}^{k-1} (\sigma\pnorm{\X}{\op})^{2(k-\ell-1)}\Big(\sum_{j=1}^p \big(|\x_j^\top \X \bar\bbeta|^\ell + \delta^\ell|\x_j^\top \X \bm{1}|^\ell + |\x_j^\top \beps|^\ell\big) |\x_j^\top \X \bm{1}|\Big)^2\Big]\\
&\leq (Ck)^{2k}k\sum_{\ell=0}^{k-1} (\sigma\pnorm{\X}{\op})^{2(k-\ell-1)} \bigg\{\underbrace{\E\Big[\delta^2 \Big(\sum_{j=1}^p |\x_j^\top \X\bar\bbeta|^\ell|\x_j^\top \X\bm{1}|\Big)^2\Big]}_{(I_\ell)} + 
\underbrace{\E\Big[\delta^{2\ell+2} \Big(\sum_{j=1}^p |\x_j^\top \X\bm{1}|^{\ell+1}|\Big)^2\Big]}_{(II_\ell)}\\
&\quad + \underbrace{\E\Big[\delta^2 \Big(\sum_{j=1}^p |\x_j^\top \beps|^\ell|\x_j^\top \X\bm{1}|\Big)^2\Big]}_{(III_\ell)}\bigg\}.
\end{align*}

To bound $(I_\ell)$, we have
\begin{align*}
(I_\ell)^2 &\leq \E[\delta^4]\E\Big[\Big(\sum_{j=1}^p |\x_j^\top \X\bar\bbeta|^\ell|\x_j^\top \X\bm{1}|\Big)^4\Big] \leq \E[\delta^4]\E\Big[\Big(\sum_{j=1}^p |\x_j^\top \X\bar\bbeta|^{2\ell}\Big)^2\Big(\sum_{j=1}^p (\x_j^\top \X\bm{1})^2\Big)^2\Big]\\
&= \E[\delta^4]\pnorm{\X^\top\X\bm{1}}{}^4\E\Big[\Big(\sum_{j=1}^p |\x_j^\top \X\bar\bbeta|^{2\ell}\Big)^2\Big] \leq p\E[\delta^4]\pnorm{\X^\top\X\bm{1}}{}^4 \E\Big[\sum_{j=1}^p (\x_j^\top \X\bar\bbeta)^{4\ell}\Big]\\
&\leq (C\ell)^{2\ell}M^{4\ell}p\E[\delta^4]\pnorm{\X^\top\X\bm{1}}{}^4\sum_{j=1}^p \pnorm{\X^\top \x_j}{}^{4\ell} \leq (C\ell)^{2\ell}M^{4\ell}p^2\pnorm{\X}{\op}^{8\ell+4}\pnorm{\X\bm{1}}{}^4 \E[\delta^4].
\end{align*}
Using the definition $\delta = \hat m_1 - m_1 = \frac{\bm{1}^\top \X^\top \bar\y}{\pnorm{\X\bm{1}}{}^2}$ and Assumption \ref{assump:prior}, we have 
\begin{align*}
\E[\delta^4] \leq C\frac{\E (\bm{1}^\top\X^\top\X\bar\bbeta)^4 + \E(\bm{1}^\top \X^\top\beps)^4}{\pnorm{\X\bm{1}}{}^8} \leq C\frac{\pnorm{\X}{\op}^4}{\pnorm{\X\bm{1}}{}^4}\Big(M+\frac{\sigma}{\pnorm{\X}{\op}}\Big)^4,
\end{align*}
which further implies 
\begin{align*}
(I_\ell)^2 \leq (C\ell)^{2\ell}M^{4\ell}p^2\pnorm{\X}{\op}^{8\ell+8}\Big(M+\frac{\sigma}{\pnorm{\X}{\op}}\Big)^4,
\end{align*}
and hence $(I_\ell) \leq (C\ell)^{\ell}M^{2\ell}p\pnorm{\X}{\op}^{4\ell+4}\Big(M+\frac{\sigma}{\pnorm{\X}{\op}}\Big)^2$. A similar consideration yields 
\begin{align*}
(III_\ell) \leq (C\ell)^\ell p\sigma^{2\ell}\pnorm{\X}{\op}^{2\ell+4}\Big(M+\frac{\sigma}{\pnorm{\X}{\op}}\Big)^2.
\end{align*}
Lastly, to bound $(II_\ell)$, we have
\begin{align*}
(II_\ell) &= \E[\delta^{2\ell+2}] \Big(\sum_{j=1}^p |\x_j^\top \X\bm{1}|^{\ell+1}\Big)^2 \leq p\E[\delta^{2\ell+2}] \sum_{j=1}^p (\x_j^\top \X\bm{1})^{2(\ell+1)} \leq p\E[\delta^{2\ell+2}]\pnorm{\X}{\op}^{2\ell}\pnorm{\X\bm{1}}{}^{2\ell} \sum_{j=1}^p (\x_j^\top \X\bm{1})^2\\
&= p\E[\delta^{2\ell+2}]\pnorm{\X}{\op}^{2\ell}\pnorm{\X\bm{1}}{}^{2\ell}\pnorm{\X^\top \X\bm{1}}{}^2 \leq p\E[\delta^{2\ell+2}]\pnorm{\X}{\op}^{2\ell+2}\pnorm{\X\bm{1}}{}^{2\ell+2},
\end{align*}
where
\begin{align*}
\E[\delta^{2\ell+2}] &= \frac{\E (\bm{1}^\top \X^\top\bar \y)^{2\ell+2}}{\pnorm{\X\bm{1}}{}^{4\ell+4}} \leq (C\ell)^{\ell+1}\frac{M^{2\ell+2}\pnorm{\X^\top\X\bm{1}}{}^{2\ell+2} + \sigma^{2\ell+2}\pnorm{\X\bm{1}}{}^{2\ell+2}}{\pnorm{\X\bm{1}}{}^{4\ell+4}}\\ &\leq \frac{(C\ell)^{\ell+1}\pnorm{\X}{\op}^{2\ell+2}\Big(M+\frac{\sigma}{\pnorm{\X}{\op}}\Big)^{2\ell+2}}{\pnorm{\X\bm{1}}{}^{2\ell+2}}.
\end{align*}
Thus we conclude
\begin{align*}
(II_\ell) \leq (C\ell)^{\ell+1}p\pnorm{\X}{\op}^{4\ell+4}\Big(M+\frac{\sigma}{\pnorm{\X}{\op}}\Big)^{2\ell+2}.
\end{align*}
Combining the estimates of $(I_\ell)$-$(III_\ell)$ yields %\zf{below exponent should be $2l+2$ instead of $2\ell$?} 
\begin{align*}
(I_\ell) + (II_\ell) + (III_\ell) \leq (C\ell)^{\ell+1} p\pnorm{\X}{\op}^{4\ell+4}\Big(M + \frac{\sigma}{\pnorm{\X}{\op}}\Big)^{2\ell+2},
\end{align*}
which implies that
\begin{align*}
\E[(F_k(\hat\y) - F_k(\bar\y))^2] &\leq (Ck)^{2k}k\sum_{\ell=0}^{k-1}(\sigma\pnorm{\X}{\op})^{2(k-\ell-1)}(Ck)^{k+1} p\pnorm{\X}{\op}^{4\ell+4}\Big(M + \frac{\sigma}{\pnorm{\X}{\op}}\Big)^{2\ell}\\
&\leq p(Ck)^{3k+2}\pnorm{\X}{\op}^{4k} \sum_{\ell=0}^{k-1} \Big(\frac{\sigma}{\pnorm{\X}{\op}}\Big)^{2(k-\ell-1)}\Big(M+\frac{\sigma}{\pnorm{\X}{\op}}\Big)^{2\ell}\\
&\leq p\Big(\frac{Ck}{c_0}\Big)^{3k+2}\pnorm{\X}{\op}^{4k}\Big(M+\frac{\sigma}{\pnorm{\X}{\op}}\Big)^{2k}\\
&\leq p\Big(\frac{Ck M}{c_0^2}\Big)^{3k+2}\pnorm{\X}{\op}^{4k},
\end{align*}
where we apply Assumption \ref{assump:design}. Hence using the lower bound of $A_k$ in Assumption \ref{assump:design}, we have
\begin{align*}
\E\Big[\Big(\frac{F_k(\hat\y) - F_k(\bar\y)}{A_k}\Big)^2\Big] \leq \Big(\frac{Ck M}{c_0^4}\Big)^{3k+2}\frac{1}{p}\Big(\frac{p}{n}\vee 1\Big)^{2k}.
\end{align*}

\paragraph{Bounding the second term} For the individual difference $ \hat{\mu}_{d_1} \dots \hat{\mu}_{d_t} - \bar\mu_{d_1} \dots \bar\mu_{d_t}$ for $(d_1,\ldots,d_t)$ such that $\sum_{i=1}^t d_i = k$ and $d_i\geq 2$, we have
\begin{align*}
\hat{\mu}_{d_1} \dots \hat{\mu}_{d_t} - \bar\mu_{d_1} \dots \bar\mu_{d_t} = \sum_{i=1}^t \hat{\mu}_{d_1} \cdots \hat{\mu}_{d_{i-1}} (\hat{\mu}_{d_i} - \bar\mu_{d_i}) \bar\mu_{d_{i+1}} \cdots \bar\mu_{d_t},
\end{align*}
where on the event $\cF_{k-1}$ we have $|\hat{\mu}_{d_i}| \vee |\bar\mu_{d_i}|\leq 2 K_{d_i}$ with $ K_d := M^d d^{d/2}$. This implies
\begin{align*}
\E\Big[\bm{1}_{\cF_{k-1}}\big(\hat{\mu}_{d_1} \dots \hat{\mu}_{d_t} - \bar\mu_{d_1} \dots \bar\mu_{d_t}\big)^2\Big] &\leq t \sum_{i=1}^t \E\Big[\bm{1}_{\cF_{k-1}}(\hat{\mu}_{d_1} \cdots \hat{\mu}_{d_{i-1}} (\hat{\mu}_{d_i} - \bar\mu_{d_i}) \bar\mu_{d_{i+1}} \cdots \bar\mu_{d_t})^2\Big]\\
&\leq t 2^{t-1} M^{2k} k^k \sum_{i=1}^t \E[\bm{1}_{\cF_{k-1}}(\hat{\mu}_{d_i} - \bar\mu_{d_i})^2]\\
&\leq t^2 2^{2t} M^{2k} k^k \Delta_{k-1},
\end{align*}
applying the induction hypothesis since $d_i\leq k-1$.
Recall the bound of $ \gamma^{{(d_1,\dots,d_t)}} = \sum_{s_1,\ldots,s_t = 1}^p \tilde{T}^{(d_1,\ldots,d_t)}_{s_1,\ldots,s_t} $ in Lemma \ref{lem:gamma}. Using a similar argument as in (\ref{eq:bias_of_prod}), we obtain
\begin{align*}
&\E\Big[\bm{1}_{\cF_{k-1}}\Big(\frac{1}{A_k} \sum_{t=2}^k \sum_{d_1 + \ldots + d_t = k: d_j \geq 2} \gamma^{{(d_1,\dots,d_t)}} (\hat{\mu}_{d_1} \ldots \hat{\mu}_{d_t} - \bar\mu_{d_1} \dots \bar\mu_{d_t})\Big)^2\Big]\\
&\quad\leq \frac{1}{A_k^2} k^{2k} (k^{2k} (n \wedge p)^2 \opnorm{\X}^{4k}) (k^2 2^{2k} M^{2k} k^k \Delta_{k-1}) \leq \Big( \frac{CM}{c_0} \Big)^{2k} k^{5k+2} \Big( \frac{p}{n} \vee 1 \Big)^{2k-2} \Delta_{k-1}.
% \leq \frac{1}{|A_k|^2} \ell^{2\ell} (\ell^{2\ell} p^2 \opnorm{\X}^{4\ell}) (\ell^2 2^{2\ell} M^{2\ell} \ell^{\ell} \delta_{\ell-1}) = C^\ell M^{2\ell} \ell^{6\ell + 2} \pth{\frac{p}{n} \vee 1}^{2\ell} \delta_{\ell-1}
\end{align*}
Hence, combining the two bounds and using the expression of $\Delta_{k-1}$ with a sufficiently large constant $C_0$, we have %\zf{first term has already been simplified further above}
\begin{align*}
\E[\bm{1}_{\cF_{k-1}}(\hat{\mu}_k - \bar\mu_k)^2] &\leq \Big(\frac{Ck M}{c_0^4}\Big)^{3k+2}\frac{1}{p}\Big(\frac{p}{n}\vee 1\Big)^{2k} + \Big( \frac{CM}{c_0} \Big)^{2k} k^{5k+2} \Big( \frac{p}{n} \vee 1 \Big)^{2k-2} \Delta_{k-1}\\
&\leq \Delta_k,
\end{align*}
concluding the induction for (\ref{eq:induction-hyp3}).

Lastly we prove (\ref{eq:induction-hyp4}). By definition, with $K_k = M^k k^{k/2}$, we have
\begin{align*}
\Prob(\cF_{k-1}) - \Prob(\cF_k) &= \Prob(\cF_{k-1}) - \Prob\Big(\cF_{k-1} \cap |\hat\mu_k - \mu_k| \leq K_k \cap |\bar\mu_k - \mu_k| \leq K_k\Big)\\
&=\Prob\Big(\{\cF_{k-1} \cap |\hat\mu_k - \mu_k| > K_k\} \cup \{\cF_{k-1} \cap |\bar\mu_k - \mu_k| > K_k\} \Big)\\
&\leq \Prob\Big(\cF_{k-1} \cap |\hat\mu_k - \mu_k| > K_k\Big) + \Prob\Big(\cF_{k-1} \cap |\bar\mu_k - \mu_k| > K_k\Big)\\
&= \Prob\Big(\bm{1}_{\cF_{k-1}}|\hat\mu_k - \mu_k| > K_k\Big) + \Prob\Big(\bm{1}_{\cF_{k-1}}|\bar\mu_k - \mu_k| > K_k\Big)\\
&\leq \frac{1}{K_k^2}\Big(\E\big[(\hat\mu_k - \mu_k)^2\bm{1}_{\cF_{k-1}}\big] + \E\big[(\bar\mu_k - \mu_k)^2\bm{1}_{\cF_{k-1}}\big]\Big)\\
&\leq \frac{1}{K_k^2}\Big(2\E\big[(\hat\mu_k - \bar\mu_k)^2\bm{1}_{\cF_{k-1}}\big] + 3\E\big[(\bar\mu_k - \mu_k)^2\bm{1}_{\cF_{k-1}}\big]\Big)  \leq \frac{5\Delta_k}{M^{2k}k^k},
\end{align*}
where in the last step we apply (\ref{eq:induction-hyp3}) and (\ref{eq:induction-hyp1}). Using induction hypothesis $\Prob(\cF_{k-1})\geq 1-\Delta_{k-1}$, we have
\begin{align*}
\Prob(\cF_k) \geq \Prob(\cF_{k-1}) - \frac{5\Delta_k}{M^{2k}k^k} \geq 1 - \Delta_{k-1} - \frac{5\Delta_k}{M^{2k}k^k} \geq 1 - \Delta_k,
\end{align*}
with the last step again holding for sufficiently large $C_0$. This concludes the induction for (\ref{eq:induction-hyp4}).
\end{proof}

\subsubsection{Completing the proof}\label{subsection:complete_noncentered}

Combining Propositions \ref{prop:MSE_centered} and \ref{prop:perturbation_bound} and noting that $\cF_k\subset \cE_k$ we have
%\nbwu{$\bm{1}_{\cF_\ell}$ is $\bm{1}_{\cF_{\ell-1}}$? Also, $\ell\to k$?}
\begin{align*}
\E[\bm{1}_{\cF_{k-1}}(\hat{\mu}_k - \mu_k)^2] &\leq 2 \pth{\E[\bm{1}_{\cF_{k-1}}(\hat{\mu}_k - \bar\mu_k)^2] + \E[\bm{1}_{\cF_{k-1}}(\bar\mu_k - \mu_k)^2]} \leq 4\Delta_k.
\end{align*}
Moreover, we have $\Prob(\cF_{k-1})\geq 1 - \Delta_{k-1}\geq 1-\Delta_k$, as desired.

\subsection{Proof of Lemma \ref{lem:random_lsi}}\label{subsec:proof_assumption}

It is clear that with probability converging to $1$, $\pnorm{\bar \X}{\op} \geq C\sqrt{n+p}$ for some $C = C(\kappa)$, so that $\pnorm{\X}{\op} = \lambda(n,p)\pnorm{\bar\X}{\op} \geq c_0\sigma$ for some $c_0 = c_0(\kappa)$.

Next we verify the lower bound condition on $A_k$. Since this condition is scale invariant, it suffices to verify it for $\bar \X$. Let $\bar \H := \bar \X^\top \bar\X$. We first compute $\E[\sum_{i,j=1}^p \bar H_{ij}^k]$. Note that $\bar H_{ij}^k = \iprod{\bar\X_i}{\bar\X_j}^k = \iprod{\bar\X_i^{\otimes k}}{\bar\X_j^{\otimes k}}$ for $i,j\in[p]$,
%\begin{equation}\label{eq:Assumption_LHS}
%\H_{ij}^k = \iprod{\x_i}{\x_j}^k = \iprod{\x_i^{\otimes k}}{\x_j^{\otimes k}},
%\end{equation}
where $\bar\X_i\in \R^n$ is the $i$th column of $\bar\X$, so
\begin{align*}
\E\Big[\sum_{i,j=1}^p \bar H_{ij}^k\Big] &= \E\iprod{\sum_{i=1}^p \bar\X_i^{\otimes k}}{\sum_{j=1}^p \bar\X_j^{\otimes k}} = \E\biggpnorm{\sum_{i=1}^p \bar\X_i^{\otimes k}}{F}^2 = \sum_{a_1,\dots,a_k =1}^n \E\Big(\sum_{i=1}^p \bar X_{a_1,i} \cdots \bar X_{a_k,i}\Big)^2\\
&\geq \sum_{a_1,\ldots,a_k\in[n]:a_1\neq \ldots \neq a_k}\E\Big(\sum_{i=1}^p \bar X_{a_1,i} \cdots \bar X_{a_k,i}\Big)^2 = n(n-1)\ldots (n-k+1) \sum_{i,j=1}^p \Sigma_{ij}^k\\
&\geq (n-k+1)^k \bm{1}^\top \bSigma^{\circ k}\bm{1}
\geq (n-k+1)^k p \lambda_{\min}(\bSigma^{\circ k}) \geq (n-k+1)^k p \kappa^k,
\end{align*}
where $\bSigma^{\circ k}$ denotes the $k$th Hadamard product of $\bSigma$, and we use the fact that $\lambda_{\min}(\bSigma^{\circ k}) \geq (\lambda_{\min}(\bSigma))^k \geq \kappa^k$. (This fact on Hadamard product can be proved as follows: Let $\{\g^{(m)}\}_{m=1}^k$ be $k$ copies of a random vector $\g$ such that $\E [\g\g^\top] = \Sigma$, then $\Sigma_{ij}^k = \E [\prod_{m=1}^k g^{(m)}_i g^{(m)}_j]$, and hence for any unit norm $\v$, 
\begin{align*}
\v^\top \bSigma^{\circ k}\v &= \E \Big(\sum_{i=1}^p v_i a_ig^{(k)}_i\Big)^2 = \E \sum_{i,j=1}^p v_iv_ja_ia_j \Sigma_{ij} = \E (\v\circ \a)^\top \bSigma (\v\circ \a)\\
&\geq \kappa \E \pnorm{\v\circ \a}{}^2 = \kappa \sum_{i=1}^p v_i^2 \Sigma_{ii}^{k-1} \geq \ldots \geq \kappa^k,
\end{align*}
where $\a = (a_1,\ldots,a_p)$ with $a_i := \prod_{m=1}^{k-1} g^{(m)}_i$.) Then using $(n-k+1)^k = n^k(1-\frac{k-1}{n})^k \geq n^k/2^k$ for $k\leq n/2$, we have 
\begin{align*}
\E\Big[\sum_{i,j=1}^p \bar H_{ij}^k\Big] \geq (\kappa/2)^k pn^k.
\end{align*} 

Next we bound $\Var(\sum_{i,j=1}^p \bar H_{ij}^k)$. Let $f(\bar\X) = \sum_{i,j=1}^p (\bar\X^\top\bar\X)_{ij}^k$. Then for any $\alpha\in[n]$ and $\beta \in [p]$, using 
\begin{align*}
\frac{\partial}{\partial \bar X_{\alpha\beta}} (\bar\X^\top \bar\X)_{ij} = \bar X_{\alpha j}\bm{1}_{i = \beta} + \bar X_{\alpha i}\bm{1}_{j = \beta},
\end{align*}
we have
\begin{align*}
\frac{\partial f(\bar\X)}{\partial \bar X_{\alpha\beta}} &= k\sum_{i,j=1}^p (\bar\X^\top \bar\X)_{ij}^{k-1} \partial_{\alpha\beta}[(\bar\X^\top \bar\X)_{ij}]= 2k\sum_{i,j=1}^p (\bar\X^\top \bar\X)_{ij}^{k-1}\bar X_{\alpha j}\bm{1}_{i = \beta} = 2k \Big[\bar\X (\bar\X^\top \bar\X)^{\circ (k-1)}\Big]_{\alpha\beta}. 
\end{align*}
By tensorization \cite[Proposition 4.3.1]{bakry2014analysis}, the distribution of $\bar \X$ in $\R^{n\times p}$ satisfies the Poincar\'e inequality with the same constant, i.e., 
\begin{align*}
\Var\big[f(\bar\X)\big] \leq C_{\PI}\E \pnorm{\nabla f(\bar\X)}{F}^2,
\end{align*}
for all absolutely continuous $f:\R^{n\times p}\rightarrow\R$. This implies
\begin{align*}
\Var\Big(\sum_{i,j=1}^p (\bar\X^\top \bar\X)_{ij}^k\Big) &\leq 4C_{\PI}k^2 \E \Bigpnorm{\bar\X (\bar\X^\top \bar\X)^{\circ (k-1)}}{F}^2\\
&\leq 4C_{\PI}k^2 \E \pnorm{\bar\X}{F}^2 \pnorm{\bar\X}{\op}^{4(k-1)}\\
&\leq 4C_{\PI}k^2  \E^{1/2} \pnorm{\bar\X}{F}^4 \cdot \E^{1/2} \pnorm{\bar\X}{\op}^{8(k-1)}\\
&\overset{(*)}{\leq} (Ck^2)^k   (np) \cdot (\sqrt{n} + \sqrt{p})^{4(k-1)}\\
&\leq (Ck^2)^k  (n+p)^{2k},
\end{align*}
for some $C > 0$ that depends on $(\kappa,C_{\PI},C_0)$, where $(\ast)$ follows from the bound of $\big(\E[\pnorm{\bar\X}{\op}^r]\big)^{1/r} \leq C(\sqrt{n} + \sqrt{p} + \sqrt{r})$ with $r = 8(k-1) $ by integrating the tail in \cite[Theorem 4.6.1]{vershynin2018high}
and $\E\pnorm{\bar\X}{F}^4 \leq C(np)^2$ via a further application of Poincar\'e inequality to control $\Var(\pnorm{\bar\X}{F}^2)$.

Let $\mathcal{E}_k$ denote the event that $\sum_{i,j=1}^p \bar H_{ij}^k \geq \E[\sum_{i,j=1}^p \bar H_{ij}^k]/2$. Then applying Chebyshev inequality yields
\begin{align*}
\Prob(\mathcal{E}_k^c) \leq \Prob\Big(\Big|\sum_{i,j=1}^p \bar H_{ij}^k - \E\Big[\sum_{i,j=1}^p \bar H_{ij}^k\Big]\Big|\geq \frac{1}{2}\E\Big[\sum_{i,j=1}^p \bar H_{ij}^k\Big]\Big) \leq \frac{4\Var(\sum_{i,j=1}^p \bar H_{ij}^k)}{(\E[\sum_{i,j=1}^p \bar H_{ij}^k])^2} \leq \frac{1}{p^2}(Ck^2)^k \Big(1+\frac{p}{n}\Big)^{2k}. 
\end{align*}
Taking the union bound yields that
\begin{align*}
\Prob((\cap_{k=1}^{k_0} \mathcal{E}_k)^c) \leq \sum_{k=1}^{k_0}\frac{1}{p^2}(Ck^2)^k \Big(1+\frac{p}{n}\Big)^{2k} \leq \frac{1}{p^2} (Ck_0^2)^{k_0}\Big[C\Big(1+\frac{p}{n}\Big)\Big]^{2k_0+2} \rightarrow 0,
\end{align*}
where we apply $n\geq p^{1-r}$ and $k_0 \leq c(r^{-1}\wedge \frac{\log p}{\log\log p})$ for some small $c>0$. On the other hand, define the event $\cE = \{\pnorm{\bar\X}{\op} \leq C(\sqrt{n} + \sqrt{p})\}$ which holds with probability converging to $1$ \cite[Theorem 4.6.1]{vershynin2018high}. On this event, it holds that
\begin{align*}
\pnorm{\bar\X}{\op}^{2k} \cdot p\Big(\frac{n}{p}\wedge 1\Big)^k \leq C^k\Big(n\vee p\Big)^k \cdot p\Big(\frac{n}{p}\wedge 1\Big)^k = C^k pn^k.
\end{align*}
Hence we conclude that on the event $\cap_{k=1}^{k_0}\mathcal{E}_k \cap \mathcal{E}$, which holds with probability converging to $1$, we have  
\begin{align*}
\sum_{i,j=1}^p \bar H_{ij}^k \geq \frac{\E\Big[\sum_{i,j=1}^p \bar H_{ij}^k\Big]}{2} \geq    \frac{1}{2} \Big(\frac{\kappa}{2}\Big)^k pn^k \geq c_0^k \pnorm{\bar\X}{\op}^{2k} \cdot p\Big(\frac{n}{p}\wedge 1\Big)^k
\end{align*}
for all $1\leq k\leq k_0$ and some $c_0 = c_0(\kappa,C_{\PI},C_0) > 0$, as desired.

\subsection{Proof of \prettyref{cor:prior}}
\label{sec:pf-cor_prior}

\begin{lemma}\label{lmm:W1truncate}
For any one-dimensional distribution $g$ and real number $A>0$, let $g_A$ be the conditional version of $g$ on $[-A,A]$, i.e., $g_A(B) = \frac{g(B \cap [-A,A])}{g([-A,A])}$ for any Borel set $B \subset\R$.
Then 
\begin{align}\label{eq:W1_general}
W_1(g,g_A) \leq \mathbb{E}_g\big[(|X|-A)_+\big] + 2A\,\mathbb{P}_g(|X|>A),
\end{align}
where $X \sim g$. Furthermore, if $g$ is subgaussian with some constant $M > 0$, then 
\begin{align*}
W_1(g,g_A)  \leq C e^{-A^2/C}
\end{align*} 
for some constant $C=C(M)$.
\end{lemma}
\begin{proof}
Denote the CDF of  $g$ by $G$. 
The CDF of $g_A$ is
\[
G_A(x) =
\begin{cases}
0, & x < -A, \\
\dfrac{G(x)-G(-A)}{1-p}, & -A \le x \le A, \\
1, & x > A
\end{cases}
\]
where 
$p \equiv \mathbb{P}_g(|X|>A) = G(-A)+1-G(A)$. In one dimension,
\begin{align*}
W_1(g,g_A) = \int_{-\infty}^{\infty} |G(x)-G_A(x)|\d x.
\end{align*}
The contribution outside $[-A,A]$ is
\begin{align*}
\int_{(-\infty,-A] \cup [A,\infty)}|G(x)-G_A(x)|\d x 
= \int_{-\infty}^{-A} G(x)\d x + \int_A^\infty (1-G(x))\d x = \mathbb{E}_g\big[(|X|-A)_+\big].
\end{align*}
Inside the interval, for $x\in[-A,A]$, $
\Delta(x) \equiv G(x)-G_A(x) = \frac{G(-A)-p\,G(x)}{1-p}$.
Note that $\Delta$ is monotone and bounded between $G(-A)$ and $G(-A)-p$, and so $
|\Delta(x)| \le p$ for all $ x\in[-A,A]$. 
Therefore, $\int_{-A}^A |G(x)-G_A(x)|\d x \leq 2Ap$. This proves (\ref{eq:W1_general}). The bound for subgaussian $g$ follows from applying the tail estimates.
\end{proof}

\begin{proof}[Proof of \prettyref{cor:prior}]
Let
\[
m:=m_1(g),
\qquad
g_0:=\text{Law}(\beta-m), \quad \beta\sim g.
\]
Thus $g_0$ is centered and
\[
m_\ell(g_0)=\mu_\ell(g), \qquad \ell\ge 1,
\]
where $\mu_1(g)=0$. Let $\widetilde g$ denote the centered DMM
estimator and recall that the final estimator is
\[
\widehat g=\tau_{\widehat m_1}\#\widetilde g.
\]
Since translations are isometries for $W_1$, the triangle inequality
gives
\begin{equation}
W_1(g,\widehat g)
\le
W_1(g_0,\widetilde g)
+
|\widehat m_1-m|.
\label{eq:W1-center-shift}
\end{equation}

We first construct a simultaneous high-probability event on which all
of the estimated moments are sufficiently accurate. Set
\[
\eta:=k^{-4L},
\qquad
\widehat{\boldsymbol\mu}_L
:=
(0,\widehat\mu_2,\ldots,\widehat\mu_L).
\]
For each $2\le \ell\le L$, let $\mathcal C_\ell$ be the event supplied
by Theorem~\ref{thm:main_moment}. Thus
\[
\mathbb P(\mathcal C_\ell^c)\le \Delta_\ell,
\qquad
\mathbb E\left[
(\widehat\mu_\ell-\mu_\ell)^2
\mathbf 1_{\mathcal C_\ell}
\right]
\le \Delta_\ell.
\]
Consequently, Markov's inequality gives
\begin{align}
\mathbb P\left(
|\widehat\mu_\ell-\mu_\ell|>\eta
\right)
&\le
\mathbb P(\mathcal C_\ell^c)
+
\mathbb P\left(
|\widehat\mu_\ell-\mu_\ell|>\eta,\,
\mathcal C_\ell
\right) \notag\\
&\le
\Delta_\ell
+
\eta^{-2}
\mathbb E\left[
(\widehat\mu_\ell-\mu_\ell)^2
\mathbf 1_{\mathcal C_\ell}
\right] \notag\\
&\le
(1+\eta^{-2})\Delta_\ell.
\label{eq:moment-hp-correct}
\end{align}
Similarly, Proposition~\ref{prop:MSE_first_two} yields
\begin{equation}
\mathbb P\left(
|\widehat m_1-m|>\eta
\right)
\le
\frac{C}{N\eta^2}.
\label{eq:mean-hp-correct}
\end{equation}
Define
\[
\mathcal G
:=
\left\{
|\widehat m_1-m|\le\eta,\quad
|\widehat\mu_\ell-\mu_\ell|\le\eta
\text{ for every }2\le\ell\le L
\right\}.
\]
By \eqref{eq:moment-hp-correct}--\eqref{eq:mean-hp-correct},
\begin{equation}
\mathbb P(\mathcal G^c)
\le
\frac{C}{N\eta^2}
+
\sum_{\ell=2}^L(1+\eta^{-2})\Delta_\ell.
\label{eq:G-prob-first}
\end{equation}

We next simplify the right side. Since $n\ge p^{1-\varepsilon}$ and
$0<\varepsilon\le 1/2$, we have
\begin{equation}
\frac pn\vee 1\le N^{2\varepsilon}.
\label{eq:aspect-N}
\end{equation}
Indeed, this is immediate when $n\ge p$. When $n<p$, we have $N=n$
and
\[
\frac pn
\le
n^{\varepsilon/(1-\varepsilon)}
\le n^{2\varepsilon}.
\]
It follows from Theorem~\ref{thm:main_moment} that, for $\ell\le L$, %\zf{$5\ell^2$ instead of $5\ell/2$?}
\[
\Delta_\ell
\le
(C\ell)^{5\ell^2}
N^{-1+4\varepsilon\ell^2}.
\]
Since $L=2k-1$, $\eta^{-2}=k^{8L}$, and
$k\le c\varepsilon^{-1/2}$, \eqref{eq:G-prob-first} implies
\begin{equation}
\mathbb P(\mathcal G^c)
\le
k^{C_0k^2}N^{-1+a},
\qquad
a:=4\varepsilon L^2\le 16c^2,
\label{eq:G-prob-second}
\end{equation}
for a constant $C_0>0$. We choose $c$ sufficiently small that
$a\le 1/4$.

Fix any constant $C_2>0$. We claim that, after choosing $C_1$
sufficiently large,
\begin{equation}
k^{C_0k^2}N^{-1+a}
\le
\frac{k^{C_1k^2}}{N}
+
k^{-C_2}.
\label{eq:G-prob-split}
\end{equation}
To see this, first suppose that
\[
N^a\le k^{C_1k^2/2}.
\]
Then
\[
k^{C_0k^2}N^{-1+a}
\le
\frac{k^{(C_0+C_1)k^2/2}}{N}
\le
\frac{k^{C_1k^2}}{N}
\]
provided $C_1 \geq 2C_0$. Otherwise,
$N^a>k^{C_1k^2/2}$. Since $a\le 1/4$,
\[
k^{C_0k^2}N^{-1+a}
\le
k^{C_0k^2}N^{-3/4}
\le
k^{C_0k^2-3C_1k^2/(8a)}
\le k^{-C_2},
\]
by choosing $C_1$ sufficiently large (both $C_1$ and $C_2$ are independent of $a,\eps,c$). 
%\zf{Perhaps clarify that $k^2/(8a)$ has a universal lower bound for all $\eps \leq 1/2$, hence $C_1,C_2$ do not depend on $c,a,\eps$.}
Combining
\eqref{eq:G-prob-second} and \eqref{eq:G-prob-split}, we obtain
\begin{equation}
\mathbb P(\mathcal G^c)
\le
\frac{k^{C_1k^2}}{N}
+
k^{-C_2}.
\label{eq:G-prob-final}
\end{equation}

We now work on the event $\mathcal G$. Let $g_{0,A}$ be the
conditional distribution of $g_0$ on $[-A,A]$, and let $q_A$ be the
$k$-point Gaussian quadrature of $g_{0,A}$. Thus $q_A$ is supported on
$[-A,A]$, has at most $k$ atoms, and satisfies
\begin{equation}
m_\ell(q_A)=m_\ell(g_{0,A}),
\qquad
1\le\ell\le L=2k-1.
\label{eq:quadrature-match}
\end{equation}
By Lemma~\ref{lmm:W1truncate},
\begin{equation}
W_1(g_0,g_{0,A})
\le
C\exp(-A^2/C).
\label{eq:W-trunc-correct}
\end{equation}
Furthermore, since $g_{0,A}$ and $q_A$ are both supported on
$[-A,A]$ and their first $L$ moments agree, \cite[Lemma ~7.4.1]{WY-fnt} yields that
%\zf{reference for this fact?}
\begin{equation}
W_1(g_{0,A},q_A)
\le
\frac{CA}{L}.
\label{eq:W-quadrature-correct}
\end{equation}

We next control the difference between the moments of $g_0$ and
$g_{0,A}$. Let $Z\sim g_0$ and
\[
p_A:=\mathbb P(|Z|>A).
\]
For every $1\le\ell\le L$,
\[
m_\ell(g_{0,A})
=
\frac{
m_\ell(g_0)
-
\mathbb E[Z^\ell\mathbf 1_{\{|Z|>A\}}]
}{
1-p_A
},
\]
and hence, provided $p_A\le 1/2$,
\begin{equation}
|m_\ell(g_{0,A})-m_\ell(g_0)|
\le
2p_A|m_\ell(g_0)|
+
2\mathbb E\left[
|Z|^\ell\mathbf 1_{\{|Z|>A\}}
\right].
\label{eq:truncated-moment-difference}
\end{equation}
The correct tail-integration identity is
\begin{equation}
\mathbb E\left[
|Z|^\ell\mathbf 1_{\{|Z|>A\}}
\right]
=
A^\ell\mathbb P(|Z|>A)
+
\int_A^\infty
\ell t^{\ell-1}\mathbb P(|Z|>t)\,dt.
\label{eq:tail-integration-correct}
\end{equation}
By subgaussianity,
\[
\mathbb P(|Z|>t)\le 2e^{-ct^2},
\qquad
|m_\ell(g_0)|\le (C\sqrt{\ell})^\ell.
\]
Since $A^2=C_Ak\log k$ and $L=2k-1$, choosing $C_A$ sufficiently
large gives
\begin{equation}
\tau_A
:=
\max_{1\le\ell\le L}
|m_\ell(g_{0,A})-m_\ell(g_0)|
\le
k^{-5L}
\le \eta.
\label{eq:truncated-moments-small}
\end{equation}

Let
\[
\widetilde{\boldsymbol\mu}_L
=
\argmin_{\boldsymbol m\in\mathcal M_L([-A,A])}
\left\|
\boldsymbol m-\widehat{\boldsymbol\mu}_L
\right\|_2
\]
be the projected moment vector in the DMM construction, and let
$\widetilde g$ be an at-most-$k$-atomic distribution on $[-A,A]$
whose first $L$ moments equal
$\widetilde{\boldsymbol\mu}_L$. Since
$\boldsymbol m_L(q_A)\in\mathcal M_L([-A,A])$, the defining property
of the projection gives
\begin{align}
\left\|
\boldsymbol m_L(\widetilde g)
-
\boldsymbol m_L(q_A)
\right\|_2
&\le
\left\|
\boldsymbol m_L(\widetilde g)
-
\widehat{\boldsymbol\mu}_L
\right\|_2
+
\left\|
\widehat{\boldsymbol\mu}_L
-
\boldsymbol m_L(q_A)
\right\|_2 \notag\\
&\le
2
\left\|
\widehat{\boldsymbol\mu}_L
-
\boldsymbol m_L(q_A)
\right\|_2 \notag\\
&\le
2\sqrt L\,(\eta+\tau_A).
\label{eq:projection-correct}
\end{align}
Here the coordinate $\ell=1$ is also covered because
$m_1(g_0)=0$ and
$|m_1(g_{0,A})|\le\tau_A$. Therefore,
\begin{equation}
\delta
:=
\max_{1\le\ell\le L}
|m_\ell(\widetilde g)-m_\ell(q_A)|
\le
4\sqrt L\,k^{-4L}.
\label{eq:delta-moments}
\end{equation}

Both $\widetilde g$ and $q_A$ have at most $k$ atoms and are supported
on $[-A,A]$. Rescaling them by $x\mapsto x/A$ and applying the
moment-comparison theorem for two $k$-atomic distributions \cite[Proposition 6.3.1]{WY-fnt} gives
\begin{align}
W_1(\widetilde g,q_A)
&\le
CkA\,\delta^{1/(2k-1)}
=
CkA\,\delta^{1/L} \notag\\
&\le
CkA
\left(
4\sqrt L\,k^{-4L}
\right)^{1/L}
\le
\frac{CA}{k^3}.
\label{eq:W-moment-comparison-correct}
\end{align}
Notice that the exponent here is $1/(2k-1)=1/L$, and that the
multiplicative factor $k$ must be retained.

Combining
\eqref{eq:W1-center-shift},
\eqref{eq:W-trunc-correct},
\eqref{eq:W-quadrature-correct}, and
\eqref{eq:W-moment-comparison-correct}, on $\mathcal G$ we have
\begin{align}
W_1(g,\widehat g)
&\le
|\widehat m_1-m|
+
W_1(g_0,g_{0,A})
+
W_1(g_{0,A},q_A)
+
W_1(q_A,\widetilde g) \notag\\
&\le
\eta
+
Ce^{-A^2/C}
+
\frac{CA}{L}
+
\frac{CA}{k^3} \notag\\
&\le
\frac{CA}{L}
\le
C\sqrt{\frac{\log k}{k}}.
\end{align}
Together with \eqref{eq:G-prob-final}, this proves the result.
\end{proof}

\section{Proofs of lower bounds}
\label{sec:lb-pf}

\subsection{Proof of \prettyref{prop:lower_bound-block}}
\label{sec:lower_bound-block}

For simplicity, we take the noise variance to be $\sigma^2 = 1$.
Fix $L\geq 2$. Let $g,g'$ be two distinct symmetric distributions supported on $[-c,c]$ whose first $2L-1$ moments agree, where the universal constant $c>0$ will be chosen sufficiently small. Such a pair exists, for example, by taking $g$ to be the uniform distribution on $[-c,c]$ and $g'$ to be its $L$-point Gaussian quadrature, which is also symmetric and matches all moments through degree $2L-1$.

Recall the block design $\X$ in \prettyref{eq:blockdesign}, where $m = p/n \in \naturals$.
Let $\bbeta=(\beta_1,\ldots,\beta_p) \iiddistr g$. Then $\btheta = \X \bbeta \sim (g^{(m)})^{\otimes n}$, where $g^{(m)}  = \text{Law}\pth{\frac{\beta_1+\ldots+\beta_m}{\sqrt{m}}}$ is the rescaled $m$-fold self-convolution of $g$ defined in  \prettyref{eq:self-conv}. Thus,
$\y = \btheta + \beps$ has a product distribution $P_g(\y) = (g^{(m)} * \varphi)^{\otimes n}$, where $\varphi$ is the standard normal density. So we can bound the total variation between product distribution by
\[
\TV(P_g(\y),P_{g'}(\y)) 
\leq 
n \cdot\TV(g^{(m)} * \varphi, g'^{(m)} * \varphi).
\]
Hence it suffices to bound this one dimensional total variation.

The distance between the one-dimensional normal mixture can be bounded using standard moment matching argument:
\begin{align}
4 \TV(g^{(m)} * \varphi, g'^{(m)} * \varphi)^2
&\stepa{\leq} \int \frac{\big(
g^{(m)} * \varphi- g'^{(m)} * \varphi\big)(y)^2}{\varphi(y)} \d y \nonumber \\
& \stepb{=}  \sum_{k=0}^\infty \frac{1}{k!} \big(m_k(g^{(m)}) - m_k(g'^{(m)})\big)^2,
\label{eq:hermite-block}
% & = \sum_{k=0}^\infty \frac{1}{k!} m^{-k}\big(m_k(g^{* m}) - m_k(\tilde{g}^{* m})\big)^2,
\end{align}
where (a) follows from the variational representation $\TV(p,q) = \frac{1}{2} \inf_r \sqrt{\int \frac{(p-q)^2}{r}}$ with infimum is over all probability densities $r$ (cf.~\cite[Exercise I.36]{PW-it});
(b) follows from expanding the mixture density under the Hermite basis whose coefficients are given by the moments of the mixing distribution (see \cite[Lemma 9]{wu2020optimal}).

It remains to bound the moment differences of the self-convolutions. Abbreviate
\[
\mu_r\equiv m_r(g),\qquad \mu_r'\equiv m_r(g').
\]
Let $\beta_1,\ldots,\beta_m\iiddistr g$ and $\beta_1',\ldots,\beta_m'\iiddistr g'$, and write
\[
T_m=\beta_1+\cdots+\beta_m,
\qquad
T_m'=\beta_1'+\cdots+\beta_m'.
\]
For an even integer $k\geq 0$, let $\cE_k$ denote the collection of set partitions of $[k]$ into nonempty blocks, each of even cardinality, and $\cE_0 \equiv \{\emptyset\}$ by convention.
 Since $g$ and $g'$ are symmetric, grouping the expansion of the $k$th moment according to the equality pattern among the indices gives, for even $k$,
\begin{equation}
\E[T_m^k]
=\sum_{\pi\in\cE_k}(m)_{|\pi|}\prod_{B\in\pi}\mu_{|B|},
\qquad
\E[(T_m')^k]
=\sum_{\pi\in\cE_k}(m)_{|\pi|}\prod_{B\in\pi}\mu_{|B|}',
\label{eq:partition-moment-block}
\end{equation}
where 
the falling factorial
$(m)_b\equiv m(m-1)\cdots(m-b+1)$ is the number of ways to assign distinct indices in $[m]$ to $b$ blocks. For odd $k$, both moments vanish.

For $\pi\in\cE_k$, let
\[
N_L(\pi)=\#\{B\in\pi:|B|\geq 2L\}.
\]
Note that  $\mu_r=\mu_r'$ for $r\leq 2L-1$ and $0\leq\mu_r,\mu_r'\leq c^r$ for even $r$, and the fact that 
for $a_j,b_j\geq 0$,
\[
\left|\prod_j a_j-\prod_j b_j\right|
\leq \sum_j |a_j-b_j| \prod_{\ell\neq j} \max\{a_\ell,b_\ell\},
\]
% a telescoping expansion of the difference of the two products 
Applying this to \prettyref{eq:partition-moment-block} yields
\[
\left|\prod_{B\in\pi}\mu_{|B|}
-\prod_{B\in\pi}\mu_{|B|}'\right|
\leq 2c^k N_L(\pi).
\]
Consequently,
\begin{equation}
\left|\E[T_m^k]-\E[(T_m')^k]\right|
\leq 2c^k\sum_{\pi\in\cE_k}(m)_{|\pi|}N_L(\pi).
\label{eq:marked-partition-block}
\end{equation}

To evaluate this sum, note that
\[
\sum_{\pi\in\cE_k}(m)_{|\pi|}N_L(\pi)
= |\{(\pi,B_*,\lambda): 
\pi\in \cE_k, B_* \in \pi, |B_*| \geq 2L, \lambda: \pi\to[m] \text{ injective}
\}|.
\]
To enumerate all such triples, note that if the marked block $B_*$ has size $r$, first choose its $r$ elements from $[k]$ and its index assignment from  $[m]$, and then partition the remaining positions into even blocks and assign them distinct labels from the remaining $m-1$ indices. This gives the exact identity
\begin{equation}
\sum_{\pi\in\cE_k}(m)_{|\pi|}N_L(\pi)
=m\sum_{\substack{r\geq 2L\\r\text{ even}}}^{k}
\binom{k}{r}
\sum_{\rho\in\cE_{k-r}}(m-1)_{|\rho|}.
\label{eq:marked-block-identity}
\end{equation}
 Applying 
\prettyref{eq:partition-moment-block} again, we recognize that 
\[
\sum_{\rho\in\cE_q}(m-1)_{|\rho|}
=\E\left(\eps_1+\cdots+\eps_{m-1}\right)^q
\]
where $\eps_i$'s are i.i.d.\ Rademacher random variables.
Since each $\eps_i$ is subgaussian with constant proxy variance, this sum is 
subgaussian with proxy variance proportional to $m$. Hence
% for every integer $q\geq 1$,
\[
\Expect[|\eps_1+\cdots+\eps_{m-1}|^q]
\leq (C\sqrt{q(m-1)})^q, \quad q\geq 1
\]
for some universal constant $C$. Combining this estimate with \prettyref{eq:marked-partition-block}--\prettyref{eq:marked-block-identity}, for even $k\geq 2L$ we obtain
\begin{align*}
\left|m_k(g^{(m)})-m_k(g'^{(m)})\right|
&=m^{-k/2}\left|\E[T_m^k]-\E[(T_m')^k]\right|\\
&\leq 2c^k m^{-k/2}m
\sum_{\substack{r\geq 2L\\r\text{ even}}}^{k}
\binom{k}{r}\big(C\sqrt{(k-r)(m-1)}\big)^{k-r}\\
&\leq 2c^k m^{-(L-1)}
\sum_{r=0}^{k}\binom{k}{r}(C\sqrt{k})^{k-r}\\
&=2c^k m^{-(L-1)}(1+C\sqrt{k})^k\\
&\leq m^{-(L-1)}(C_1c)^k k^{k/2}\exp(C_0\sqrt{k})
\end{align*}
for universal constants $C_0,C_1>0$. 
% In the second line, the term with $r=k$ is interpreted using the zeroth moment, and in the third line we used $r\geq 2L$.

Substituting this bound into \prettyref{eq:hermite-block} and recalling that the moment differences vanish for $k\leq 2L-1$ and for odd $k$, we get
\begin{align*}
4\TV(g^{(m)}*\varphi,g'^{(m)}*\varphi)^2
&\leq m^{-2(L-1)}
\sum_{k=2L}^{\infty}
\frac{(C_1c)^{2k}k^k}{k!}\exp(2C_0\sqrt{k})\\
&\leq C_2m^{-2(L-1)},
\end{align*}
where the last inequality follows from Stirling's approximation, provided that the universal constant $c>0$ (defining the support of $g,g'$) is chosen sufficiently small. Therefore, we obtain the desired
\[
\TV(g^{(m)}*\varphi,g'^{(m)}*\varphi)
\leq C m^{-(L-1)}.
\]  \qed

\subsection{Proof of \prettyref{thm:lower_bound}}
\label{sec:lower_bound-pf}

For simplicity, we take the error variance $\sigma^2 = 1$. Let $h,h'$ be two centered distributions supported on $[-1/2,1/2]$ with matching first $L$ moments. Let
\begin{align*}
\bbeta := \bgamma + \z, \quad \bbeta' := \bgamma' + \z, \quad \text{ where } \bgamma \sim h^{\otimes p}, \quad \bgamma' \sim h'^{\otimes p}, \quad \z\sim \N(0,\I_p),   
\end{align*}
so that $\bbeta$ and $\bbeta'$ have i.i.d.~entries distributed as $g = h \ast \N(0,1)$ and $g' = h'\ast \N(0,1)$ respectively, both of which are 1-subgaussian. Let
\begin{align*}
\tilde\beps := \X\z + \beps \sim \N(0, \X\X^\top + \I_n),\quad \w := (\X\X^\top + \I_n)^{-1/2}\tilde\beps \sim \N(0, \I_n).
\end{align*}
For any $0\leq k \leq p$, let $\bgamma^{[k]}\in\R^p$ be a random vector such that $(\bgamma^{[k]})_\ell = \gamma_\ell$ for $1\leq \ell\leq k$, and $(\bgamma^{[k]})_{\ell} = \gamma'_\ell$ for $k+1 \leq \ell \leq p$, so that $\bgamma^{[0]} = \bgamma'$ and $\bgamma^{[p]} = \bgamma$. Then
\begin{align}
\notag\dTV(P_g(\y), P_{g'}(\y)) &= \dTV(\X\bgamma^{[p]} + \tilde\beps, \X\bgamma^{[0]}+ \tilde\beps)\\
\notag&\overset{(\ast)}{=} \dTV(\tilde \X \bgamma^{[p]} + \w, \tilde \X \bgamma^{[0]} + \w)\\
&\leq \sum_{k=1}^p \dTV(\tilde\X\bgamma^{[k]}+ \w, \tilde\X\bgamma^{[k-1]}+ \w),
\label{eq:TV1}
\end{align}
where $(\ast)$ uses the fact that the total variation distance is invariant under invertible linear transformations, and
\begin{align*}
\tilde \X = (\X\X^\top + \I_n)^{-1/2}\X.
\end{align*}
Let $\tilde\x_j$ denote the $j$th column of $\tilde\X$. Then for any $1\leq k\leq p$, by the data processing inequality and rotational invariance of the total variation distance, 
\begin{align}
\notag&\dTV(\tilde\X\bgamma^{[k]}+ \w, \tilde\X\bgamma^{[k-1]}+ \w)\\
&= \dTV\Big(\sum_{\ell=1}^{k-1}\tilde\x_\ell \gamma_\ell + \tilde\x_k \gamma_k + \sum_{\ell=k+1}^p \tilde\x_\ell \gamma_\ell' + \w, \sum_{\ell=1}^{k-1}\tilde\x_\ell \gamma_\ell + \tilde\x_k \gamma_k' + \sum_{\ell=k+1}^p \tilde\x_\ell \gamma_\ell' + \w\Big)\nonumber\\
&\leq \dTV(\tilde\x_k \gamma_k + \w, \tilde\x_k \gamma_k' + \w) = \dTV(\pnorm{\tilde\x_k}{}\gamma_k + w_1, \pnorm{\tilde\x_k}{}\gamma_k' + w_1).\label{eq:TV2}
\end{align}

Next we apply a standard moment matching argument to bound the distance between these two univariate normal mixtures. 
% Let $Y = \pnorm{\x_k}{}\bbeta_k + \beps_1$ and $Y' = \pnorm{\x_k}{}\bbeta_k' + \beps_1$.
Recall the assumption that $\max_{k\in[p]}\pnorm{\tilde\x_k}{} \leq C_0 (\frac{n}{p})^\alpha (\log p)^\gamma
\leq 1/2$ which holds for 
$n \leq p^{1-\delta}$ and all large enough $p$.
Applying \cite[Theorem 3.3.3, Part 2]{WY-fnt}, 
we get for some universal constant $C$,
\begin{align*}
\chi^2(\pnorm{\tilde\x_k}{}\gamma_k + w_1, \pnorm{\tilde\x_k}{}\gamma_k' + w_1) \leq (C \pnorm{\tilde\x_k}{})^{2(L+1)}.
\end{align*}
Finally, combining this with \prettyref{eq:TV1}--\prettyref{eq:TV2} and
applying $\TV^2 \leq \frac{1}{4} \chi^2$ (see, e.g., \cite[Proposition 7.15]{PW-it}), we get
\begin{align*}
\dTV(P_g(\y), P_{g'}(\y)) \leq  p C^{L+1} \cdot \max_{k\in[p]}\pnorm{\tilde\x_k}{}^{L+1} \leq  p  \qth{CC_0 \Big(\frac{n}{p}\Big)^\alpha (\log p)^\gamma}^{L+1} \leq 0.1,
\end{align*}
by choosing $L$ to be large enough (depending only on $C_0,\alpha,\delta,\gamma$). The proof is complete. \qed

\subsection{Proof of Lemma \ref{lem:lower_gaussian}}\label{subsec:proof_lower_gaussian}
\begin{proof}
By the definition of $\X$,
\begin{align*}
\X\X^\top+\sigma^2\I_n
&=
\sigma^2\left(
\I_n+\frac{1}{n+p}\bar\X\bar\X^\top
\right).
\end{align*}
It follows that
\begin{align*}
\tilde\X =
(\X\X^\top+\sigma^2\I_n)^{-1/2}\X =
\frac{1}{\sqrt{n+p}}
\left(
\I_n+\frac{1}{n+p}\bar\X\bar\X^\top
\right)^{-1/2}
\bar\X,
\end{align*}
and hence for every $j\in[p]$,
\begin{equation}
\label{eq:lower-gaussian-contraction}
\pnorm{\tilde\X\e_j}{}^2
\leq
\frac{1}{n+p}\pnorm{\bar\X\e_j}{}^2.
\end{equation}

It therefore remains to control the column norms of $\bar\X$.
Let $\bar X_{ij}$ denote the $(i,j)$th entry of $\bar\X$.
By the subgaussian assumption and
$\bSigma\preceq \kappa\I_p$, there exists a constant
$K=K(C_0,\kappa)$ such that
\begin{align*}
\max_{i\in[n],\,j\in[p]}
\pnorm{\bar X_{ij}}{\psi_2}
\leq K,
\end{align*}
which further implies that for a universal constant $C>0$,
\begin{align*}
\pnorm{
\bar X_{ij}^2-\E\bar X_{ij}^2
}{\psi_1}
\leq CK^2.
\end{align*}
For each fixed $j$, the random variables $\bar X_{ij}^2-\E\bar X_{ij}^2$, $i\in[n]$ are independent and centered. Bernstein's inequality for sub-exponential random variables therefore gives, for every $t\geq 0$,
\begin{align}
\label{eq:lower-gaussian-bernstein}
\Prob\left(
\left|
\sum_{i=1}^n
\left(
\bar X_{ij}^2-\E\bar X_{ij}^2
\right)
\right|
>
CK^2\left(\sqrt{nt}+t\right)
\right)
\leq 2e^{-t},
\end{align}
where $C>0$ is a universal constant. Since the rows of $\bar\X$ have covariance $\bSigma$,
\begin{align*}
\E\bar X_{ij}^2
=
\Sigma_{jj}
\leq
\pnorm{\bSigma}{\op}
\leq \kappa.
\end{align*}
Taking $t=12\log p$ in
\eqref{eq:lower-gaussian-bernstein} and applying a union bound over
$j\in[p]$, we obtain, with probability at least $1-2p^{-11}$,
\begin{align*}
\max_{j\in[p]}\pnorm{\bar\X\e_j}{}^2
\leq
\kappa n
+
CK^2\left(
\sqrt{12n\log p}+12\log p
\right) \leq
C'(n+\log p),
\end{align*}
where $C'=C'(C_0,\kappa)$. Combining this estimate with
\eqref{eq:lower-gaussian-contraction}, we conclude that, with
probability at least $1-2p^{-11}$,
\begin{align*}
\max_{j\in[p]}\pnorm{\tilde\X\e_j}{}^2 \leq
C'\frac{n+\log p}{n+p} \leq
2C'\frac{n}{p}\log p
\end{align*}
and hence condition
\eqref{eq:design_assumption_lower} holds with probability converging to one, with $\alpha=\gamma=\frac{1}{2}$.
\end{proof}

\appendix

\section{Unbiased MoM with minimal variance}
\label{app:bettermom}
The EBMoM method is based on constructing an  unbiased estimator for each moment assuming all lower moments are known. As discussed in \prettyref{sec:construction}, there are in fact a continuum of such unbiased estimators, including the EBMoM estimator as a special case. It is of interest to consider the best moment estimator that attains the minimal variance. We explore this for the first two moments 
in this appendix, discuss the connections to the Rao's MIVQUE theory, and some of the computational challenges.

\paragraph{First moment}
To start, let us consider estimating the first moment $m_1$. 
Fix a test vector $\b \in \reals^n$.
Then
\[
\iprod{\b}{\y} = 
\mu \iprod{\ones}{\X^\top \b} + 
\iprod{\bbeta-\mu \ones}{\X^\top \b} + 
\iprod{\beps}{\b}.
\]
Provided that $\iprod{\ones}{\X^\top \b} \neq 0$, the estimator $\tilde m_1 = \frac{\iprod{\b}{\y}}{\iprod{\ones}{\X^\top \b}}$ is unbiased and satisfies
\[
\Var(\tilde m_1)
= \frac{\sigma^2 \|\b\|^2 + \mu_2 \|\X^\top \b\|^2}{\iprod{\b}{\X \ones}^2}
\]
where $\mu_2$ is the variance of the prior.
Minimizing this variance among all test vectors leads to the (oracle) choice of test vector
\begin{equation}
\b_* = \Q^{-1} \X \ones, \quad \Q = \mu_2\X\X^\top + \sigma^2 \I_n
\label{eq:oracle-b}    
\end{equation}
leading to the minimal variance
\[
\frac{1}{(\X\ones)^\top \Q^{-1} \X\ones }.
\]
In comparison, our estimator \prettyref{eq:m1hat} in \prettyref{sec:method} corresponds to choosing $\b = \X\ones$.
We say \prettyref{eq:oracle-b} is an oracle choice because it relies the knowledge of the true second moment (and the noise level). In practice we can plug in any consistent estimator of the second moment such as those from \prettyref{sec:method}.
In principle, one can iterate this because the second moment estimate relies on the first moment estimate. \prettyref{fig:m1} gives simulation result that shows the benefit of choosing the optimal test vector. In both cases, the variance is reduced by about a quarter.
\begin{figure}[!ht]
    \centering
\includegraphics[width=0.9\linewidth]{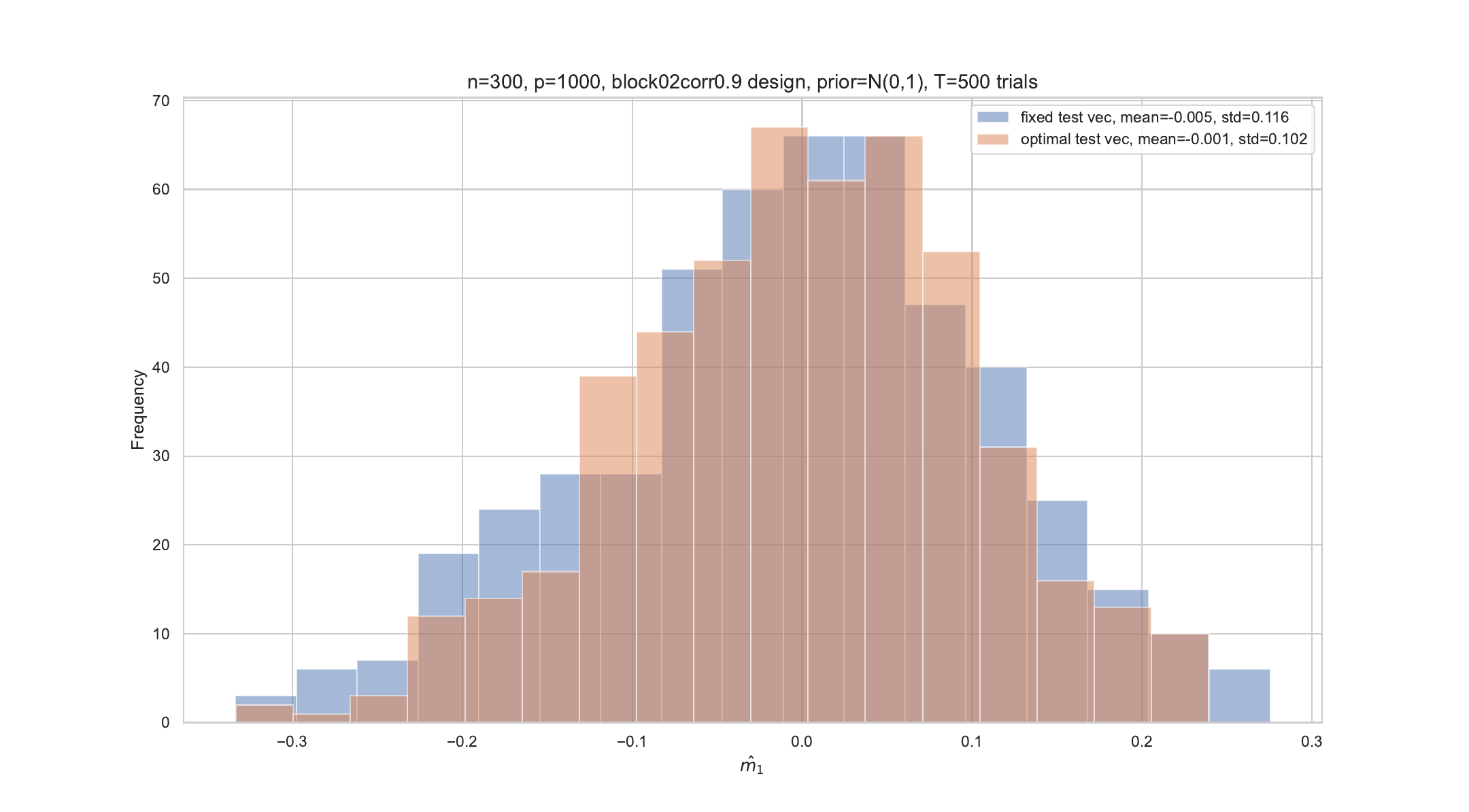}    \includegraphics[width=0.9\linewidth]{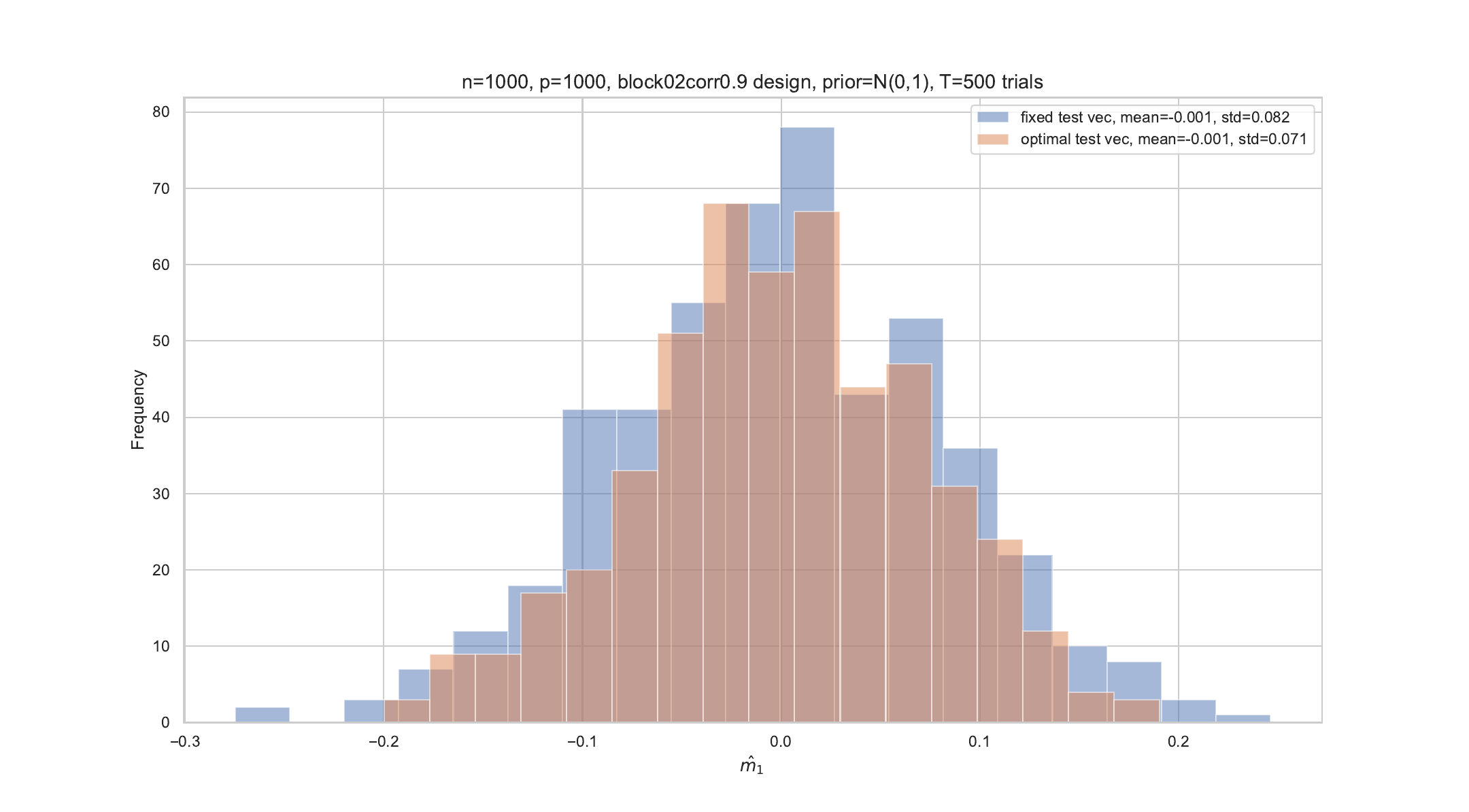}\caption{Histogram of first moment estimates: the simple estimator \prettyref{eq:m1hat} versus the oracle estimator \prettyref{eq:oracle-b}, with  estimated second moment using \prettyref{eq:mu2hat}.}
    \label{fig:m1}
\end{figure}

Note that not all test vectors work equally well. As a counterexample, consider $\b=\ones$ so that $\tilde{m}_1 = \frac{\bm{1}^\top \y}{\bm{1}^\top \X \bm{1}}$. This is unbiased, but
\begin{align*}
\var(\tilde{m}_1) = \frac{\var(\bm{1}^\top \X\bbeta) + \var(\bm{1}^\top \beps)}{(\bm{1}^\top \X \bm{1})^2} = \frac{\mu_2\pnorm{\X^\top\bm{1}}{}^2 + n\sigma^2}{(\bm{1}^\top \X \bm{1})^2}. 
\end{align*}
For block design with $n\leq p$ (see (\ref{eq:blockdesign}) for the definition), $\pnorm{\X^\top\bm{1}}{}^2 = n$ and $\bm{1}^\top \X\bm{1} = \sqrt{np}$, so the variance is $O(1/p)$. But for Gaussian design, $\pnorm{\X^\top\bm{1}}{}^2 = \Theta(n)$ and $(\bm{1}^\top \X \bm{1})^2 = \Theta(n)$ with high probability, so $\tilde{m}_1$ is not consistent regardless of how big  $n,p$ are.

% Conceivably, similar improved estimators can be found for the second and higher moments, but the calculation 

\paragraph{Second moment}
For simplicity, assume that the prior is already centered. To estimate its variance $\mu_2$, fix a symmetric $\Q \in \reals^{n\times n}$ and consider a quadratic form $\y^\top \Q \y$.
Then 
$\Expect[\y^\top \Q \y] = 
\tr(\X^\top \Q \X) \mu_2 + \tr(\Q) \sigma^2$,
so
\[
\hat \mu_2 = \frac{\y^\top \Q \y - \tr(\Q) \sigma^2}{\tr(\X^\top \Q \X)}
\]
is an unbiased estimator for $\mu_2$.
  Repeating the same calculation as in the proof of \prettyref{prop:MSE_first_two} (which corresponds to $\Q = \X\X^\top$), we get the counterpart of \prettyref{eq:m2var}:
\begin{equation}
\var(\hat \mu_2) 
 = 
 \frac{
(\mu_4 -3\mu_2^2) \sum_{a=1}^p (\X^\top \Q\X)_{aa}^2 + 2 \mu_2^2   \Fnorm{\X^\top \Q \X}^2
+ 4  \sigma^2 \mu_2  \tr(\Q \X\X^\top \Q)
+ 2 \sigma^4 \Fnorm{\Q}^2 }{\tr(\X^\top \Q \X)^2}.
    \label{eq:mu2-variance-general}
\end{equation}
Because this ratio is scale-invariant, we can minimize it by fixing the denominator as a constraint and minimizing the quadratic objective in the numerator. For Gaussian priors (or any distributions with vanishing 4th cumulant), which satisfy $\mu_4 = 3\mu_2^2 $, the first term in \prettyref{eq:mu2-variance-general} drops, and the optimal $\Q $ can be found by the first-order optimality condition: with $\bar\H = \X\X^\top$,
 \[
 \mu_2^2 
 \bar\H \Q \bar\H +\mu_2\sigma^2 (\bar\H \Q+\Q\bar\H)+\sigma^4\Q = \bar\H,
 % \Tr(\H \Q)=1.
 \]
 which has a simple closed form solution through the eigenvalue decomposition of $ \bar\H = \U \D \U^\top $, given as
$$ \Q_* = \U \bLambda \U^\top,\ \bLambda = \diag \Big( \frac{D_{ii}}{(\mu_2 D_{ii} + \sigma^2)^2} \Big)_{i=1}^n. $$
Similar to Rao's MIVQUE estimator (which simultaneously estimates $\mu_2$ and $\sigma^2$ under the same excess kurtosis condition \cite{rao1971minimum,swallow1978minimum}), this optimal estimator depends on the unknown parameter $\mu_2$ and hence requires a pilot estimator.
Without assuming normality, 
\prettyref{eq:mu2-variance-general} can no longer be solved spectrally; instead, it can be  solved as a quadratic optimization problem using vectorization, which involves $O(n^2)$ variables. Conceivably, optimizing the estimators for higher moments is even more computationally intensive. 
For this reason, we advocate for the explicit (and natural) construction in the EBMoM estimators.

\section{Improved analysis of third moment estimate}
\label{app:m3}

Under appropriate assumptions on the design and prior, Theorem \ref{thm:main_moment} suggests that our central moment estimator $\hat\mu_k$ will be consistent for $\mu_k$ as long as $n \gg p^{1-1/(1+2k^2)}$. This sub-linear gap $1/(1+2k^2)$ is rather crude for finite $k$ and can be significantly improved upon more careful analysis of $\hat\mu_k$. As a proof of concept, we show such a result for the third moment estimator $\hat\mu_3$ in the special case of Gaussian design. For simplicity, we assume $m_1 = 0$, in which case debiasing of lower moments in (\ref{eq:moment_def}) is not necessary, and, using $H_3(x) = x^3 - 3x$,
\begin{align*}
\hat\mu_3 = \frac{F_3(\y)}{A_3} = \frac{1}{A_3}\Big[\sum_{j=1}^p (\x_j^\top \y)^3 - 3\pnorm{\x_j}{}^2 (\x_j^\top \y)\Big], 
\end{align*}
with $A_3$ given in (\ref{eq:Ak}).

\begin{proposition}
Suppose $\X$ has i.i.d.~$\N(0,1/p)$ entries and the prior Assumption \ref{assump:prior} holds. Assume further $m_1 = 0$, $\sigma = 1$, and $p^2/n^3 \rightarrow 0$. Then there exists an $\X$-measurable event $\cE_n$ with $\Prob_{\X}(\cE_n) \rightarrow 1$ such that
\begin{align*}
\E\big[(\hat\mu_3 - m_3)^2\bm{1}_{\cE_n}|\X\big] \rightarrow 0.
\end{align*}
\end{proposition}
%\zf{It gets a bit confusing in this statement and the proof below which $\E,\P,\Var$ are conditional on $\X$ and which are marginally over $\X$.}
\begin{proof}
Throughout the proof, we write $\E,\Var,\P$ for probabilistic statements conditioning on $X$, and $\E_{\X},\Var_{\X},\P_{\X}$ for those with respect to the randomness of $\X$. Since $m_1 = 0$, the third moment tensor $\M^{(3)} = \E[\bbeta^{\otimes 3}] = (m_3\bm{1}_{i=j=k})_{i,j,k\in[p]}$, hence (\ref{eq:F_k2}) shows that
\begin{align*}
\E[\hat\mu_3] = \frac{1}{A_3} \iprod{\T^{(3)}}{ (m_3\bm{1}_{i=j=k})_{i,j,k\in[p]}} = m_3 \frac{\Tr \T^{(3)}}{A_3} = m_3, 
\end{align*}
showing that $\hat\mu_3$ is unbiased. To compute the variance, we have
\begin{align}\label{eq:var_decompose}
\notag\var(\hat{m}_3) &= A_3^{-2}\var\Big(\sum_{j=1}^p (\x_j^\top \y)^3 - 3 \pnorm{\x_j}{}^2 (\x_j^\top \y)\Big)\\
&\leq A_3^{-2}\Big[2\underbrace{\var\Big(\sum_{j=1}^p (\x_j^\top \y)^3\Big)}_{(I)} + 18 \underbrace{\var\Big(\sum_{j=1}^p\pnorm{\x_j}{}^2 (\x_j^\top \y)\Big)}_{(II)}\Big].
\end{align}
We will show that on a sequence of event $\cE_n(\X)$ with probability converging to 1, $(I)/A_3^2 \rightarrow 0$ as $n,p \rightarrow\infty$. Similar calculation will yield that $(II)/A_3^2 \rightarrow 0$, which completes the proof. For $(I)$, we have
\begin{align*}
\sum_{j=1}^p (\x_j^\top \y)^3 = \sum_{j=1}^p \Big[(\x_j^\top \X\bbeta)^3 + 3(\x_j^\top \X\bbeta)^2(\x_j^\top \beps) + 3(\x_j^\top \X\bbeta)(\x_j^\top \beps)^2 + (\x_j^\top \beps)^3\Big].
\end{align*}
Let us bound $\var(\sum_{j=1}^p(\x_j^\top \X\bbeta)^3) = \E[(\sum_{j=1}^p(\x_j^\top \X\bbeta)^3)^2] - (\E[\sum_{j=1}^p(\x_j^\top \X\bbeta)^3])^2$, and the variances of other terms can be bounded by similar arguments. 
% For the latter, we have
% \begin{align*}
% \var\Big(\sum_{j=1}^p(\x_j^\top \beps)^3\Big) = \E \Big(\sum_{j=1}^p(\x_j^\top \beps)^3\Big)^2 = \E \Big[\sum_{j=1}^p (\x_j^\top \beps)^6\Big] + \sum_{j\neq j'} \E\Big[(\x_j^\top \beps)^3(\x_{j'}^\top \beps)^3\Big].
% \end{align*}
% Note that for $j\neq j'$,
% \begin{align*}
% \begin{pmatrix}
% \x_j^\top \beps\\
% \x_{j'}^\top \beps
% \end{pmatrix} \overset{d}{=}
% \N\bigg(
% \begin{pmatrix}
% 0\\
% 0
% \end{pmatrix},
% \begin{pmatrix}
% \pnorm{\x_j}{}^2 & \x_j^\top \x_{j'}\\
% \x_j^\top \x_{j'} & \pnorm{\x_{j'}}{}^2,
% \end{pmatrix}
% \bigg)
% \end{align*}
% hence $\E[\sum_{j=1}^p (\x_j^\top \beps)^6] = 15\sum_{j=1}^p \pnorm{\x_j}{}^6$, and
% \begin{align*}
% \E\big[(\x_j^\top \beps)^3(\x_{j'}^\top \beps)^3\big] = 6(\x_j^\top \x_{j'})^3 + 9(\x_j^\top \x_{j'})\pnorm{\x_j}{}^2\pnorm{\x_{j'}}{}^2,
% \end{align*}
% leading to
% \begin{align*}
% \var\Big(\sum_{j=1}^p(\x_j^\top \beps)^3\Big) = \sum_{j,j'=1}^p 6(\x_j^\top \x_{j'})^3 + 9(\x_j^\top \x_{j'})\pnorm{\x_j}{}^2\pnorm{\x_{j'}}{}^2.
% \end{align*}
% For iid Gaussian design, this is $O(n^3/p^2)$. 
Note that
\begin{align*}
\sum_{j=1}^p (\x_j^\top \X\bbeta)^3 = \sum_{j=1}^p \sum_{a,b,c=1}^p (\x_j^\top \x_a)(\x_j^\top \x_b)(\x_j^\top \x_c) \beta_a\beta_b\beta_c,
\end{align*}
so using $m_1 = 0$, we have
\begin{align*}
\E\Big[\sum_{j=1}^p(\x_j^\top \X\bbeta)^3\Big] = m_3 \cdot \sum_{j,k=1}^p (\x_j^\top \x_k)^3,  
\end{align*}
and
\begin{align*}
\E\Big[\Big(\sum_{j=1}^p(\x_j^\top \X\bbeta)^3\Big)^2\Big] = \sum_{j,k=1}^p \underbrace{\sum_{a,b,c,d,e,f=1}^p (\x_j^\top \x_a)(\x_j^\top \x_b)(\x_j^\top \x_c)(\x_k^\top \x_d)(\x_k^\top \x_e)(\x_k^\top \x_f)\E[\beta_a\beta_b\beta_c\beta_d\beta_e\beta_f]}_{E(j,k)}.
\end{align*}
As a notation shorthand, for fixed $j,k\in[p]$, let
\begin{align*}
u_a := u_a(j) = \x_j^\top \x_a, \quad v_a := v_a(k) = \x_k^\top \x_a, \quad a\in[p],
\end{align*}
and
\begin{align*}
S_{rs} := S_{rs}(j,k) := \sum_{a = 1}^p u_a^rv_a^s = \sum_{a=1}^p (\x_j^\top \x_a)^r (\x_k^\top \x_a)^s. 
\end{align*}
To have non-zero contribution towards $E(j,k)$, we have the following configurations for the multiplicities of $a$-$f$: (i) 6; (ii) 2+2+2; (iii) 2 + 4; (iv) 3 + 3, which correspond to
\begin{equation*}
\begin{gathered}
 m_6 \cdot\sum_{j,k=1}^p \sum_{a=1}^p (\x_j^\top \x_a)^3(\x_k^\top \x_a)^3 \quad \text{case} (i)\\
 m_2^3 \cdot \Big[\sum_{j,k=1}^p \sum_{a\neq b\neq c} 9(\x_j^\top \x_a)^2 (\x_j^\top \x_b)(\x_k^\top \x_b)(\x_k^\top \x_c)^2 + 6(\x_j^\top \x_a)(\x_k^\top \x_a)(\x_j^\top \x_b)(\x_k^\top \x_b)(\x_j^\top \x_c)(\x_k^\top \x_c)\Big] \quad \text{case} (ii)\\
 m_2m_4 \cdot \Big[\sum_{j,k=1}^p \sum_{a\neq b} 6(\x_j^\top \x_a)^2 (\x_j^\top \x_b)(\x_k^\top \x_b)^3 + 9(\x_j^\top \x_a)(\x_j^\top \x_b)^2(\x_k^\top \x_a)(\x_k^\top \x_b)^2\Big] \quad \text{case} (iii)\\
 m_3^2 \cdot\Big[\sum_{j,k=1}^p \sum_{a\neq b} (\x_j^\top \x_a)^3(\x_k^\top \x_b)^3 + 9 (\x_j^\top \x_a)(\x_j^\top \x_b)^2(\x_k^\top \x_a)^2(\x_k^\top \x_b)\Big] \quad \text{case} (iv).
\end{gathered}
\end{equation*}
Adding up the contributions of these terms, we have %\zf{should second line be $S_{30}S_{03}+9S_{12}S_{21}+9S_{21}S_{12}-19S_{33}$?} \ys{I think I had a typo in the case (iv) expression, the current 9 was previous written as 18}
\begin{align*}
E(j,k) &= m_6\cdot S_{33} + m_2m_4\cdot \Big[3(S_{20}S_{13} + S_{31}S_{02}) + 9 S_{11}S_{22} - 15 S_{33}\Big]\\
&\quad + m_3^2\cdot\Big[S_{30}S_{03} + 9S_{12}S_{21} - 10S_{33}\Big]\\
&\quad + m_2^3 \cdot \Big[9S_{20}S_{11}S_{02} - 9(S_{31}S_{02} + S_{20}S_{13}) + 6 S_{11}^3 - 27 S_{11}S_{22} + 30 S_{33}\Big].
\end{align*}
Define the quantities
\begin{align*}
T_1 &:= \sum_{j,k=1}^p S_{20}(j,k)S_{11}(j,k)S_{02}(j,k),\\
T_2 &:= \sum_{j,k=1}^p \Big(S_{20}(j,k)S_{13}(j,k) + S_{31}(j,k)S_{02}(j,k)\Big),\\
T_3 &= \sum_{j,k = 1}^p S_{11}(j,k)S_{22}(j,k),\quad T_4 = \sum_{j,k=1}^p S_{11}(j,k)^3,\\
T_5 &= \sum_{j,k=1}^p S_{12}(j,k)S_{21}(j,k),\quad T_6 = \sum_{j,k=1}^pS_{33}(j,k).
\end{align*}
Since $(\E[\sum_{j=1}^p(\x_j^\top \X\bbeta)^3])^2 = m_3^2\sum_{j,k=1}^p S_{03}(j,k)S_{3,0}(j,k)$, we have
\begin{align*}
\var\Big(\sum_{j=1}^p(\x_j^\top \X\bbeta)^3\Big) &= m_6\cdot T_6 + m_2m_4\cdot \Big(3T_2 + 9 T_3 - 15 T_6\Big) + m_3^2\cdot\Big(9T_5 - 10T_6\Big)\\
&\quad + m_2^3 \cdot \Big(9T_1 - 9T_2 + 6 T_4 - 27 T_3 + 30 T_6\Big).
\end{align*}

Under the Gaussian assumption on $\X$, direct calculation yields that
\begin{align*}
\E_{\X}[T_1], \E_{\X}[T_4] &\asymp \frac{n^3}{p^2} \vee \frac{n^6}{p^5}, \quad \E_{\X}[T_2], \E_{\X}[T_3] \asymp \frac{n^3}{p^3} \vee \frac{n^5}{p^4} \vee \frac{n^6}{p^5}, \quad \E_{\X}[T_5], \E_{\X}[T_6] \asymp \frac{n^3}{p^4} \vee \frac{n^6}{p^5},  
\end{align*}
while the Gaussian Poincar\'e inequality yields the variance bound
\begin{align*}
\Var_{\X}(T_1) \vee \Var_{\X}(T_4) &\leq C\Big(\frac{n^5}{p^4} \vee \frac{n^{10}}{p^9}\Big),\\
\Var_{\X}(T_2) \vee \Var_{\X}(T_3) &\leq C\Big(\frac{n^5}{p^6} \vee \frac{n^9}{p^8} \vee \frac{n^{10}}{p^9}\Big),\\
\Var_{\X}(T_5) \vee \Var_{\X}(T_6) &\leq C\Big(\frac{n^5}{p^7} \vee \frac{n^8}{p^9} \vee \frac{n^{10}}{p^{11}}\Big).
\end{align*}
Define the $\X$-event
\begin{align*}
\cE_{1,n}(\X) = \{|T_i| \leq 1.5\E[T_i], \quad i = 1,...,6\},
\end{align*}
which holds with probability converging to 1 using the above variance bounds. On this event, we have by $n^3/p^2 \rightarrow \infty$ that
\begin{align*}
\var\Big(\sum_{j=1}^p(\x_j^\top \X\bbeta)^3\Big)\bm{1}_{\cE_{1,n}} \leq C\Big(\frac{n^3}{p^2} + \frac{n^6}{p^5}\Big).
\end{align*}
On the other hand, let $\cE_{2,n} := \{A_3 \geq cn^3/p^2\}$ for some small enough $c > 0$. By Lemma \ref{lem:random_lsi} and the fact that $\P(\pnorm{\X}{\op} \geq (1+\sqrt{n/p})/2) \rightarrow 1 $, we have $\Prob(\cE_{2,n}) \rightarrow 1$. Hence on the event $\cE_{n} = \cE_{1,n}\cap \cE_{2,n}$, we have
\begin{align*}
\frac{\var\Big(\sum_{j=1}^p(\x_j^\top \X\bbeta)^3\Big)}{A_3^2}\bm{1}_{\cE_{n}} \rightarrow 0. 
\end{align*}
The proof is then completed by applying similar arguments to the terms in $(I)$ and $(II)$ of (\ref{eq:var_decompose}).
\end{proof}

\section{Implementation details of EBMoM and EBflow}\label{sec:implementation}

%\ys{Delete DMM GMM and combine with EBflow}

\subsection{EBMoM}

In the simulation Section \ref{sec:exp}, we use DMM, GMM, and non-parametric splines to turn the EBMoM moment estimates into a prior estimator. We refer to \cite{wu2020optimal} for the implementation details of DMM and GMM, and explain below those of non-parametric splines. 

The input is the moment vector $\hat\mu = (\hat\mu_1,\dots,\hat\mu_L)$ with
$\hat\mu_1 = 0$. The output is a discrete distribution
\begin{equation*}
  \tilde g^{\mathrm{spl}} = \sum_{i=1}^{N} \hat p_i\, \delta_{x_i},
\end{equation*}
supported on the grid $\{x_i\}_{i=1}^N$, obtained as follows by a roughness-penalized moment fit.

We first project the raw moments onto the moment space. Applying the
projection~\eqref{eq:DMM} to $\hat\mu$ on the interval $[-A, A]$ yields a
valid moment vector $\tilde\mu = (\tilde\mu_1,\dots,\tilde\mu_L)$. Next, fix a uniform grid on $[-A, A]$,
\begin{equation*}
  x_i = -A + (i-1)\,\delta, \quad i = 1,\dots,N,
  \qquad \delta = \frac{2A}{N-1},
\end{equation*}
with $N = 201$ points. We seek a probability vector $p \in \mathbb{R}^N$,
representing masses $p_i$ at the grid points $x_i$, whose moments match
$\tilde\mu$ while the induced density is smooth. To this end, define the
moment (Vandermonde) matrix $V \in \mathbb{R}^{L \times N}$ and the
second-difference operator $D_2 \in \mathbb{R}^{(N-2)\times N}$ by
\begin{equation*}
  V_{\ell, i} = x_i^{\ell}, \quad 1 \le \ell \le L,
  \qquad
  (D_2 v)_i = \frac{v_i - 2v_{i+1} + v_{i+2}}{\delta^2},
  \quad 1 \le i \le N-2,
\end{equation*}
so that $(Vp)_\ell = \sum_{i=1}^N x_i^\ell p_i$ is the $\ell$th moment of
$p$, and $D_2(p/\delta)$ is a discrete second derivative of the density
$f = p/\delta$. The weight vector is the solution of the convex program
\begin{equation}
  \hat p = \operatorname*{arg\,min}_{p \in \mathbb{R}^N}
  \; \bigl\| V p - \tilde\mu \bigr\|_2^2
  \;+\; \lambda\,\delta\,\bigl\| D_2 (p/\delta) \bigr\|_2^2
  \quad \text{s.t.} \quad
  p \ge 0, \;\; \mathbf{1}^\top p = 1,
  \label{eq:spline_program}
\end{equation}
with penalty parameter $\lambda = 10^{-3}$. The first term enforces moment
matching, and the second penalizes roughness: since
$\delta \sum_i (f''(x_i))^2 \approx \int (f'')^2$, it approximates a
penalty on the squared curvature of the fitted density.

\subsection{EBflow}\label{sec:EBflow}

EBflow is applied to minimize an unregularized negative log-likelihood objective
\[F_n(g)={-}\frac{1}{p}\log \int \exp\Big({-}\frac{\|\y-\X\bbeta\|^2}{2\sigma^2}\Big)\prod_{j=1}^p \d g(\beta_j).\]
Fixing a tuning parameter $\tau^2 \in (0,\sigma^2/\|\X\|_{\op}^2)$, this is equivalently
\[F_n(g)={-}\frac{1}{p}\log \int \exp\Big({-}\frac{1}{2}(\y-\X\bvarphi)^\top \bSigma^{-1}(\y-\X\bvarphi)\Big)\prod_{j=1}^p \d [\N(0,\tau^2)*g](\varphi_j), \quad \bSigma=\sigma^2\I-\tau^2\X\X^\top.\]

Nonparametric estimation is carried out via a Langevin diffusion in $\bvarphi$ and gradient flow in $g$,
\begin{align}
\d \bvarphi_t&=\left({-}\X^\top \bSigma^{-1}(\X\bvarphi_t-\y)+\left(\frac{[\N(0,\tau^2)*g_t]'(\varphi_{t,j})}{[\N(0,\tau^2)*g_t](\varphi_{t,j})}\right)_{j=1}^p \right)\d t+\sqrt{2}\,\d\b^t,\label{eq:langevin}\\
\frac{\d}{\d t} g_t(\beta)&=\alpha_t g_t(\beta) \left(\N(0,\tau^2)*\frac{\frac{1}{p}\sum_{j=1}^p \delta_{\varphi_{t,j}}}{\N(0,\tau^2)*g_t}(\beta)-1\right).\label{eq:priorflow}
\end{align}
These are implemented with a fixed discrete support for $g$ having $N=201$ uniformly spaced points on $[-A,A]$, using a time discretization with learning rate $\eta_t^\varphi$ for \eqref{eq:langevin} and $\eta_t^w$ for \eqref{eq:priorflow} (corresponding to the relative learning rate $\alpha_t=\eta_t^w/\eta_t^\varphi$), as described in \cite[Section 2.3]{fan2023gradient}. We choose $\tau^2=0.5\sigma^2/\|\X\|_{\op}^2$ as recommended in
\cite[Section 4.1]{fan2023gradient}. %\zf{TODO: update these section numbers after the arxiv paper is updated}. 
For the uninformative initialization, EBflow is initialized with $g_0$ being the discrete uniform prior over this $N$-point grid, run for 1000 burn-in iterations with $\eta_t^w=0$ (i.e.\ fixed prior) and $\eta_t^\varphi=1/\|\X^\top\bSigma^{-1}\X+\tau^{-2}\I\|_\op$ starting at $\bvarphi_0=0$, followed by 29000 iterations with $\eta_t^w \in [10^{-3},10^{-2}]$ and $\eta_t^\varphi \in [10^{-3},1]/\|\X^\top\bSigma^{-1}\X+\tau^{-2}\I\|_\op$ both following exponential rates of decay from the maximum to minimum value in this range. For the EBMoM initialization, EBflow is initialized with $g_0$ as the EBMoM estimate, run for 3000 burn-in iterations as above, followed by 27000 iterations with fixed learning rates $\eta_t^w=10^{-3}$ and $\eta_t^\varphi=10^{-3}/\|\X^\top\bSigma^{-1}\X+\tau^{-2}\I\|_\op$.

For parametric estimation of $g$, parametrizing $g_t(\beta)$ as $g(\beta \mid \theta_t)$, \eqref{eq:priorflow} is replaced by the gradient flow equation of prior parameters
\begin{equation}\label{eq:thetaflow}
\frac{\d}{\d t} \theta_t
=\alpha_t\nabla_\theta \left(\frac{1}{p}\sum_{j=1}^p \log [\N(0,\tau^2)*g(\cdot \mid \theta_t)](\varphi_{t,j})\right).
\end{equation}
For the Rademacher and sparse prior examples,
we take the parametric family $g(\cdot \mid \theta)$ to be a mixture of two discrete point masses (the same number as in the DMM estimates of EBMoM) with unknown locations and weights. We discretize in time \eqref{eq:langevin} 
with learning rate $\eta_t^\varphi$ and \eqref{eq:thetaflow} with learning rate $\eta_t^\theta$, and again choose $\tau^2=0.5\sigma^2/\|\X\|_{\op}^2$. The uninformative prior initialization is given by $g(\beta \mid \theta_0)=\frac{1}{2}\delta_{\beta=-0.1}+\frac{1}{2}\delta_{\beta=0.1}$, while the EBMoM initialization is given by the two-point DMM estimate. Burn-in periods are as described above for nonparametric estimation, followed in the case of the uninformative initialization by 29000 iterations with $\eta_t^\theta \in [10^{-2},10^{-1}]$ and $\eta_t^\varphi \in [10^{-3},1]/\|\X^\top\bSigma^{-1}\X+\tau^{-2}\I\|_\op$ following exponential decay, or in the case of the EBMoM initialization by 27000 iterations with fixed learning rates $\eta_t^\theta=10^{-2}$ and $\eta_t^\varphi=10^{-3}/\|\X^\top\bSigma^{-1}\X+\tau^{-2}\I\|_\op$.

\bibliographystyle{alpha}
\bibliography{Regression}

@article{bayati2011dynamics,
  title={The dynamics of message passing on dense graphs, with applications to compressed sensing},
  author={Bayati, Mohsen and Montanari, Andrea},
  journal={IEEE Transactions on Information Theory},
  volume={57},
  number={2},
  pages={764--785},
  year={2011},
  publisher={IEEE}
}

@article{lee2026parametric,
  title={Parametric Mean-Field empirical {B}ayes in high-dimensional linear regression},
  author={Lee, Seunghyun and Deb, Nabarun},
  journal={arXiv preprint arXiv:2601.16842},
  year={2026}
}

@article{morgante2023flexible,
  title={A flexible empirical {B}ayes approach to multivariate multiple regression, and its improved accuracy in predicting multi-tissue gene expression from genotypes},
  author={Morgante, Fabio and Carbonetto, Peter and Wang, Gao and Zou, Yuxin and Sarkar, Abhishek and Stephens, Matthew},
  journal={PLoS Genetics},
  volume={19},
  number={7},
  pages={e1010539},
  year={2023},
  publisher={Public Library of Science San Francisco, CA USA}
}

@article{spence2022flexible,
  title={A flexible modeling and inference framework for estimating variant effect sizes from {GWAS} summary statistics},
  author={Spence, Jeffrey P and Sinnott-Armstrong, Nasa and Assimes, Themistocles L and Pritchard, Jonathan K},
  journal={BioRxiv},
  pages={2022--04},
  year={2022},
  publisher={Cold Spring Harbor Laboratory}
}

@inproceedings{kuntz2023particle,
  title={Particle algorithms for maximum likelihood training of latent variable models},
  author={Kuntz, Juan and Lim, Jen Ning and Johansen, Adam M},
  booktitle={International Conference on Artificial Intelligence and Statistics},
  pages={5134--5180},
  year={2023},
  organization={PMLR}
}

@inproceedings{lim2024momentum,
  title={Momentum particle maximum likelihood},
  author={Lim, JN and Kuntz, J and Power, S and Johansen, Adam M},
  booktitle={Proceedings of 41st International Conference on Machine Learning (ICML)},
  volume={235},
  pages={29816--29871},
  year={2024}
}

@article{caprio2025error,
  title={Error bounds for particle gradient descent, and extensions of the log-{S}obolev and {T}alagrand inequalities},
  author={Caprio, Rocco and Kuntz, Juan and Power, Samuel and Johansen, Adam M},
  journal={Journal of Machine Learning Research},
  volume={26},
  number={103},
  pages={1--38},
  year={2025}
}

@article{akyildiz2023interacting,
  title={Interacting Particle {L}angevin Algorithm for Maximum Marginal Likelihood Estimation},
  author={Akyildiz, {\"O} Deniz and Crucinio, Francesca Romana and Girolami, Mark and Johnston, Tim and Sabanis, Sotirios},
  journal={arXiv preprint arXiv:2303.13429},
  year={2023}
}

@article{fan2025dynamicalI,
  title={Dynamical mean-field analysis of adaptive {L}angevin diffusions: {P}ropagation-of-chaos and convergence of the linear response},
  author={Fan, Zhou and Ko, Justin and Loureiro, Bruno and Lu, Yue M and Shen, Yandi},
  journal={arXiv preprint arXiv:2504.15556},
  year={2025}
}

@article{fan2025dynamicalII,
  title={Dynamical mean-field analysis of adaptive {L}angevin diffusions: {R}eplica-symmetric fixed point and empirical {B}ayes},
  author={Fan, Zhou and Ko, Justin and Loureiro, Bruno and Lu, Yue M and Shen, Yandi},
  journal={arXiv preprint arXiv:2504.15558},
  year={2025}
}

@incollection{neal1998view,
  title={A view of the {EM} algorithm that justifies incremental, sparse, and other variants},
  author={Neal, Radford M and Hinton, Geoffrey E},
  booktitle={Learning in graphical models},
  pages={355--368},
  year={1998},
  publisher={Springer}
}

@article{delyon1999convergence,
  title={Convergence of a stochastic approximation version of the {EM} algorithm},
  author={Delyon, Bernard and Lavielle, Marc and Moulines, Eric},
  journal={Annals of statistics},
  pages={94--128},
  year={1999},
  publisher={JSTOR}
}

@article{diebolt1993asymptotic,
  title={Asymptotic properties of a stochastic {EM} algorithm for estimating mixing proportions},
  author={Diebolt, Jean and Celeux, Gilles},
  journal={Stochastic Models},
  volume={9},
  number={4},
  pages={599--613},
  year={1993},
  publisher={Taylor \& Francis}
}

@article{neath2013convergence,
  title={On convergence properties of the {M}onte {C}arlo {EM} algorithm},
  author={Neath, Ronald C},
  journal={Advances in modern statistical theory and applications: a Festschrift in Honor of Morris L. Eaton},
  volume={10},
  pages={43--63},
  year={2013},
  publisher={Institute of Mathematical Statistics}
}

@article{fort2003convergence,
  title={Convergence of the {M}onte {C}arlo expectation maximization for curved exponential families},
  author={Fort, Gersende and Moulines, Eric},
  journal={The Annals of Statistics},
  volume={31},
  number={4},
  pages={1220--1259},
  year={2003},
  publisher={Institute of Mathematical Statistics}
}

@article{chan1995monte,
  title={{M}onte {C}arlo {EM} estimation for time series models involving counts},
  author={Chan, KS and Ledolter, Johannes},
  journal={Journal of the American Statistical Association},
  volume={90},
  number={429},
  pages={242--252},
  year={1995},
  publisher={Taylor \& Francis}
}

@article{zhou2021fast,
  title={A fast and robust {B}ayesian nonparametric method for prediction of complex traits using summary statistics},
  author={Zhou, Geyu and Zhao, Hongyu},
  journal={PLoS genetics},
  volume={17},
  number={7},
  pages={e1009697},
  year={2021},
  publisher={Public Library of Science}
}

@article{ge2019polygenic,
  title={Polygenic prediction via {B}ayesian regression and continuous shrinkage priors},
  author={Ge, Tian and Chen, Chia-Yen and Ni, Yang and Feng, Yen-Chen Anne and Smoller, Jordan W},
  journal={Nature communications},
  volume={10},
  number={1},
  pages={1776},
  year={2019},
  publisher={Nature Publishing Group UK London}
}

@article{lloyd2019improved,
  title={Improved polygenic prediction by {B}ayesian multiple regression on summary statistics},
  author={Lloyd-Jones, Luke R and Zeng, Jian and Sidorenko, Julia and Yengo, Lo{\"\i}c and Moser, Gerhard and Kemper, Kathryn E and Wang, Huanwei and Zheng, Zhili and Magi, Reedik and Esko, T{\~o}nu and others},
  journal={Nature communications},
  volume={10},
  number={1},
  pages={5086},
  year={2019},
  publisher={Nature Publishing Group UK London}
}

@article{zhou2013polygenic,
  title={Polygenic modeling with Bayesian sparse linear mixed models},
  author={Zhou, Xiang and Carbonetto, Peter and Stephens, Matthew},
  journal={PLoS genetics},
  volume={9},
  number={2},
  pages={e1003264},
  year={2013},
  publisher={Public Library of Science San Francisco, USA}
}

@article{polyanskiy2021sharp,
  title={Sharp regret bounds for empirical Bayes and compound decision problems},
  author={Polyanskiy, Yury and Wu, Yihong},
  journal={arXiv preprint arXiv:2109.03943},
  year={2021}
}

@article {jiang2009general,
    AUTHOR = {Jiang, Wenhua and Zhang, Cun-Hui},
     TITLE = {General maximum likelihood empirical {B}ayes estimation of
              normal means},
   JOURNAL = {Ann. Statist.},
  FJOURNAL = {The Annals of Statistics},
    VOLUME = {37},
      YEAR = {2009},
    NUMBER = {4},
     PAGES = {1647--1684},
      ISSN = {0090-5364,2168-8966},
   MRCLASS = {62C12 (62C25 62G05)},
  MRNUMBER = {2533467},
MRREVIEWER = {Tonglin\ Zhang},
       DOI = {10.1214/08-AOS638},
       URL = {https://doi.org/10.1214/08-AOS638},
}

@article {li2005convergence,
    AUTHOR = {Li, Jianjun and Gupta, Shanti S. and Liese, Friedrich},
     TITLE = {Convergence rates of empirical {B}ayes estimation in
              exponential family},
   JOURNAL = {J. Statist. Plann. Inference},
  FJOURNAL = {Journal of Statistical Planning and Inference},
    VOLUME = {131},
      YEAR = {2005},
    NUMBER = {1},
     PAGES = {101--115},
      ISSN = {0378-3758,1873-1171},
   MRCLASS = {62C12 (62G05)},
  MRNUMBER = {2137529},
MRREVIEWER = {Yu-ling\ Tseng},
       DOI = {10.1016/j.jspi.2003.12.017},
       URL = {https://doi.org/10.1016/j.jspi.2003.12.017},
}

@book {gautschi2004orthogonal,
    AUTHOR = {Gautschi, Walter},
     TITLE = {Orthogonal polynomials: computation and approximation},
    SERIES = {Numerical Mathematics and Scientific Computation},
      NOTE = {Oxford Science Publications},
 PUBLISHER = {Oxford University Press, New York},
      YEAR = {2004},
     PAGES = {x+301},
      ISBN = {0-19-850672-4},
   MRCLASS = {42-02 (33C45 33F05 42C05)},
  MRNUMBER = {2061539},
MRREVIEWER = {Leonid\ B.\ Golinski\u i},
}

@article {golub1969calculation,
    AUTHOR = {Golub, Gene H. and Welsch, John H.},
     TITLE = {Calculation of {G}auss quadrature rules},
   JOURNAL = {Math. Comp.},
  FJOURNAL = {Mathematics of Computation},
    VOLUME = {23},
      YEAR = {1969},
     PAGES = {221--230; addendum, ibid. 23 (1969), no. 106, loose microfiche
              suppl. A1--A10},
      ISSN = {0025-5718,1088-6842},
   MRCLASS = {65.55},
  MRNUMBER = {245201},
MRREVIEWER = {Frank\ Stenger},
       DOI = {10.2307/2004418},
       URL = {https://doi.org/10.2307/2004418},
}

@article{barbier2020mutual,
  title={Mutual information and optimality of approximate message-passing in random linear estimation},
  author={Barbier, Jean and Macris, Nicolas and Dia, Mohamad and Krzakala, Florent},
  journal={IEEE Transactions on Information Theory},
  volume={66},
  number={7},
  pages={4270--4303},
  year={2020},
  publisher={IEEE}
}

@book {bakry2014analysis,
    AUTHOR = {Bakry, Dominique and Gentil, Ivan and Ledoux, Michel},
     TITLE = {Analysis and geometry of {M}arkov diffusion operators},
    SERIES = {Grundlehren der mathematischen Wissenschaften [Fundamental
              Principles of Mathematical Sciences]},
    VOLUME = {348},
 PUBLISHER = {Springer, Cham},
      YEAR = {2014},
     PAGES = {xx+552},
      ISBN = {978-3-319-00226-2; 978-3-319-00227-9},
   MRCLASS = {60J25 (58J65 60J35 60J60)},
  MRNUMBER = {3155209},
MRREVIEWER = {Ming\ Liao},
       DOI = {10.1007/978-3-319-00227-9},
       URL = {https://doi.org/10.1007/978-3-319-00227-9},
}

@article{wu2020optimal,
  title={Optimal estimation of Gaussian mixtures via denoised method of moments},
  author={Wu, Yihong and Yang, Pengkun},
  journal={Annals of Statistics},
  volume={48},
  number={4},
  year={2020}
}

@article{rao1971estimation,
  title={Estimation of variance and covariance components--{MINQUE} theory},
  author={Rao, C Radhakrishna},
  journal={Journal of multivariate analysis},
  volume={1},
  number={3},
  pages={257--275},
  year={1971},
  publisher={Elsevier}
}

@article{zhou2017unified,
  title={A unified framework for variance component estimation with summary statistics in genome-wide association studies},
  author={Zhou, Xiang},
  journal={The Annals of Applied Statistics},
  volume={11},
  number={4},
  pages={2027},
  year={2017}
}

@book{PW-it,
  author = {Yury Polyanskiy and Yihong Wu},
  title = {Information Theory: From Coding to Learning},
  year = {2024},
  publisher={Cambridge University Press},
  note={{A}vailable at: \url{http://www.stat.yale.edu/~yw562/teaching/itbook-export.pdf}}
}

@article{WY-fnt,
 title={Polynomial methods in statistical inference: theory and practice},
 author={Yihong Wu and Pengkun Yang},
 journal={Monograph in \emph{Foundations and Trends in Communications and Information Theory}},
month={Oct},
 year={2020},
 Volume={17},number={4}, pages={402--586}
}

@book{vershynin2018high,
  title={High-dimensional probability: An introduction with applications in data science},
  author={Vershynin, Roman},
  volume={47},
  year={2018},
  publisher={Cambridge university press}
}

@article{hansen1982large,
  title={Large sample properties of generalized method of moments estimators},
  author={Hansen, Lars Peter},
  journal={Econometrica: Journal of the econometric society},
  pages={1029--1054},
  year={1982},
  publisher={JSTOR}
}

@article{casella1985introduction,
  title={An introduction to empirical Bayes data analysis},
  author={Casella, George},
  journal={The American Statistician},
  volume={39},
  number={2},
  pages={83--87},
  year={1985},
  publisher={Taylor \& Francis}
}

@article{zhang2003compound,
  title={Compound decision theory and empirical Bayes methods},
  author={Zhang, Cun-Hui},
  journal={Annals of Statistics},
  pages={379--390},
  year={2003},
  publisher={JSTOR}
}

@book{efron2012large,
  title={Large-scale inference: empirical Bayes methods for estimation, testing, and prediction},
  author={Efron, Bradley},
  volume={1},
  year={2012},
  publisher={Cambridge University Press}
}

@article{ver1996parametric,
  title={Parametric empirical Bayes methods for ecological applications},
  author={Ver Hoef, Jay M},
  journal={Ecological Applications},
  volume={6},
  number={4},
  pages={1047--1055},
  year={1996},
  publisher={Wiley Online Library}
}

@article{efron2001empirical,
  title={Empirical Bayes analysis of a microarray experiment},
  author={Efron, Bradley and Tibshirani, Robert and Storey, John D and Tusher, Virginia},
  journal={Journal of the American statistical association},
  volume={96},
  number={456},
  pages={1151--1160},
  year={2001},
  publisher={Taylor \& Francis}
}

@article{brown2008season,
  title={In-Season Prediction of Batting Averages: A Field Test of Empirical Bayes and Bayes Methodologies},
  author={Brown, Lawrence D},
  journal={The Annals of Applied Statistics},
  pages={113--152},
  year={2008},
  publisher={JSTOR}
}

@inproceedings{robbins1951asymptotically,
  title={Asymptotically Subminimax Solutions of Compound Statistical Decision Problems},
  author={Robbins, Herbert},
  booktitle={Proceedings of the Second Berkeley Symposium on Mathematical Statistics and Probability},
  volume={2},
  pages={131--149},
  year={1951},
  organization={University of California Press}
}

@inproceedings{robbins1956empirical,
  title={An Empirical Bayes Approach to Statistics},
  author={Robbins, Herbert},
  booktitle={Proceedings of the Third Berkeley Symposium on Mathematical Statistics and Probability, Volume 1: Contributions to the Theory of Statistics},
  volume={3},
  pages={157--164},
  year={1956},
  organization={University of California Press}
}

@article{fan2023gradient,
  title={Gradient flows for empirical Bayes in high-dimensional linear models},
  author={Fan, Zhou and Guan, Leying and Shen, Yandi and Wu, Yihong},
  journal={arXiv preprint arXiv:2312.12708},
  year={2023}
}

@article{mukherjee2023mean,
  title={A mean field approach to empirical Bayes estimation in high-dimensional linear regression},
  author={Mukherjee, Sumit and Sen, Bodhisattva and Sen, Subhabrata},
  journal={arXiv preprint arXiv:2309.16843},
  year={2023}
}

@article{henderson1953estimation,
  title={Estimation of variance and covariance components},
  author={Henderson, Charles R},
  journal={Biometrics},
  volume={9},
  number={2},
  pages={226--252},
  year={1953},
  publisher={JSTOR}
}

@article{rao1972estimation,
  title={Estimation of variance and covariance components in linear models},
  author={Rao, C Radhakrishna},
  journal={Journal of the American Statistical Association},
  volume={67},
  number={337},
  pages={112--115},
  year={1972},
  publisher={Taylor \& Francis}
}

@article{bulik2015ld,
  title={LD Score regression distinguishes confounding from polygenicity in genome-wide association studies},
  author={Bulik-Sullivan, Brendan K and Loh, Po-Ru and Finucane, Hilary K and Ripke, Stephan and Yang, Jian and Patterson, Nick and Daly, Mark J and Price, Alkes L and Neale, Benjamin M},
  journal={Nature genetics},
  volume={47},
  number={3},
  pages={291--295},
  year={2015},
  publisher={Nature Publishing Group}
}

@article{o2019extreme,
  title={Extreme polygenicity of complex traits is explained by negative selection},
  author={O'Connor, Luke J and Schoech, Armin P and Hormozdiari, Farhad and Gazal, Steven and Patterson, Nick and Price, Alkes L},
  journal={The American journal of human genetics},
  volume={105},
  number={3},
  pages={456--476},
  year={2019},
  publisher={Elsevier}
}

@article{o2021distribution,
  title={The distribution of common-variant effect sizes},
  author={O'Connor, Luke J},
  journal={Nature genetics},
  volume={53},
  number={8},
  pages={1243--1249},
  year={2021},
  publisher={Nature Publishing Group US New York}
}

@article{yang2011gcta,
  title={GCTA: a tool for genome-wide complex trait analysis},
  author={Yang, Jian and Lee, S Hong and Goddard, Michael E and Visscher, Peter M},
  journal={The American Journal of Human Genetics},
  volume={88},
  number={1},
  pages={76--82},
  year={2011},
  publisher={Elsevier}
}

@article{loh2015contrasting,
  title={Contrasting genetic architectures of schizophrenia and other complex diseases using fast variance-components analysis},
  author={Loh, Po-Ru and Bhatia, Gaurav and Gusev, Alexander and Finucane, Hilary K and Bulik-Sullivan, Brendan K and Pollack, Samuela J and Schizophrenia Working Group of the Psychiatric Genomics Consortium and de Candia, Teresa R and Lee, Sang Hong and Wray, Naomi R and others},
  journal={Nature genetics},
  volume={47},
  number={12},
  pages={1385--1392},
  year={2015},
  publisher={Nature Publishing Group US New York}
}

@article{pearson1894contributions,
  title={Contributions to the mathematical theory of evolution},
  author={Pearson, Karl},
  journal={Philosophical Transactions of the Royal Society of London. A},
  volume={185},
  pages={71--110},
  year={1894},
  publisher={JSTOR}
}

@book{hall2005generalized,
  title={Generalized Method of Moments},
  author={Hall, Alastair R},
  year={2005},
  publisher={Oxford University Press}
}

@article{nebebe1986bayes,
  title={Bayes and empirical Bayes shrinkage estimation of regression coefficients},
  author={Nebebe, Fassil and Stroud, TWF},
  journal={Canadian Journal of Statistics},
  volume={14},
  number={4},
  pages={267--280},
  year={1986},
  publisher={Wiley Online Library}
}

@article{george2000calibration,
  title={Calibration and empirical Bayes variable selection},
  author={George, EdwardI and Foster, Dean P},
  journal={Biometrika},
  volume={87},
  number={4},
  pages={731--747},
  year={2000},
  publisher={Oxford University Press}
}

@article{yuan2005efficient,
  title={Efficient empirical Bayes variable selection and estimation in linear models},
  author={Yuan, Ming and Lin, Yi},
  journal={Journal of the American Statistical Association},
  volume={100},
  number={472},
  pages={1215--1225},
  year={2005},
  publisher={Taylor \& Francis}
}

@article{Verzelen2018Adaptive,
  author    = {Verzelen, Nicolas and Gassiat, Elisabeth},
  title     = {Adaptive estimation of high-dimensional signal-to-noise ratios},
  journal   = {Bernoulli},
  volume    = {24},
  number    = {4B},
  pages     = {3683--3710},
  year      = {2018}
}

@article{Dicker2014Variance,
  author    = {Dicker, Lee H.},
  title     = {Variance estimation in high-dimensional linear models},
  journal   = {Biometrika},
  volume    = {101},
  number    = {2},
  pages     = {269--284},
  year      = {2014}
}

@article{swallow1978minimum,
  title={Minimum variance quadratic unbiased estimation (MIVQUE) of variance components},
  author={Swallow, William H and Searle, SR},
  journal={Technometrics},
  volume={20},
  number={3},
  pages={265--272},
  year={1978},
  publisher={Taylor \& Francis}
}

@article{rao1971minimum,
  title={Minimum variance quadratic unbiased estimation of variance components},
  author={Rao, C Radhakrishna},
  journal={Journal of Multivariate Analysis},
  volume={1},
  number={4},
  pages={445--456},
  year={1971},
  publisher={Elsevier}
}

@book{horn2012matrix,
  title={Matrix analysis},
  author={Horn, Roger A and Johnson, Charles R},
  year={2012},
  publisher={Cambridge university press},
edition={2nd}
}

@article{kong2018estimating,
  title={Estimating learnability in the sublinear data regime},
  author={Kong, Weihao and Valiant, Gregory},
  journal={Advances in Neural Information Processing Systems},
  volume={31},
  year={2018}
}

@article{lindsay1989moment,
  title={Moment matrices: applications in mixtures},
  author={Lindsay, Bruce G},
  journal={The Annals of Statistics},
  volume={17},
  number={2},
  pages={722--740},
  year={1989},
  publisher={Institute of Mathematical Statistics}
}
	
\end{document}